\documentclass[reqno]{amsart}
\usepackage[T1]{fontenc}
\usepackage{amssymb}
\usepackage{amsfonts}
\usepackage{amsmath,mathtools}
\usepackage{mathrsfs}
\usepackage{graphicx}
\usepackage{color}
\usepackage{esint}
\usepackage{verbatim}
\usepackage{bbm}
\usepackage{tikz}
\usepackage{enumitem}
\usepackage{setspace}
\usepackage{stackrel}

\usepackage{xcolor}
\definecolor{blue_links}{RGB}{13,0,180} 
\usepackage{hyperref}
\hypersetup{
	colorlinks=true, 
	linktoc=all,     
	linkcolor=blue_links,  
	citecolor=blue_links,
	urlcolor=blue_links,
}

\newcommand{\EEE}{\color{black}} 
 
\newcommand{\BBB}{\color{black}} 

\definecolor{vg}{rgb}{0.0, 0.4, 0.1}

\newcommand{\R}[1]{\mathbb{R}^{#1}}
\renewcommand{\S}[1]{\mathbb{S}^{#1}}

\newcommand{\cC}{\mathcal C}

\newcommand{\cH}{\mathcal H}

\newcommand{\cL}{\mathcal L}
\newcommand{\cM}{\mathcal M}

\newcommand{\bulk}{\mathrm{bulk}}

\newcommand{\de}{\mathrm d}

\newcommand{\surface}{\mathrm{surf}}

\newcommand{\eps}{\varepsilon}

\newcommand{\wto}{\rightharpoonup}
\newcommand{\lwto}{\longrightharpoonup}
\newcommand{\wsto}{\stackrel{*}{\rightharpoonup}}

\newcommand{\wSD}[1]{\stackrel[SD]{#1}\lwto}

\renewcommand{\geq}{\geqslant}
\renewcommand{\leq}{\leqslant}

\newcommand{\longrightharpoonup}{\relbar\joinrel\rightharpoonup}

\newcommand{\average}{{\mathchoice {\kern1ex\vcenter{\hrule
				height.4pt width 8pt depth0pt}
			\kern-11pt} {\kern1ex\vcenter{\hrule height.4pt width 4.3pt
				depth0pt} \kern-7pt} {} {} }}
\newcommand{\ave}{\average\int}

\newcommand{\res}{\mathop{\hbox{\vrule height 7pt width .5pt depth
			0pt\vrule height .5pt width 6pt depth0pt}}\nolimits}

\mathchardef\emptyset="001F

\makeatletter
 \@namedef{subjclassname@2020}{%
 \textup{2020} Mathematics Subject Classification}
 \makeatother

\providecommand{\U}[1]{\protect\rule{.1in}{.1in}}
\numberwithin{equation}{section}
\newtheorem{definition}{Definition}[section]
\newtheorem{theorem}[definition]{Theorem}
\newtheorem{lemma}[definition]{Lemma}
\newtheorem{proposition}[definition]{Proposition}

\newtheorem{corollary}[definition]{Corollary}

\newtheorem{remark}[definition]{Remark}

\begin{document}

\title[Structured deformations in   linearized elasticity]{Structured deformations in   linearized elasticity}


\author[M. Friedrich]{Manuel Friedrich} 
\address[Manuel Friedrich]{Department of Mathematics, Johannes Kepler Universit\"at Linz. Altenbergerstrasse 69, 4040 Linz,
    Austria}
\email{manuel.friedrich@jku.at}

\author[J. ~Matias]{Jos\'{e} Matias}
\address[J.~Matias]{Centro de An\'{a}lise Matem\'{a}tica, Geometria e Sistemas Din\^{a}micos, Departamento de Matem\'{a}tica, Instituto Superior T\'{e}cnico, Universidade de Lisboa, Av. Rovisco Pais 1, 1049-001 Lisbon, Portugal}
\email{jose.c.matias@tecnico.ulisboa.pt}

\author[E.~Zappale]{Elvira Zappale}
\address[E.~Zappale]{Dipartimento di Scienze di Base ed applicate per l'Ingegneria, Sapienza - Universit\`{a} di Roma, Via A. Scarpa, 16, 00161, Roma, Italy,
\& Research Centre in Mathematics and Applications,
University of \'Evora, Portugal}
\email{elvira.zappale@uniroma1.it}

\date{\today}

	


\maketitle

\begin{abstract}
We extend the theory of structured deformations to  the setting  of  linearized elasticity by providing an integral representation for the underlying energy that features  bulk and surface contributions.  Our derivation is obtained both via a direct approach by means of a global method for relaxation in  $BD$ and via an approximation from nonlinear elastic energies associated to  {nonsimple} materials.
\color{black}
\end{abstract}

\maketitle

\section{Introduction}

The purpose of this paper  is to establish a theory for the description of first-order \emph{structured deformations} in the context of linear elasticity.  Structured deformations were  originally introduced in a seminal work  by Del Piero and Owen \cite{DPO1993} in order to provide a mathematical framework that captures the effects at the macroscopic level of smooth deformations and of non\--\-smooth deformations (disarrangements) at  a sub-macroscopic level. In the classical theory of mechanics, the deformation of the body is characterized exclusively by the macroscopic deformation field~$g$ and its gradient~$\nabla g$ whereas  in the framework of structured deformations  an additional geometrical field~$G$ is introduced.  The underlying  idea for a pair of fields   $(g,G)$,  called a structured deformation,  is that  $G$ captures the contribution at the macroscopic scale of smooth sub-macroscopic changes, while the difference $\nabla g-G$ captures the contribution at the macroscopic scale of non-smooth sub-macroscopic changes, such as slips and separations (referred to as \emph{disarrangements} \cite{DO2002}). In this sense,  $G$ is called the deformation without disarrangements, and, heuristically, the disarrangement tensor $  \nabla g-G$ is an indication of how ``non-classical'' a structured deformation is. A key observation made by Del Piero and Owen is the so-called {\it Approximation theorem}, namely, that a structured deformation  $(g, G)$ can always be approximated suitably  by a sequence of  classical deformations.  This means that the theory is suitable  to address mechanical phenomena such as elasticity, plasticity, and the behavior of crystals with defects.

The variational formulation for first-order structured deformations in the $SBV$-setting \cite{AFP}  was first addressed by Choksi and Fonseca \cite{CF1997} where a  structured deformation is defined to be a pair 
$(g,G)\in SBV(\Omega;\R{d})\times L^p(\Omega;\R{d\times N})$,   for some $\Omega \subset \mathbb{R}^N$ and $p \geq 1.$ Therein, as a fundamental basis, 
 a suitable notion of convergence and a related {\it Approximation Theorem} have been introduced. Then,   departing from 
a functional of bulk-surface-type,  an integral representation for the ``most economical way'' to approach a given structured deformation was established.  More precisely, the considered functionals assign to any deformation of the body an energy featuring a bulk contribution measuring  the deformation (gradient) and  a surface  contribution accounting for the energy needed to fracture the body.

The theory of first-order structured deformations was broadened by Owen and Paroni \cite{OP2000} to second\--\-order structured deformations, which also account for other geometrical changes at a  sub-macroscopic level, such as curvature. The corresponding  variational formulation  was carried out by Barroso, Matias, Morandotti, and Owen \cite{BMMO2017}, and features  jumps for  both the approximating fields as well as for their gradients. We refer  the interested reader to \cite{MMO} and the references therein for a comprehensive survey about the theory of structured deformations, as well as applications. Let us just highlight the paper  \cite{KKMZ} which, in a purely homogeneous framework,  extends the  theory to the case where both the fields $g$ and $G$ are no longer $SBV$ and Lebesgue-type measures, respectively, but  may also account for Cantor-type behavior and very general concentration  phenomena.  \color{black} 

The goal of this paper is to develop an analogous theory in the context of linearized elasticity where,  under the assumption of infinitesimally small deformations, the material response depends only the symmetric part of the displacement gradient. The challenge consists in adapting the existing theory to the natural function space for problems in the linear setting, namely to the   space of functions of bounded deformation (denoted in the sequel by $BD$). We start by proving an {\it ad hoc Approximation Theorem} in $BD(\Omega) \times L^p(\Omega;\mathbb R^{N\times N}_{\rm sym})$, see Theorem \ref{appTHMh} (and  Theorem \ref{StSBDappthm} for a simplified version, restricted to the $SBD$ setting).   Our first main result consists in  an integral representation for a functional defined on structured deformations, see Theorem \ref{representation} below,    where we   identify the bulk and the surface density of the relaxed functional.  Specializing to the case where the original  energy functional does not depend on the material variable, we derive a  full relaxation result in  Theorem \ref{representationbis}, i.e., we  also identify  the Cantor part of the  relaxed energy density. 

  As usual in this context, the starting point of our analysis is a global method result. In our setting, it is tailored  for structured deformations in the  full $BD$-setting with the additional (symmetrical) geometrical field $G$ belonging to $L^p$, for any $p\geq1$, see  Theorem \ref{GMthmHSDsbd}.  Indeed,  the global method for relaxation, introduced by Bouchitt\'e, Fonseca, and Mascarenhas \cite{BFM} in the $BV$-setting, and later addressed in the $SBV^p$-setting by Bouchitt\'e, Fonseca, Leoni, and Mascarenhas \cite{BFLM}, provides a general method for the identification of the integral representation of a class of
functionals on  $BV$.   Since its inception, this  method for relaxation has known numerous applications and generalizations. Without being exhaustive, we mention recent developments in the context of variable exponent spaces  \cite{SSS}, spaces of bounded deformations  \cite{CFVG, BMZ, Conti-Focardi-Iurlano:15, crismale-friedrich-solombrino},  second-order structured deformations in the space $BH$  \cite{FHP}, and in the context of hierarchical systems of structured deformations  \cite{BMZ2024, BMMOZ2024}.  

 Let us emphasize that the global method  in the linear elastic setting  is not the mere  analog   of the global method for structured deformations  in the nonlinear setting obtained in \cite{Baía-Matias-Santos, KKMZ}.  Indeed,   in the current framework,    the minimum problems are defined in the \emph{entire} $BD(\Omega)\times L^p(\Omega;\R{N\times N}_{\rm sym})$, whereas in \cite{Baía-Matias-Santos, KKMZ} they are restricted to $SBV\times L^p$. Note, however, that our approach  also allows for a variant with  minimum problems defined on  $SBD$, see Corollary \ref{remStSBDrep}.  In this sense, our perspective is slightly more general compared to \cite{KKMZ} as we formulate potentially different global methods on $BD$ and $SBD$. Only in the relaxation result for  structured deformations (see Theorem \ref{representation}) it turns out that they coincide, due to the special form of the original energy defined on $SBD$.    Summarizing, we derive  an analogous result to the one in the nonlinear elastic framework \cite{KKMZ}, yet  obtained with a different technique.

Although the extension of the  existing theory of structured deformations to the functional framework of $BD$ is of general interest to us, 
our original motivation lies in the application to solid mechanics. Indeed, if deformation gradients are very close to the set of orientation preserving rigid motions, one expects  the theory of linear elasticity  to deliver a suitable description. In this context, it is   crucial to understand the relation between different models that  capture  the materials behavior at different strain regimes.    This has been  analyzed in a purely elastic setting by Dal Maso, Negri, and Percivale \cite{DalMasoNegriPercivale:02} who showed that, in the limit of vanishing displacements,  energies in nonlinear elasticity $\Gamma$-converge to the quadratic energy functional of linear elasticity. Their result has been extended in various directions in the last two decades. Without being exhaustive, let us mention incompressible materials \cite{Jesenko-Schmidt:20, edo}, traction forces \cite{mora},  the passage from atomistic-to-continuum models \cite{Braides-Solci-Vitali:07, Schmidt:2009}, multiwell energies \cite{alicandro.dalmaso.lazzaroni.palombaro, reggiani,  davoli, Schmidt:08},   fracture \cite{Friedrich, Friedrich:15-2},  plasticity \cite{Ulisse}, epitaxial growth of thin films \cite{Kostas}, or thermoviscoelasticity \cite{visco1, visco2}.
 
In the present work, we perform such a rigorous analysis in the context of structured deformations, see Theorem \ref{th: gamma-conv}. More precisely, starting from a nonlinear bulk-surface energy, we simultaneously perform a linearization of the elastic energies  and a relaxation   in terms of structured deformations on $BD$.   Here, in the nonlinear setting we resort to a model for  nonsimple materials  \cite{Toupin:62,Toupin:64}, i.e., we consider a model in the framework of free discontinuity and gradient discontinuity problems with second-order contributions in the bulk and the surface energy. Indeed, similar to the derivation of linearized models in fracture \cite{Friedrich}, the additional second-gradient contribution allows to derive rigidity and compactness results that are analogous to the ones in elasticity \cite{DalMasoNegriPercivale:02, FrieseckeJamesMueller:02}. Let us mention, however, that the derivation is more delicate compared to \cite{Friedrich} as in the present setting surface energies genuinely also depend on the jump height of the deformations. Among others, this requires a suitable variant of the Korn-Poincar\'e inequality in $BD$, see  Appendix \ref{sec: aux-app}.

The plan of this paper is as follows: after introducing the problem and setting the notation, in Section $2$ we state the main results whose proofs are in the subsequent  sections.   Section $3$ is dedicated to the abstract result via the global method, Section $4$ to the relaxation results, and  Section $5$ to the linearization result. Finally, in  Appendix \ref{BD}  we state and prove some results in  $BV$ and  $BD$ 
spaces, and Appendix \ref{sec: aux-app} is devoted to some auxiliary statements.  We close the introduction be introducing some relevant notation which will be used throughout the paper.

\begin{itemize}
	\item $\mathbb N$ denotes the set of natural numbers without the zero element;
	\item $\Omega \subset \mathbb R^{N}$ is a bounded domain with Lipschitz boundary; 
	\item $\mathbb S^{N-1}$ denotes the unit sphere in $\mathbb R^N$;
	\item $Q\coloneqq (-\tfrac12,\tfrac12)^N$ denotes the open unit cube of $\mathbb R^{N}$ centered at the origin. We further set $Q(r) = rQ^N$ for $r>0$.  For any $\nu\in\mathbb S^{N-1}$, $Q_\nu$ denotes any open unit cube in $\mathbb R^{N}$ centered at the origin  with two faces orthogonal to $\nu$; 
	\item for any $x\in\mathbb R^{N}$ and $r>0$, $Q(x,r)\coloneqq x+ Q(r)$  denotes the open cube in $\mathbb R^{N}$ centered at $x$ with side length  $r$.   Likewise, $Q_\nu(x, r  )\coloneqq x+r Q_\nu$.
    \item   For any $r>0$ and $x \in \mathbb R^N$, by $B(x,r)$ we denote the  open  ball of center $x$ and radius $r$, i.e., $\{y \in \mathbb R^N: |x-y| <  r\}$, with $|\cdot|$ denoting the Euclidean distance in $\mathbb R^N$. When $x=0$ and $r=1$, the ball will be simply denoted by $B_1$. 
	\item ${\mathcal O}(\Omega)$ is the family of all open subsets of $\Omega $, whereas ${\mathcal O}_\infty(\Omega)$ is the family of all open subsets of $\Omega $ with Lipschitz boundary.  We write $U \Subset V$ for $U,V \in \mathcal{O}$ if $U$ is compactly contained in $V$; 
	\item $\mathcal L^{N}$ and $\mathcal H^{N-1}$ denote the  $N$-dimensional Lebesgue measure and the $\left(  N-1\right)$-dimensional Hausdorff measure in $\mathbb R^N$, respectively. The symbol $\de x$ will also be used to denote integration with respect to $\mathcal L^{N}$; 
    \item For every $y \in \mathbb R^N$ and any measurable function $u$, we define  $\tau_y u(x):= u(x-y)$.  For any set $ T\subset \mathbb{R}^N$ we define  $\tau_y T := \lbrace y + x \colon x\in T\rbrace$;
	\color{black}
    \item The set of symmetric and skew-symmetric matrices in $\mathbb R^{N\times N}$ will be denoted by $\mathbb R^{N\times N}_{\rm sym}$ and $\mathbb R^{N\times N}_{\rm skew}$, respectively;
    \item  For every element $A\in \mathbb R^{N\times N}_{\rm sym}$, and every element $\xi, \eta \in \mathbb R^N$, the quadratic form associated to $A$ acting on $\xi $ and $\eta$ is given by
    $(A \xi, \eta) = A \xi \cdot \eta$;  
	\item $\mathcal M(\Omega;\mathbb R^{N\times N})$ is the set of finite matrix-valued Radon measures on $\Omega$; 
	given $\mu\in\mathcal M(\Omega;\mathbb R^{N\times N})$,  
	the measure $|\mu|$ 
	denotes the total variation of $\mu$;
	\item $L^p(\Omega; \mathbb R^N)$ is the set of  vector-valued  $p$-integrable functions with values in $\mathbb R^N$; 
\item  We use standard notation of $SBV$ functions, see \cite[Section 4]{AFP}. In particular, we let
\begin{align*}
SBV^2(\Omega;\mathbb{R}^N) = \lbrace y \in SBV(\Omega;\mathbb{R}^N): \ \nabla y \in L^2(\Omega;\mathbb{R}^{N\times  N}), \ \mathcal{H}^{N-1}(J_y) < + \infty \rbrace,
\end{align*}
where $\nabla y(x)$ denotes the approximate differential at $\mathcal{L}^N$-a.e.\ $x \in \Omega$  and $J_y$ the jump set;
 \EEE
	\item   We use standard notations for $BD$ and $SBD$ functions. In particular, the  distributional symmetric gradient is denoted by  $ Eu$ and  admits the decomposition $ Eu=\mathcal Eu  +  E^s u$. If $Eu$ has vanishing Cantor part, we say $u \in SBD(\Omega)$ and it holds $ Eu = \mathcal E u\cL^N+[u]\odot\nu_u\mathcal H^{N-1}\res J_u$, where $J_u$ is the jump set of~$u$, $[u]$ denotes the jump of~$u$ on $J_u$, $\nu_u$ is the unit normal vector to $J_u$, and   $\odot$ the symmetric vector product. Basic properties are collected in Appendix \ref{BD};  
    \item  We denote by $\mathcal R$ the kernel of the linear operator $E$ consisting of the class of infinitesimal rigid motions in $\mathbb R^N$, i.e., affine maps of the form $Mx + b$, where $M \in \mathbb R^{N \times N}_{\rm skew}$  and $b\in \mathbb R^N$. Recalling that $\mathcal R$ is closed and finite dimensional it is possible to define the orthogonal projection $P \colon BD(\Omega) \to \mathcal R$.     
		\item  $C$ represents a generic positive constant that may change from line to line.
\end{itemize}


\medskip

\section{Setting and main results}

In this section we give an overview of our main results. In the sequel, $\Omega \subset \mathbb R^{N}$ denotes a bounded domain  with Lipschitz boundary.  \EEE

\subsection{Definition of  structured deformations in $BD$ and approximation theorem}\label{subsec. approx}

 Given $p \geq 1$, we define the class of \emph{structured deformations in $BD$ and $SBD$}  as 
\begin{equation*} StBD^p(\Omega) := \big\{ (g, G): g \in BD(\Omega), \ G \in L^p(\Omega; \mathbb{R}_{\text{sym}}^{N\times N})\big\},
\end{equation*}
and
\begin{equation*} StSBD^p(\Omega) := \big\{ (g, G): g \in SBD(\Omega),\  G \in L^p(\Omega; \mathbb{R}_{\text{sym}}^{N\times N})\big\},
\end{equation*}
respectively. We have the following approximation result.

\begin{theorem}[Approximation]\label{appTHMh}
Let   $(g,G) \in  StBD^p(\Omega)$. Then, there exists $\lbrace u_n\rbrace \subset SBD(\Omega)$ such that $u_n \to g$   strongly in $L^1(\Omega;\mathbb R^{N})$, $\mathcal {E}u_n = G$, and 
\begin{align}\label{EuntoG}
E u_n \overset{\ast}{\rightharpoonup} E g \hbox{ in }\mathcal M(\Omega;\mathbb R^{N\times N}_{\rm sym}).
\end{align}
Moreover, there exists $C(N)>0$ such that    
\begin{align}\label{estub}|E u_n|(\Omega) \leq C(N)\big(|E g|(\Omega)+ \|G\|_{L^1(\Omega)}\big) \quad \text{ for all $n \in \mathbb{N}$}.
\end{align}
\color{black}
\end{theorem}
\begin{proof}
The first part of the  proof  would be  an immediate consequence of \cite[Theorem 1.1]{Silhavy}, see  Theorem \ref{silhavy_app} in the Appendix. In particular, $\lbrace u_n\rbrace \subset SBV(\Omega;\mathbb R^N)$ and $\nabla u_n= \mathcal{E} u_n$  with $\mathcal E u_n = G$.  Yet, to ensure also the  second part of the statement,
we  prefer to follow another path 
similar to the one in \cite[Theorem~2.3]{KKMZ} for the $BV$ case,  by applying Alberti's Theorem.

By Proposition \ref{prop:densityuptoboundary}, there exists $\{g_n \}\subset C^\infty(\overline{\Omega};\mathbb R^N)$ such that
$g_n \overset{L^1}{\to} g$,  $E g_n=\mathcal{E}g_n\wsto  Eg$ in  the sense of measures, and in particular
\begin{equation}\label{818bis}
|Eg|(\Omega) = \lim_{n\to \infty}| \mathcal E g_n|(\Omega). 
\end{equation}
Observe that for any $g_n \in SBD(\Omega)$ (in particular $g_n \in C^\infty(\overline \Omega;\mathbb R^N$)) and $G \in L^1(\Omega;\mathbb R^{N\times N}_{\rm sym})$, by Alberti's Theorem \ref{Al} there exists $h_n \in SBV(\Omega;\mathbb R^N)$ such that 
\begin{align}\label{818tris}
 \mathcal E h_n=\nabla h_n= G-\mathcal E g_n,
 \end{align} and such that 
 \begin{align}\label{818quin}
 |D h_n| (\Omega) \leq C'(N) \|G- \mathcal E g_n\|_{L^1(\Omega)}.
 \end{align}
 By Lemma \ref{ctap} there exists a  piecewise constant sequence $\{h_{n,k}\}_k \subset SBV(\Omega;\mathbb R^N)$  such that $h_{n,k} \to h_n  $ in $L^1(\Omega;\mathbb R^N)$ and 
 \begin{align}\label{818quater}|D h_{n,k}|(\Omega) \to |D h_n|(\Omega) \hbox{ as }k \to +\infty.
 \end{align}
 Observe that, up to extending $h_n$ outside of $\Omega$, with $|Dh_n|(\partial \Omega) = 0 $, it also follows that $|E h_{n,k}|(\Omega) \to |E h_n|(\Omega)$ as $k \to +\infty$. Indeed,  $|E h_{n,k}| (\Omega) $ is bounded in $k$, hence, up to a subsequence, $E u_{n,k}$ is weakly* converging to $E h_n$, and   $|D h_n| (\partial \Omega) = 0$, 
see \eqref{Ehn*} in the appendix for the precise argument. 
  Define
 \begin{align*}u_{n,k}:=g_n + h_n - h_{n,k}.
 \end{align*}
 Then, 
 $$
 \lim_{n\to \infty} \lim_{k\to \infty} u_{n,k}= \lim_{ n \to \infty} g_n = g  \ \ \ \hbox{ strongly in } L^1(\Omega;\mathbb R^{N}),$$
 $$
 \mathcal E u_{n,k}= \mathcal E g_n + \mathcal E h_n= 
 G \quad \text{for every $n, k \in \mathbb N$,} $$
 as well as 
 \begin{align*}
 \lim_{n\to \infty} \lim_{k \to \infty} E u_{n,k} &= \lim_{n\to \infty} \lim_{k\to \infty} (E g_n + E h_n - E h_{n,k})= \lim_{n\to \infty} \lim_{k\to \infty} (\mathcal E g_n + E h_n - E h_{n,k})  = 
 \lim_{n\to \infty} \mathcal E g_n = E g 
 \end{align*}
 weakly* in the sense of measures,  where in the penultimate equality the strict convergence of $E h_{n,k}$ to $E h_n$ as $k \to +\infty$
 has been used and in the last equality we have applied \eqref{818bis}.

   
To conclude the proof, we choose a suitable diagonal sequence and use the metrizability of weak* convergence on bounded sets: the triangle inequality,
 \eqref{818bis}, \eqref{818tris}, \eqref{818quin}, and \eqref{818quater} lead to the right-hand side of \eqref{estub} for $|E u_{n,k}|$ in place of $|E u_{n}|$,  up to considering  suitable tails  in $n$ and $k$ of the sequence $\{u_{n, k}\}$.
 Then, defining a diagonal sequence  $u_n:= u_{n, k(n)}$ 
we get \eqref{estub} and \eqref{EuntoG}.
\end{proof}

\begin{remark}\label{remStBDconv}
We observe that the bound in \eqref{estub} is obtained as in \cite[Proposition 2.1]{MMOZ}.

If in the above theorem $g \in SBD(\Omega)$, a more direct argument could be implemented, namely, we can take $u_n= g+h-h_n$,   where $h \in SBV(\Omega;\mathbb R^N)$ such that $\nabla h= \mathcal E h= G- \mathcal E g$ and $\lbrace h_n\rbrace$ is a piecewise constant approximation of $h$, leading to $\mathcal E u_n= G$.  In this case, the use of the triangle inequality gives  $|E u_n|(\Omega)\leq C(N)(|E g|(\Omega)+ \|G\|_{L^1(\Omega)})$. A detailed proof can be found in the Appendix, see Theorem \ref{StSBDappthm}. 
\end{remark}


\begin{definition}\label{StBDconv}
A sequence $\lbrace g_n\rbrace\subset SBD(\Omega)$ is said to converge to $(g,G)\in StBD^p(\Omega)$ if and only if $g_n \to g$ in $L^1(\Omega;\mathbb R^N)$ and 
$$\mathcal E g_n  \overset{\ast}{\rightharpoonup}  G \text{ in } \mathcal M(\Omega;\mathbb R^{N\times N}_{\rm sym})  \   \text{ if   $p=1$ \quad \quad  or } \ \ \quad \quad    \mathcal E u_n \wto G \text{ in } L^p(\Omega;\mathbb R_{\rm sym}^{N\times N}) \   \text{ if   $p>1$}.   $$
This convergence will be denoted as 
\begin{equation*}
g_n  \wSD{*}  (g, G).
\end{equation*}
\end{definition}

\subsection{Integral representation}\label{irep}

We now present an integral representation result for functionals defined on structured deformations in $BD$.  

 Let $p \geq 1$ and let $\mathcal F\colon StBD^p(\Omega) \times   \mathcal 
B(\Omega) \to [0, +\infty]$  be a functional satisfying the following hypotheses:
\begin{enumerate}
	\item[(H1)] for every $(g, G) \in StBD^p(\Omega) $, $\mathcal F(g, G;\cdot)$ is a Radon measure; 
	\item[(H2)] for every $O \in \mathcal O(\Omega)$, $\mathcal F(\cdot, \cdot; O)$ is $StBD^p$-lower semicontinuous, i.e., 
	if $(g, G) \in StBD^p(\Omega) $ and $(g_n, G_n)$ $\in StBD^p(\Omega) $ are such that $g_n \to g$ in $L^1(\Omega;\mathbb R^{N} )$, $G_n\rightharpoonup G$ in $L^p(\Omega;\mathbb{R}_{\text{sym}}^{N\times N})$ if $p>1$ (and $G_n\overset{*}{\rightharpoonup} G$  in $\mathcal M(\Omega;\mathbb{R}_{\text{sym}}^{N\times N})$ if $p=1$), as $n \to +\infty$, then
	$$
	\mathcal F(g, G;O)\leq \liminf_{n\to +\infty}\mathcal F(g_n,G_n;O);
	$$
    \color{black}
	\item[(H3)] for all $O \in \mathcal O(\Omega)$, $\mathcal F(\cdot, \cdot;O)$ is local, that is, if $g= u$, $G= U$ a.e.\ in $O$, then 
	$\mathcal F(g, G;O)= \mathcal F(u, U;O)$;
	\item[(H4)] there exists a constant $C>0$ such that
	\begin{align*}
	 \frac{1}{C}\left( \|G\|^p_{L^p(O)}  + | E g|(O)\right)&\leq \mathcal F(g, G;O) \leq
	  C\left(\mathcal L^N(O) + \|G\|_{L^p(O)}^p  +  |E g|(O)\right)
	\end{align*}
 for every $(g, G) \in StBD^p(\Omega)$ and every $O \in  \mathcal B(\Omega)$.

\end{enumerate}

\medskip

To formulate the main result, we need some more notation.
Given $(g, G)\in StBD^p(\Omega)$ and 
$O \in \mathcal O_{\infty}(\Omega)$, we introduce the space of \emph{competitors}
\begin{align*}
	\mathcal C_{StBD^p}(g, G; O):=\Big\{(u, U)\in StBD^p(\Omega) \colon \,  u=g \hbox{ in a neighborhood of } \partial O,  
	 \int_O (G-U) \,\de x=0  \Big\}, 	 
\end{align*}
and we let $m\colon StBD^p(\Omega)\times \mathcal O_\infty(\Omega)  \to [0,+\infty] $ be the functional defined by
\begin{align}\label{mdef}
m(g, G;O):=\inf\big\{\mathcal F(u, U; O): (u, U)\in \mathcal C_{StBD^p}(g, G; O)\big\}.
\end{align}
 For $x_0\in \Omega$, $ a \in \mathbb R^N$, and  $\xi \in \mathbb{R}^{N\times N}$ we define $\ell_{x_0,a,\xi}(x) = a+ \xi(\cdot-x_0)$, and for $x_0\in \Omega$, $ \theta,\lambda \in \mathbb R^N$,  $\nu \in \mathbb S^{N-1}$ we let 
\begin{align}\label{stepfun}
v_{x_0,\lambda,\theta, \nu}(x) := \begin{cases} \lambda, &\hbox{if } (x-x_0)\cdot \nu > 0\\
\theta, &\hbox{if } (x-x_0)\cdot \nu \leq 0.\end{cases}
\end{align}
Finally, by $0$ we denote the zero matrix in $\mathbb R^{N \times N}$.

Following the ideas of the global method of relaxation introduced in \cite{BFM}, our aim  is to prove the theorem below.   
 
\begin{theorem}\label{GMthmHSDsbd}
	Let $p \geq  1$  and let $\mathcal F\colon StBD^p(\Omega) \to [0, +\infty]$ be a functional satisfying {\rm (H1)}--{\rm (H4)} on $StBD^p(\Omega)$. 
Then, for every $(g,G) \in StSBD^p(\Omega)$ and $O \in \mathcal O(\Omega)$,  it  holds that 
	\begin{align*}
	\mathcal F(g, G;O)= \int_O \!f\big(x,g(x),   \nabla  g(x), G(x)\big)\,\de x +\int_{O \cap J_g}\!\!\!\!\!\Phi\big(x, g^+(x), g^-(x),\nu_g(x)\big) 
	\,\de \mathcal H^{N-1}(x),\end{align*}
where
\begin{align}\label{f}
f(x_0, a, \xi, B):=	\limsup_{\varepsilon \to 0^+}\frac{m(\ell_{x_0,a,\xi}, B; Q(x_0,\varepsilon))}{\varepsilon^N},
	\end{align}
	 for all $x_0\in \Omega$, $ a \in \mathbb R^N$,   $\xi \in \mathbb R^{N\times N}$, and  $ B \in \mathbb{R}_{{\rm sym}}^{N\times N}$, and
\begin{align}\label{Phi}
	\Phi(x_0, \lambda, \theta, \nu):= \limsup_{\varepsilon \to 0^+}\frac{m(v_{x_0,\lambda, \theta,\nu}, 0; Q_{\nu}(x_0,\varepsilon))}{\varepsilon ^{N-1}},
	\end{align}
for all $x_0\in \Omega$, $ \lambda, \theta \in \mathbb R^N$, and $\nu \in \mathbb S^{N-1}$.
\end{theorem}

\color{black}
  The result above will follow as a corollary of the subsequent result which identifies bulk and jump densities for  all functions in $BD(\Omega)$. 
  
\begin{proposition}\label{GMthmHSD}
	Let $p \geq  1$  and let $\mathcal F \colon StBD^p(\Omega) \to [0, +\infty]$ be a functional satisfying {\rm (H1)}--{\rm (H4)}. Then, for every $(g,G) \in StBD^p(\Omega)$ it holds that 
\begin{align}\label{fproof}
	\frac{ {\rm d} \mathcal F(g, G;\cdot)}{{\rm d} \mathcal L^N} (x_0)=	
	f\big(x_0, g(x_0),   \nabla   g(x_0), G(x_0)\big)
	\end{align}
    for $\mathcal L^N$-a.e.\ $x_0 \in \Omega$,  and 
\begin{align}\label{Phiproof}
\frac{{\rm d} \mathcal F(g, G; \cdot)}{{\rm d} \mathcal H^{N-1}\lfloor{J_g}}(x_0) = \Phi\big(x_0, g^+(x_0), g^-(x_0), \nu_g(x_0)\big)
\end{align}
 for $\mathcal H^{N-1}$-a.e.\ $x_0 \in J_u$,  where $f$ and $\Phi$ are given by \eqref{f} and \eqref{Phi}, respectively.

\end{proposition}

\begin{remark}[Invariance]
\label{traslinv}
(i) It follows immediately from the definitions given in \eqref{f}--\eqref{Phi} 
that, if $\mathcal F$ is translation invariant in the first variable, i.e., if
$$\mathcal F(g+a, G;O)= \mathcal F(g, G;O) \quad \text{for every  $((g,G),O) \in StBD^p(\Omega)\times \mathcal O(\Omega)$ \color{black} and for every $a\in \mathbb R^N$,}$$ 
then the function $f$ in \eqref{f} does not depend on $a$ and the function $\Phi$ in \eqref{Phi} does not depend on $\lambda$ and $\theta$ but only on the difference $\lambda- \theta$. Indeed, in this case we conclude that
$$f(x_0, a, \xi, B) = f(x_0, 0, \xi, B) \quad \mbox{and}
\quad
\Phi(x_0, \lambda, \theta, \nu) = \Phi(x_0, \lambda - \theta,0, \nu)$$
for all $x_0\in \Omega$, $ a, \theta,\lambda \in \mathbb R^N$,   $\xi\in \mathbb R^{N \times N}$, $B \in \mathbb R^{N\times N}_{\rm sym}$,   and $\nu \in \mathbb S^{N-1}$. With an abuse of notation, we write
$$f(x_0,\xi, B) = f(x_0, 0, \xi, B)  \quad \mbox{and}
\quad \Phi(x_0, \lambda - \theta, \nu) = \Phi(x_0, \lambda - \theta,0, \nu).$$ 
(ii) Arguing as in \cite[Lemma 4.3.3]{BFM}, one can further  prove that   both $f$ in \eqref{f} and $\Phi$ in \eqref{Phi} do not depend on the spatial variable $x$, provided  $\mathcal F(\tau_y g,\tau_y G;\tau_y O)= \mathcal F(g,G;O)$ for every $(g,G)\in StBD^p(\Omega)$ such that $\tau_y O$ is an open subset of $\Omega$.

 (iii) If $\mathcal F (g+ M x, G;O)= \mathcal F (g, G; O)$ for every  $((g,G),O) \in StBD^p(\Omega)\times \mathcal O(\Omega)$ and   for every $M \in \mathbb R^{N \times N}_{\rm skew}$,  then $f$ depends only on elements in $ \mathbb R^{N \times N}_{\rm sym}$  in the third variable, namely
$$f(x_0, a, \xi, B) = f\left(x_0, a, \frac{\xi+\xi^T}{2}, B\right).$$
\end{remark}



    The same proof as the one  of Proposition  \ref{GMthmHSD} allows to obtain a representation  for functionals defined only on $StSBD^p(\Omega)$. 

\begin{corollary}\label{remStSBDrep}
Let $p \geq  1$  and let $\mathcal F\colon StSBD^p(\Omega) \to [0, +\infty]$ be a functional satisfying {\rm (H1)}--{\rm (H4)} on $StSBD^p(\Omega)$. Then,   for every $(u,U) \in StSBD^p(\Omega)$  and $O \in \mathcal O(\Omega)$,  it holds that
	\begin{align*}
	\mathcal F(g, G;O)= \int_O \!\tilde f\big(x,g(x),   \nabla   g(x), G(x)\big)\,\de x +\int_{O \cap J_g}\!\!\!\!\!\tilde \Phi\big(x, g^+(x), g^-(x),\nu_g(x)\big) 
	\,\de \mathcal H^{N-1}(x),\end{align*}
where
\begin{align*}
\tilde f(x_0, a, \xi, B):= \limsup_{\varepsilon \to 0^+}\frac{\tilde{m}(\ell_{x_0,a,\xi}, B; Q(x_0,\varepsilon))}{\varepsilon^N},
	\end{align*}
\begin{align*} 
	\tilde\Phi(x_0, \lambda, \theta, \nu):= \limsup_{\varepsilon \to 0^+}\frac{\tilde{m}(v_{x_0,\lambda, \theta,\nu}, 0; Q_{\nu}(x_0,\varepsilon))}{\varepsilon ^{N-1}},
	\end{align*}
for all $x_0\in \Omega$, $ a, \lambda, \theta \in \mathbb R^N$,   $\xi \in \mathbb R^{N\times N}$, $ B \in \mathbb{R}_{{\rm sym}}^{N\times N}$,   $\nu \in \mathbb S^{N-1}$, 
where  we define     $ \tilde m\colon StSBD^p(\Omega)\times \mathcal O_\infty(\Omega)$ as 
\begin{align*} 
\tilde m(g, G;O):=\inf\Big\{\mathcal F(u, U; O): (u, U)\in \mathcal C_{StSBD^p}(g, G; O)\Big\},
\end{align*}
 with 
\begin{align*}
\mathcal C_{StSBD^p}(g, G; O):=\Big\{(u, U)\in StSBD^p(\Omega) \colon \,  u=g \hbox{ in a neighborhood of } \partial O,  
	 \int_O (G-U) \,\de x=0  \Big\}, 	 
\end{align*}
for  $(g, G)\in StSBD^p(\Omega)$ and 
$O \in \mathcal O_{\infty}(\Omega)$. 
\end{corollary}

 The proof of these results will be given in Section \ref{gm}.

\medskip

\subsection{Relaxation}\label{appl}

As an application of the  integral representation obtained in  Theorem \ref{GMthmHSDsbd}    and Corollary~\ref{remStSBDrep}, we now present a relaxation result for structured deformations in the spirit of \cite{CF1997}. In particular, we essentially adopt the  assumptions from  
 \cite{CF1997}, suitably adapted to the setting of linear elasticity. 
 
 Given two nonnegative functions $W\colon\Omega\times\mathbb R^{N\times N}_{\rm sym}\to[0,+\infty)$ and $\psi\colon\Omega\times\R{N}\times\S{N-1}\to[0,+\infty)$, we consider the initial bulk-surface energy   defined by
\begin{equation}\label{103}
	F(u)\coloneqq \int_\Omega W\big(x,\mathcal E u(x)\big)\,\de x+\int_{\Omega\cap J_u} \psi\big(x,[u](x), \nu_u(x)\big)\,\de\cH^{N-1}(x)
\end{equation}
 for  displacements   $u\in SBD(\Omega)$. Then, as justified by the Approximation Theorem \ref{appTHMh}, we assign an energy to a structured deformation $(g,G)\in  StBD^p(\Omega)$ 
via
\begin{equation}\label{102}
	I_p(g,G)\coloneqq \inf\Big\{\liminf_{n\to\infty} F(u_n)\colon u_n \in SBD(\Omega), \ u_n\wSD{*}(g,G)\Big\},
\end{equation}
 where  $\wSD{*}$ denotes the convergence defined in Definition \ref{StBDconv}. 

Our goal is to  obtain an integral representation result for $I_p$, by means of  Theorem \ref{GMthmHSDsbd},  under a similar set of hypotheses on $W$ and $\psi$ as those considered in \cite{CF1997} and \cite{MMOZ} for the $SBV$ setting,
but requiring only measurability rather than uniform continuity of $W$ in the $x$ variable.  Precisely, we assume that $W\colon\Omega \times \mathbb R^{N\times N}_{\rm sym}\to[0,+\infty)$ is a Carath\'eodory function and $\psi\colon \Omega \times \mathbb R^N \times \mathbb S^{N-1} \to [0,+\infty) $ is continuous and such that the following conditions hold: 

\setlist[enumerate]{label=(W\arabic*)}
\begin{enumerate}
\item \label{(W1)_p}  ($p$-Lipschitz continuity) there exists $C_W >0$ such that, for a.e.\ $x\in\Omega$ and $A_1,A_2 \in \mathbb R^{N\times N}_{\rm sym}$,
	\begin{equation*}
		|W(x,A_1) - W(x,A_2)| \leq C_W |A_1 - A_2| \big(1+|A_1|^{p-1}+|A_2|^{p-1}\big);
	\end{equation*}
\item \label{W3} (control from above) there exists $A_0 \in \mathbb R^{N\times N}_{\rm sym}$ such that 
$W(\cdot, A_0)\in L^\infty(\Omega)$;  
\item \label{W4} ($p$-growth from below)  there exists $c_W>0$ such that, for a.e.\ $x \in \Omega$ and every $A \in  \mathbb R^{N\times N}_{\rm sym}$,  
	\begin{equation*}
		c_W |A|^{p}-\frac{1}{c_W}\leq W(x,A);
	\end{equation*}
	\end{enumerate}
\setlist[enumerate]{label=($\psi$\arabic*)}
\begin{enumerate} 
\item\label{psi_0} (symmetry) for every $x \in \Omega$, $\lambda \in \R{N}$, and 
$\nu \in \mathbb S^{N-1}$, 
	\begin{equation*}
		\psi (x, \lambda, \nu)= \psi (x,-\lambda, -\nu);
	\end{equation*}
\item\label{(psi1)} there exist $c_\psi,C_\psi > 0$ such that, for all $x\in\Omega$, $\lambda \in \R{N}$, and $\nu \in \S{N-1}$,
	\begin{equation*}
		c_\psi|\lambda| \leq \psi(x,\lambda, \nu) \leq C_\psi|\lambda |;
	\end{equation*}
\item\label{(psi2)} (positive $1$-homogeneity) for all $x\in\Omega$, $\lambda \in \R{N}$, $\nu \in \S{N-1}$, and $t >0$
		$$\psi(x,t\lambda, \nu) = t\psi(x, \lambda, \nu);$$
\item \label{(psi3)} (sub-additivity) for all $x\in\Omega$, $\lambda_1,\lambda_2 \in \mathbb R^N$, and $\nu \in \S{N-1}$,
		\begin{equation*}
			\psi(x, \lambda_1 + \lambda_2, \nu) \leq \psi(x,\lambda_1, \nu) +\psi(x,\lambda_2, \nu);
		\end{equation*}
\item \label{(psi4)} (continuity in $x$)  there exists a continuous function $\omega_\psi\colon[0,+\infty)\to[0,+\infty)$ with $\omega_\psi(s)\to 0$ as $s\to 0^+$ such that, for every $x_0,x_1\in\Omega$, $\lambda \in \mathbb R^N$, and $\nu \in \S{N-1}$,
		\begin{equation*}
	|\psi(x_1,\lambda,\nu)-\psi(x_0,\lambda,\nu)|\leq\omega_\psi(|x_1-x_0|)|\lambda|.
		\end{equation*}
\end{enumerate}

In order to formulate the energy densities of the relaxation, we need to introduce some further notation.  Recalling   \eqref{stepfun}, we set  $v_{\lambda,\nu} := v_{0,\lambda,0, \nu}$ for $\lambda\in\R{N}$ and $\nu\in\S{N-1}$, i.e., 
\begin{equation}\label{909}
v_{\lambda,\nu}(x)\coloneqq 
\begin{cases}
\lambda & \text{if $x\cdot\nu>0$,} \\
0 & \text{if $x\cdot\nu\leq 0$.}
\end{cases}
\end{equation}
 Recall also $Q =(0,1)^N$ and the rotated cube $Q_\nu$ for  $\nu\in\S{N-1}$.  For $A,B\in\R{N\times N}_{\rm sym}$,  and for  $\lambda\in\R{N}$, $\nu\in\S{N-1}$, we consider the classes  \begin{equation}\label{T001}
\cC_p^{\bulk}(A,B)\coloneqq \bigg\{u\in SBD(Q) \colon \,  u(x)=Ax \text{ for $x \in \partial Q$},   \  \int_Q \mathcal E u\,\de x=B, \  |\mathcal E u|\in L^p(Q) \bigg\}, 
\end{equation}
and  
\begin{align*}
\cC_p^\surface(\lambda,\nu)&\coloneqq \Big\{u\in SBD(Q_\nu)\colon  \,  u(x)=v_{\lambda,\nu}(x) \text{ for $x \in \partial Q_\nu$},  \  \mathcal E u(x)=0 \ \text{for $\cL^N$-a.e.~$x\in Q_\nu$}\Big\} \quad \text{for $p > 1$},\\
\cC_1^\surface(\lambda,\nu) &\coloneqq \Big\{u\in SBD(Q_\nu) \colon u(x)=v_{\lambda,\nu}(x) \text{ for $x \in \partial Q_\nu$}, \   \int_{Q_\nu}  \mathcal E u\,\de x=0\Big\} \quad \text{for $p = 1$}.
\end{align*}
Then, we define the densities $H_p$ and $h_p$  of the bulk and surface parts, respectively,  as  
\begin{equation}\label{906}
H_p(x_0,A,B)\coloneqq \limsup_{\varepsilon \to 0}\inf_{u\in\cC_p^\bulk(A,B)}\bigg\{  
\int_Q W(x_0+ \varepsilon y,\mathcal E u(y))\,\de y+\int_{Q\cap J_u  } \psi( x_0 ,[u](y),\nu_u(y))\,\de\cH^{N-1}(y)\bigg\},
\end{equation}
for all $x_0\in\Omega$ and $A,B\in\R{N\times N}_{\rm sym}$,
and, for all $x_0\in\Omega$, $\lambda\in\R{N}$,   $\nu\in\S{N-1}$,
\begin{equation}\label{907}
\!\!\! h_p(x_0,\lambda,\nu)\coloneqq \inf_{u\in\cC_p^\surface(\lambda,\nu)}\bigg\{ 
 \! \int_{Q_\nu\cap J_u}  \psi(x_0,[u](x),\nu_u(x))\,\de\cH^{N-1}(x)\bigg\} 
\end{equation}
 in the case  $p >1$, as well as, in the case $p=1$, 
\begin{equation}\label{intf2}
\!\!\! h_1(x_0,\lambda,\nu)\coloneqq \limsup_{\varepsilon \to 0}\inf_{u\in\cC_1^\surface(\lambda,\nu)} \bigg\{ \int_{Q_\nu} \varepsilon W\Big(x_0 + \varepsilon y, \frac{1}{\varepsilon} \mathcal E(u)(y)\Big) \, \de y
+ \!\! \int_{Q_\nu\cap J_u} \!\! \psi(x_0,[u](x),\nu_u(x))\,\de\cH^{N-1}(x)\bigg\}.
\end{equation}
 Note that the continuity assumption on $\psi$ in {\rm \ref{(psi4)}} allows to `freeze'  the $x$-variable in the surface term $\psi$. We also observe that for $p>1$ the surface density $h_p$ decouples from the bulk density in the sense that the definition in \eqref{907} does not depend on $W$.   Under this set of hypotheses, we prove the following result.

\begin{theorem}[Relaxation]\label{representation}
	Let $p\geq 1$. Consider $F$ given by \eqref{103} where  $W\colon\Omega\times\R{N\times N}_{\rm sym}\to[0,+\infty)$ is Carath\'eodory, $\psi\colon\Omega\times\R{N}\times\S{N-1}\to[0,+\infty)$ is continuous and they satisfy {\rm \ref{(W1)_p}}--{\rm \ref{(psi1)}}.  
 Let $I_p$ be defined by \eqref{102}.    	Then,  $I_p$ satisfies {\rm (H1)}--{\rm (H4)} and there exist 
	$f\colon \Omega \times \mathbb R_{\rm sym}^{N \times N} \times \mathbb R_{\rm sym}^{N \times N} \to [0,+\infty)$, and $\Phi\colon \Omega \times \mathbb R^N \times  \mathbb{S}^{N-1}  \to [0, +\infty)$ such that  the following holds: 
	
\begin{itemize}
\item[(i)]   For every $(g,G) \in StSBD^p(\Omega)$, it holds that 
\begin{align}\label{reprelax}
		I_p(g,G) =  \int_\Omega f\big(x,\mathcal E g(x),G(x)\big) \, \de x + 
		\int_{\Omega\cap J_g}\Phi\big(x,[g](x),\nu_g(x)\big) \, \de \mathcal H^{N-1}(x), 
		\end{align}
 where the relaxed energy densities are given by (recall \eqref{stepfun}) 
	\begin{align}\label{fdef}
		f(x_0, \xi, B):=	\limsup_{\varepsilon \to 0^+}\frac{m(  \ell_{x_0,0,\xi} ,  B; Q(x_0,\varepsilon))}{\varepsilon^N},
	\end{align}
	\begin{align}\label{Phidef}
		\Phi(x_0, \lambda,  \nu):= \limsup_{\varepsilon \to 0^+}\frac{m(  v_{x_0,\lambda,0,\nu},  0; Q_{\nu}(x_0,\varepsilon))}{\varepsilon ^{N-1}},
	\end{align}
	for all $x_0\in \Omega$, $  \lambda   \in \mathbb R^d$, 
	$\xi, B \in \mathbb R^{N \times N}_{\rm sym}$, and $\nu \in \mathbb S^{N-1}$,  where the functional $m\colon StBD^p(\Omega) 
	\times \mathcal O_\infty(\Omega)\to [0,+\infty)$ is given by \eqref{mdef} with $I_p$ in place of  $\mathcal F$.  
 \item[(ii)] If more generally    $(g,G) \in StBD^p(\Omega)$, we have 
     \begin{align}\label{frepSD}
	\frac{{\rm d} I_p(g, G;\cdot)}{{\rm d}  \mathcal L^N} (x_0)=	
	f(x_0, \mathcal E g(x_0), G(x_0)) \ \ \text{    for $\mathcal L^N$-a.e.\ $x_0 \in \Omega,$}
	\end{align}
\begin{align}
\label{PhirepBD}\frac{{\rm d}  I_p(g, G; \cdot)}{{\rm d}  \mathcal H^{N-1}\lfloor{J_g}}(x_0) = \Phi(x_0, [g](x_0), \nu_g(x_0)) \ \ \text{ for $\mathcal H^{N-1}$-a.e.\ $x_0 \in J_g$.}
\end{align}
 \item[(iii)]  If $\psi$ also satisfies {\rm \ref{(psi2)}}--{\rm \ref{(psi4)}},  then  we can identify 
\begin{align}\label{eq. effiproof}
f(x_0, \xi, B)= H_p(x_0, \xi, B)
\end{align}
    for every $x_0\in \Omega$, $\xi, B \in \mathbb R^{N\times N}_{\rm sym}$ and
\begin{align}\label{eq. psiiproof}
{\Phi(x_0, \lambda,  \nu) = h_p(x_0, \lambda, \nu)}
\end{align}
	for every $x_0\in \Omega$, $ \lambda   \in \mathbb R^N$, and 
	$\nu \in \mathbb S^{N-1}$, where $H_p$ and $h_p$ are the functions given in \eqref{906}--\eqref{intf2}. 
\end{itemize}
\end{theorem}


\begin{remark}\label{surfacewithrecession}
Suppose that $W$ is uniformly continuous in  $x$, with  a modulus of continuity $\omega_W$ such that for every $x,x_0 \in \Omega$, $A \in \mathbb R^{N\times N}$,
 \begin{align}\label{(W4)}
 |W(x,A)-W(x_0,A)|\leq \omega_W(|x-x_0|)(|A|+1).
 \end{align}
 Assume further  that  there exists $C>0$ and $  0<\alpha <1$ such that
\begin{equation}\label{rateWinfty}
\left|W^\infty(x,A)-\frac{W(x,t A)}{t}\right|\leq 
 C  \frac{|A|^{1-\alpha}}{t^\alpha}
\end{equation}
for every $x \in \Omega$, whenever $t > 0$ and $t |A| \geq 1$, 
where $W^\infty(x,A)$ denotes the  \textit{weak recession function} at infinity of $W$ with respect to the second variable, namely
\begin{equation*}
W^{\infty}(x,A) \coloneqq \limsup_{t\rightarrow +\infty} 
\frac {W(x,tA)}{t}  \quad \text{ for all $x \in \Omega$ and  $A\in\R{N\times N}_{\rm sym}$}. \end{equation*} 
Then,
{\eqref{907} and \eqref{intf2} can be jointly written as}
\begin{equation}\label{907p}
\!\!\! h_p(x_0,\lambda,\nu) = \inf\bigg\{ 
\delta_1(p) \!\! \int_{Q_\nu}  W^\infty(x_0,\mathcal E u(x))\,\de x+ \! \int_{Q_\nu\cap J_u}   \psi(x_0,[u](x),\nu_u(x))\,\de\cH^{N-1}(x)\colon  
u\in\cC_p^\surface(\lambda,\nu)\bigg\},
\end{equation}
where $\delta_1(p) = 1$ if $p=1$ and $\delta_1(p) = 0$ if $p\neq 1$. This means that the relaxed surface energy density depends on the recession function of $W$ only in the case $p=1$.  To see this, it is enough to estimate the limit in \eqref{intf2} using \eqref{(W4)} to ‘freeze’ the $x$-variable (similarly, as done in the surface term) and then obtaining the recession function $W^\infty$ by means of \eqref{rateWinfty} as in the proof of \cite[(2.12)]{CF1997}.  Note that \eqref{907p} also recovers  the case $p>1$ in  \eqref{907}.  
\end{remark}

\begin{remark}
    \label{remnox}
   Arguing as in \cite[Remark 3.10]{BFM}, under the assumption of Theorem \ref{representation} and the existence of  a  modulus of continuity on $W(\cdot,A)$,  the densities $f$ and $\Phi$ in \eqref{fdef} and \eqref{Phidef} turn out to be continuous with  respect to $x_0$. 
    
 Moreover,   if in Theorem \ref{representation}, the energy densities $W$ and $\psi$ in \eqref{103} do not depend on $x$, then arguments entirely similar to \cite[Lemma 3.7]{BDV} entail that the functional $I_p$ in \eqref{102} (extended as a set function to any open subset of $\Omega$)   satisfies
$I_p(\tau_y g,\tau_y G;\tau_y O)= I_p(g,G;O)$ for every $(g,G)\in StBD^p(\Omega)$ such that $\tau_y O$ is an open subset of $\Omega$.
  Then, by Remark \ref{traslinv}, $f$ and $\Phi$ defined in \eqref{fdef} and \eqref{Phidef}, respectively, do not depend on the first variable.
\end{remark}

\begin{proposition}[Properties of the relaxed energy densities]\label{thm_propdens}
Let $p\geq1$, and let $W$ and $\psi$ satisfy {\rm \ref{(W1)_p}}--{\rm \ref{(psi4)}}. 
The function $H_{p}$   defined in~\eqref{906} is $p$-Lipschitz continuous in the third component, namely for every $A\in \R{N \times N}_{\rm sym}$ there exists a constant $C>0$ such that for almost every $x\in\Omega$ and for every $B_1,B_2\in\R{N\times N}_{\rm sym}$
\begin{equation}\label{H_B}
|H_{p}(x,A,B_1)-H_{p}(x,A,B_2)|\leq C|B_1-B_2|(1+|B_1|^{p-1}+|B_2|^{p-1}).
\end{equation}
 Moreover,   there exist constants $\bar{c}_H,\overline{C}_{H}>0$ such that for almost every $x\in\Omega$ and for every $A,B\in\R{N\times N}_{\rm sym}$
\begin{equation}\label{JoE2.27}
\bar{c}_H (|A|+|B|^p)-\frac1{\bar{c}_H} \leq H_{p}(x,A,B) \leq \overline{C}_H (1+|A|+|B|^p).
\end{equation}
For $B\in\R{N\times N}_{\rm sym}$, let   $H_{p}^B \colon\Omega\times\R{N\times N}\to[0,+\infty)$ be defined by $(x,A)\mapsto H_{p}^B(x,A)\coloneqq H_{p}(x,A,B)$.
Then,
\begin{itemize}
\item[(i)] if $p>1$, then $(H_{p}^B, h_{p}) $ satisfy {\rm \ref{(W1)_p}}--{\rm \ref{(psi4)}};

\item[(ii)] if $p=1$, the function $H_{1}^B$  is Carath\'eodory satisfying {\rm \ref{(W1)_p}}--{\rm \ref{W4}} with $p=1$, and the function $h_{1}$ satisfies properties {\rm \ref{psi_0}} and {\rm \ref{(psi1)}}. 
\end{itemize}
\end{proposition}

We note that Theorem \ref{representation} holds for functions in $BD(\Omega)$, but characterizes only the bulk and jump density. Now, we generalize the result by identifying also the Cantor part.  To this end,  we specify to $x$-independent energies, i.e, we consider
\begin{equation}\label{103bis}
	F(u)\coloneqq \int_\Omega W(\mathcal E u(x))\,\de x+\int_{\Omega\cap J_u} \psi\big([u](x), \nu_u(x)\big)\,\de\cH^{N-1}(x),
\end{equation}
where $W\colon \R{N\times N}_{\rm sym}\to[0,+\infty)$ and $\psi\colon \R{N}\times\S{N-1}\to[0,+\infty)$ are continuous  and satisfy {\rm \ref{(W1)_p}}--{\rm \ref{(psi4)}} for $p \geq 1$. 
  We have the following result.

\begin{theorem}[Relaxation with Cantor part]\label{representationbis}
 Consider $F$ given by \eqref{103bis},  where   $W\colon \R{N\times N}_{\rm sym}\to[0,+\infty)$ and $\psi\colon \R{N}\times\S{N-1}\to[0,+\infty)$ are continuous and satisfy {\rm \ref{(W1)_p}}--{\rm \ref{(psi4)}} for  $p\geq 1$.  Let $(g,G) \in StBD^{p}(\Omega)$,  and assume that $ I_p (g,G)$ is defined by \eqref{102}. 	Then, 
	\begin{align}\label{reprelaxbis}
 I_p(g,G) =  \int_\Omega H_p(\mathcal E g ,G ) \, \de x + 
		\int_{\Omega\cap J_g}h_p([g] ,\nu_g) \, \de \mathcal H^{N-1}  + \int_{\Omega} H_p^\infty\left(\frac{d E^c g}{d |E^c g|} ,0\right) d |E^c g|, 
		\end{align}
	 where the relaxed energy densities are given by
    \eqref{906}--\eqref{intf2}, and  the weak recession  function of $H_p(\cdot,0)$ (i.e., 
    $H_p^\infty(A,0):=\limsup_{t \to +\infty} \frac{H_p(tA,0)- H_p(0,0)}{t}$) 
    in the bulk, surface, and Cantor part,  respectively. \color{black}
  
\end{theorem}

The results of this subsection are proved in Section \ref{relax}.


\subsection{Linearization of energies associated to a structured deformation}\label{sec: lin}

In the previous subsections we have developed a theory for structured deformations in linearized elasticity. As in pure elasticity, models of this kind can be understood as an approximation of nonlinear elasticity. In the spirit of rigorous linearization results in elasticity 
\cite{DalMasoNegriPercivale:02}  and fracture mechanics \cite{Friedrich:15-2, Friedrich}, we now discuss a setting which allows for a rigorous derivation of the energies  in Theorems \ref{representation} and \ref{representationbis} from a nonlinear bulk-surface energy. More precisely,   we  consider a nonlinear energy with  second-order contributions, both in the elastic and the surface part. Our goal is to simultaneously  pass to (a) infinitesimal strains and  to (b) a relaxed formulation in terms of structured deformations on $BD$.

We first introduce the relevant function space, following  the work on linearization for Griffith energies \cite{Friedrich}.  Similarly to the space $SBV^2(\Omega;\mathbb{R}^N)$,   we define the space
\begin{align*}
SBV^2_2(\Omega;\mathbb{R}^N) := \big\{ y \in SBV^2(\Omega; \mathbb{R}^N): \ \nabla y \in SBV^2(\Omega;\mathbb{R}^{N\times N})\big\}.
\end{align*}
The approximate differential and the jump set of $\nabla y$  is denoted by $\nabla^2 y$ and $J_{\nabla y}$, respectively. The jump height related to $\nabla y$ will be denoted by $[\nabla y]$.

 We let $V\colon \mathbb{R}^{N \times N} \to [0,+\infty)$ be a  single well, frame indifferent stored energy density.  More precisely,   for some $c>0$  we have 
\setlist[enumerate]{label=(V\arabic*)}
\begin{enumerate} 
\item  $V \text{ continuous, and $C^3$ in a neighborhood of $SO(N)$}$;  \label{assumptions-Wi}
\item  $\text{Frame indifference: } V(RZ) = V(Z) \text{ for all } Z \in \mathbb{R}^{N \times N}, R \in SO(N)$;  \label{assumptions-Wii} 
\item \label{assumptions-Wiii} $V(Z) \geq c{\rm dist}^2(Z,SO(N)) \ \text{ for all $Z \in \mathbb{R}^{N \times N}$}; \  V(Z) = 0 \text{ iff } Z \in SO(N)$.
\end{enumerate}
As explained in \cite[Remark 4.2]{Friedrich}, it would suffice to assume that $V$ lies in the H\"older space $C^{2,\alpha}$ for some exponent $\alpha \in (0,1]$, but we do not dwell on this point in the sequel.  

We let  $\psi\colon    \mathbb{R}^N \times \mathbb{S}^{N-1} \to [0,+\infty) $ be a continuous density satisfying {\rm \ref{psi_0}}--{\rm \ref{(psi4)}}, and additionally 
\setlist[enumerate]{label=($\psi$\arabic*)}
\begin{enumerate} 
\setcounter{enumi}{5}
\item\label{(psi5)} (frame indifference) for all   $\lambda \in \R{N}$, $\nu \in \mathbb{S}^{N-1}$,  and  $R \in SO(N)$ 
		$$\psi(R\lambda, \nu) = \psi( \lambda, \nu).$$
\end{enumerate}
Moreover, we let  $\Psi \colon    \mathbb{R}^{N\times N} \times \mathbb S^{N-1} \to [0,+\infty) $ be a continuous density satisfying {\rm \ref{psi_0}}--{\rm \ref{(psi1)}} and {\rm \ref{(psi5)}} with the only difference that $\lambda \in \mathbb{R}^N$ is replaced by $\lambda \in \mathbb{R}^{N\times N}$. 

Given the small-strain parameter $\delta>0$ and a parameter $\beta \in (\max\{\frac{2}{3},\frac{N-1}{N}\},1)$, we consider the nonlinear energy $ {F}_\delta \colon SBV^2_2(\Omega;\mathbb{R}^N) \to [0,+\infty)$, defined by
\begin{align}\label{F-en}
F_\delta(y)  = \int_\Omega \frac{1}{\delta^2} V(\nabla y) + \frac{1}{\delta^{2\beta}} |\nabla^2y|^2 \, {\rm d}x  + \frac{1}{\delta} \int_{J_y}  \psi \big( [y], \nu_y\big) \, {\rm d}\mathcal{H}^{N-1} +  \frac{ 1}{\delta^{\beta}} \int_{J_{\nabla y}}  \Psi \big( [\nabla y], \nu_{\nabla y}\big) \, {\rm d}\mathcal{H}^{N-1}    
\end{align}
for each $y\in SBV_2^2(\Omega;\mathbb{R}^N)$.   Due to the presence of the second and fourth term, we deal with a model   for \emph{nonsimple materials} going back to the ideas by {Toupin} \cite{Toupin:62,Toupin:64}: the main point is that the dependence of the elastic energy on the second gradient and the dependence of the surface energy on the first gradient  enhance    compactness and rigidity properties.     

Our idea is to perform a linearization about the identity in terms of the  \emph{rescaled displacement field}
$$u:=  \frac{y - {\rm id} }{\delta}, $$
and to pass to the limit $\delta\to 0$ in terms of this variable. Here, the parameter $\delta>0$ corresponds to the typical scaling of strains for configurations with finite energy.  Due to the frame indifference of the problem, see \ref{assumptions-Wii} and {\rm \ref{(psi5)}},  it is not a priori guaranteed that $y$ is close to the identity. Therefore, without restriction we consider  configurations $y$  which satisfy
\begin{align}\label{eq: rot back}
\fint_\Omega y(x) \,  {\rm d}x  =   \fint_\Omega x \,  {\rm d}x \quad \quad \text{and} \quad \quad     {\rm argmin}_{R \in SO(N)}   \Big|  \fint_\Omega \nabla y \, {\rm d}x - R \Big|  = {\rm Id}.
\end{align}
Indeed, this can always be achieved by replacing $y$ with $Qy +b$ for some $Q \in SO(N)$ and $b \in \mathbb{R}^N$, noting that, due to \ref{assumptions-Wii} and {\rm \ref{(psi5)}}, the energy is invariant under this transformation. 

We write  the energy in terms of the displacement under constraint \eqref{eq: rot back}. More precisely,  we define  $  {F}^{\rm dis}_{\delta}  \colon SBV^2_2(\Omega;\mathbb{R}^N) \to [0,+\infty]$ by
\begin{align}\label{def-energy}
{F}^{\rm dis}_{\delta} (u) = \begin{cases}  F_\delta({\rm id} + \delta u) & \text{if }  \text{${\rm id} + \delta u$ satisfies   \eqref{eq: rot back}},\\
+\infty & \text{else}. \end{cases}
\end{align}
Here and in the sequel, we follow the usual convention in the theory of $\Gamma$-convergence that convergence of the continuous parameter $\delta \to 0$ stands for convergence of arbitrary sequences $\lbrace \delta_n \rbrace$ with $\delta_n \to 0$ as $n \to \infty$, see e.g.\ \cite[Definition 1.45]{Braides:02}.

\begin{proposition}[Rigidity and compactness]\label{prop: rigidity}
Let $M>0$, $\frac{2}{3} < \gamma < \beta$, and $\{u_\delta\}  \subset SBV^2_2(\Omega;\mathbb{R}^N)$ with $ {F}_\delta  (u_\delta) \leq M$ for all $\delta>0$. Then, there exist sets of finite perimeter $\{S_\delta\}$ in $\Omega$ such that
\begin{align}\label{S-conv}
\lim_{\delta \to 0}  \mathcal{H}^{N-1}(\partial^* S_\delta  )  = 0 
\end{align}
and 
\begin{align}\label{gradientbound}
{\rm (i)}  & \ \  \Vert u_\delta \Vert_{L^1(  \Omega)} \leq C, \notag  \\
{\rm (ii)}  & \ \  \Vert \nabla u_\delta \Vert_{L^\infty(  \Omega \setminus S_\delta)} \leq C \delta^{\gamma-1},\notag\\
{\rm (iii)}  & \ \  \Vert \mathcal{E}  u_\delta \Vert_{L^2(  \Omega \setminus S_\delta)} \leq C .
  \end{align}
where the constant $C>0$ depends only on $M$ and $\Omega$. In particular, there exists a subsequence (not relabeled) and  $(g,G) \in StBD^2(\Omega)$  such that 
$u_\delta  \rightarrow g$ in $L^1(\Omega;{\mathbb{R}}^N)$ and  $\mathcal E u_\delta \, \chi_{\Omega \setminus S_\delta} \rightharpoonup G$ weakly in $L^2(\Omega; \mathbb R^{N\times N}_{\rm sym})$ as $\delta \to 0$. 
\end{proposition}

\begin{definition}[Convergence]
We say that a sequence $\{ u_\delta\} $ \emph{converges almost in  $SD$} to $(g,G) \in StBD^2(\Omega)$, and write $u_\delta \rightsquigarrow (g,G)$ if there exist sets of finite perimeter $\{S_\delta\}$ satisfying \eqref{S-conv} such that $u_\delta  \rightarrow g$ in $L^1(\Omega;{\mathbb{R}}^N)$ and  $\mathcal E u_\delta \, \chi_{\Omega \setminus S_\delta} \rightharpoonup G$ weakly in $L^2(\Omega; \mathbb R^{N\times N}_{\rm sym})$.
\end{definition}

Next, we introduce the limiting energy. As an auxiliary step, for $O \in \mathcal O(\Omega)$ and $u \in SBD(\Omega)$, we define
\begin{align}\label{Elin}
F^{\rm dis}(u,O)\coloneqq \int_O W(\mathcal E u)\,\de x+\int_{O\cap J_u} \psi([u], \nu_u)\,\de\cH^{N-1},
\end{align}
where $W$ is given as  the linearization of the energy density $V$ at the identity, i.e.,   $W(Z) :=  \frac{1}{2}D^2V({\rm Id}) Z \colon Z$ for all  $Z \in \mathbb{R}^{N \times N}$. By \ref{assumptions-Wii}--\ref{assumptions-Wiii},  $Z\mapsto W(Z)$ depends only on ${\rm sym}(Z)$   and is a positively definite  quadratic form on $\mathbb{R}^{N \times N}_{\rm sym}$. In particular, $W$  satisfies  {\rm \ref{(W1)_p}}--{\rm \ref{W4}} for $p=2$. 

 Applying Theorem \ref{representation} and Remark \ref{remnox} we define the relaxation of $F^{\rm dis}$  as 
$$I_{\rm lin}(g,G) =  \int_\Omega H(\mathcal E g,G) \, \de x + 
		\int_{J_g}h([g], \nu_g) \, \de \mathcal H^{N-1}  +   \int_{\Omega} H^\infty\left(\frac{d E^c g}{d |E^c g|} ,0\right) d |E^c g| 
		$$
for $(g,G) \in StBD^2(\Omega)$, where the densities $H$ and $h$ are given in \eqref{906} and \eqref{907} for $p=2$, respectively,  and 	$H^\infty(\cdot,0)$ denotes the weak recession function of 	$H(\cdot,0)$.



 
 \color{black} 


\begin{theorem}[$\Gamma$-convergence]\label{th: gamma-conv}
The sequence $\{{F}^{\rm dis}_{\delta}\}$ $\Gamma$-converges (with respect to the convergence $\rightsquigarrow$) to the functional  $I_{\rm lin}$. More precisely, we have: 
\begin{itemize}
\item[(i)] (Ansatz-free lower bound) For each $(g,G)\in StBD^2(\Omega)$ and each sequence $\{u_\delta\}$ with $u_\delta \rightsquigarrow (g,G)$ as $\delta\to 0$, we have
$$\liminf_{\delta \to 0} {F}^{\rm dis}_{\delta} (u_\delta) \geq I_{\rm lin}(g,G); $$

\item[(ii)] (Recovery sequence)  For each $(g,G)\in StBD^2(\Omega)$ there exists a  sequence $\{ u_\delta\}$ with $u_\delta  \rightsquigarrow  (g,G)$ as $\delta\to 0$ and $$\lim_{\delta \to 0} {F}^{\rm dis}_{\delta} (u_\delta) = I_{\rm lin}(g,G). $$
\end{itemize}
\end{theorem}

 The results announced in this subsection will be proved in Section \ref{proof:lin}.

\section{The global method in $StBD^p$ and $StSBD^p$}\label{gm}

 This section is devoted to the proof of Theorem \ref{GMthmHSDsbd}. It is based on several auxiliary results and follows the reasoning introduced in \cite[Theorem 3.7]{BFM}, developed in \cite[Theorem 4.6]{FHP}  (with the extension to second order gradients and  with the presence of an uncostrained field) and in \cite[Theorem 2.3]{CFVG} for the $BD$ setting. For this reason, we do not provide the arguments in full detail but point out only the main differences that arise in our setting.  On the other hand, for the reader's convenience, we introduce all the tools necessary for the Global Method in order to allow a direct comparison with other cases  in the  literature. We start with a simple observation.

 \begin{remark}\label{A1A2}
{\rm Due to hypotheses (H1) and (H4), given any $(g, G) \in StBD^p(\Omega)$ \color{black} and any open sets $O_1 \Subset O_2 \subseteq \Omega$, it follows that 
\begin{align*}
\mathcal F(g, G;O_2) \leq & \, \mathcal F(g, G;O_1)  +C\left(\mathcal L^N(O_2 \setminus O_1) + \|G\|_{L^p(O_2\setminus O_1)}^p+  | E g|(O_2 \setminus O_1)\right).
\end{align*}
Indeed, for $\varepsilon > 0$ small enough, let $O_{\eps} : = \left\{x \in O_1 : {\rm dist}(x, \partial O_1) > \eps\right\}$ and notice that $O_2$ is covered by the union of the two open sets $O_1$ and $O_2 \setminus \overline{O_{\eps}}.$ Thus, by (H1) and (H4) we have
\begin{align*}
\mathcal F(g, G;O_2) \leq & \, \mathcal F(g, G;O_1) + \mathcal F(g, G;O_2 \setminus \overline{O_{\eps}}) \\
\leq & \, \mathcal F(g, G;O_1)  +
C\left(\mathcal L^N(O_2 \setminus  \overline{O_{\eps}}) + \|G\|^p_{L^p(O_2\setminus \overline{O_{\eps}})} 
+ | E g|(O_2 \setminus\overline{O_{\eps}})\right).
\end{align*}
To conclude, it suffices to let $\eps \to 0^+$.
}
\end{remark}

 The following lemma is crucial for Theorem \ref{thm4.3FHP} below.  Recall the definition of $m$ in \eqref{mdef}.

\begin{lemma}\label{estHSDL}
 Assume that {\rm (H1)} and {\rm (H4)} hold in   $StBD^p(\Omega)$. For any $(u, U)\in StBD^p(\Omega)$   it follows that  
\begin{itemize}
	\item if $p>1$ and  $Q_\nu(x_0, r) \subset \Omega$,  we have 
	\begin{equation}\label{1m}
	\limsup_{\delta \to 0^+} m(u, U;Q_\nu(x_0,(1-\delta) r))
	\leq m(u, U;Q_\nu(x_0,r));
	\end{equation}
	

	\item if $p=1$ and  $O \in \mathcal O(\Omega)$,  we have 
	\begin{equation*}
	\limsup_{\delta \to 0^+} m(u, U;O_\delta)
	\leq m(u, U; O),
	\end{equation*}
	where $ O_\delta=\{x \in O: {\rm dist}(x, \partial O) > \delta\}$.
	\end{itemize}
\end{lemma}

\begin{proof}  
Suppose first that $p > 1$. Without loss of generality we can assume that $x_0=0$, $r=1$, $\nu=\bf e_1$ and $Q  := Q(1) = Q_{\bf e_1}(0, 1) \subset \Omega$.  For every $\varepsilon >0$ there exists $(v, V)\in \mathcal C_{StBD^p}( u, U; Q)$  such that 
\begin{equation}\label{uQ}
\mathcal F(v, V;Q)\leq m(u, U;Q)+ \varepsilon.
\end{equation}
Let $0 < \delta < 1$ be small enough so that $u = v$ in a neighborhood of $Q \setminus Q(1-2\delta)$, and let $\delta < \alpha(\delta) < 2\delta$ be such that 
$\displaystyle \lim_{\delta \to 0^+}\alpha(\delta) = 0$ and 
\begin{equation}\label{bound}
\dfrac{\mathcal L^N(Q \setminus Q(1-\alpha(\delta)))}{\mathcal L^N(Q(1-\delta) \setminus Q(1-\alpha(\delta)))} \leq C,
\end{equation}
where the constant $C$ depends only on the space dimension $N$ and is, therefore, independent of $\delta$. Define
$$
\overline V= \begin{cases}
V, &\hbox{ in }Q(1-\alpha(\delta))\\
\dfrac{1}{\mathcal L^N(Q(1-\delta) \setminus Q(1-\alpha(\delta)))}\Big(\displaystyle \int_{Q(1-\delta)}U \, \de x- \int_{Q(1-\alpha(\delta))}V \, \de x\Big), &\hbox{ in }Q(1-\delta)\setminus Q(1-\alpha(\delta))\\*[5mm]
U, &\hbox{ in } \Omega\setminus Q(1-\delta)	
\end{cases}
$$
and
$$
\overline v= \begin{cases} 
v, &\hbox{ in } Q(1-\alpha(\delta)) \\
u, &\hbox{ in } \Omega\setminus Q(1-\alpha(\delta)).
\end{cases}
$$
It is easily verified that $(\overline v, \overline V) \in  \mathcal C_{StBD^p}( u, U;  Q(1-\delta))$, see e.g. \cite[Proposition 2.5]{BFT}. Thus, by Remark \ref{A1A2}, by (H1), by $ \mathcal F(v, V; Q(1-\alpha(\delta))) \leq \mathcal F(v, V;Q)$, and    \eqref{uQ} we have
\begin{align}\label{uQ1-delta}
m(u, U;Q(1-\delta))&\leq 
\mathcal F(\overline v, \overline V;Q(1-\delta))\nonumber\\
& \leq \mathcal F(v, V; Q(1-\alpha(\delta))) 
+ C\Big[\mathcal L^N(Q_*^\delta )   + | E \overline v|(Q_*^\delta )
 + \|\overline V\|^p_{L^p(Q_*^\delta )}\Big]\nonumber\\
&\leq   m(u, U;Q)+ \varepsilon + 
C\Big[\mathcal L^N(Q_*^\delta )  + | E \overline v|(Q_*^\delta ) 
 + \|\overline V\|^p_{L^p(Q_*^\delta )}\Big],
\end{align}
 where we have set  $Q_*^\delta := Q(1-\delta) \setminus Q(1-\alpha(\delta))$ for convenience. 
Clearly, $\displaystyle \lim_{\delta \to 0^+}\mathcal L^N(Q_*^\delta ) = 0$ and, since $u = v$ on $\partial Q(1-\alpha(\delta))$, it also follows that 
\begin{align}\label{pnormxxx}
\lim_{\delta \to 0^+}|E \overline v|(Q_*^\delta ) =0.
\end{align}
On the other hand,   we have
\begin{align}\label{pnorm}
&\|\overline V\|^p_{L^p(Q_*^\delta)}  =\frac{1}{(\mathcal L^N(Q_*^\delta))^{p-1}} \left|\int_{Q(1-\delta)} U \, \de x - \int_{Q(1-\alpha(\delta))} V \, \de x\right|^p \nonumber\\
& \hspace{2cm}= \frac{1}{(\mathcal L^N(Q_*^\delta))^{p-1}} \left|\int_{Q(1-\alpha(\delta))} (U - V) \, \de x + \int_{Q_*^\delta} U \, \de x\right|^p \nonumber \\
& \hspace{2cm}\leq \frac{C}{(\mathcal L^N(Q_*^\delta))^{p-1}} \left(\left|\int_{Q(1-\alpha(\delta))} (U - V) \, \de x\right|^p + \left|\int_{Q_*^\delta} U \, \de x\right|^p\right).
\end{align}
Recalling that $\displaystyle \int_Q (U - V) \, \de x = 0$, 
the first term on the right-hand side of \eqref{pnorm} can be estimated by using H\"older's inequality, yielding
\begin{align*}
\frac{C}{(\mathcal L^N(Q_*^\delta))^{p-1}} \left|\int_{Q(1-\alpha(\delta))} (U - V) \, \de x\right|^p  
& = \frac{C}{(\mathcal L^N(Q_*^\delta))^{p-1}} \left|\int_{Q\setminus Q(1-\alpha(\delta))} (U - V) \, \de x\right|^p \nonumber \\
& \leq \frac{C}{(\mathcal L^N(Q_*^\delta))^{p-1}}
\, \|U - V\|^p_{L^p(Q \setminus Q(1-\alpha(\delta)))}
(\mathcal L^N(Q \setminus Q(1-\alpha(\delta))))^{p-1}.	
\end{align*}	
By \eqref{bound} and the fact that $\displaystyle \lim_{\delta \to 0^+}\mathcal L^N(Q\setminus Q(1-\alpha(\delta))) = 0$  we conclude that the first term on the right-hand side of \eqref{pnorm}  converges to $0$ as $\delta \to 0$.  Regarding the second term, a similar argument using H\"older's inequality leads to
\begin{align*}
& \limsup_{\delta \to 0^+} 
\frac{C}{(\mathcal L^N(Q_*^\delta))^{p-1}}
\left|\int_{Q_*^\delta} U \, \de x\right|^p  \leq \lim_{\delta \to 0^+}C\|U\|^p_{L^p(Q_*^\delta)} = 0.
\end{align*}
Therefore, combining our findings for \eqref{pnorm} with \eqref{pnormxxx}, from \eqref{uQ1-delta} we obtain
$$\limsup_{\delta\to 0^+}m(u, U;Q(1-\delta))\leq 
m(u, U;Q)+ \varepsilon.$$
It suffices to let $\varepsilon \to 0^+$ to complete the proof in the case 
$p > 1$.

When $p = 1$, the proof is similar and we omit the details. In this case, the estimate of the last term in \eqref{uQ1-delta} is simpler and does not require the use of H\"older's inequality. Also, in this case, more general sets other than cubes may be considered as there is no need to use inequality \eqref{bound} (see also \cite{FHP}).
\end{proof}

For the reader's convenience, with the sole aim of allowing a  comparison  with the theory introduced in \color{black} \cite{BFM, BFLM}, we define
$$\mathcal O^\star(\Omega):= \{Q_\nu (x, \varepsilon) \colon \, x \in \Omega, \nu \in \mathbb S^{N-1}, \varepsilon > 0\}.$$
For    $(u,U) \in StBD^p(\Omega)$  fixed, we set $\mu := \mathcal L^N\lfloor \Omega + | E^s u|$, 
and, given also $O \in \mathcal O(\Omega)$ and $\delta > 0$, we let
\begin{align*}
m^\delta(u,U;O) := \inf \Big\{\sum_{i=1}^\infty m(u, U; Q_i): Q_i \in \mathcal O^\star(\Omega), Q_i \subseteq O, Q_i \cap Q_j= \emptyset \; {\rm if } \; i \neq j, \notag  \\
 \quad \quad 	 {\rm diam}(Q_i) < \delta, \  \mu\Big(O \setminus \bigcup_{i=1}^\infty Q_i\Big)=0\Big\}.
\end{align*}
Since $\delta \mapsto m^\delta(u,U;O)$ is a decreasing function, we can define 
\begin{align*}
m^\star (u, U;O):= \sup_{\delta >0} m^\delta (u, U;O)= \lim_{\delta \to 0^+} m^\delta(u, U;O).
\end{align*}
Adapting the reasoning in \cite[Lemma 4.2 and Theorem 4.3]{FHP}  (where an extra unconstrained field $U$ is considered) and exploiting  a Poincar\'e's-type inequality in $BD$  (see e.g.\ \cite[Theorem 2.7]{BFT}) \color{black} 
we obtain the two results below. 

\begin{lemma}
	\label{Lemma 4.2FHP} Let $p \geq 1$ and
	assume that {\rm (H1)}--{\rm (H4)} hold. Then, 
	for all $(u, U)\in StBD^p(\Omega)$  and all 
	$O \in \mathcal O(\Omega)$,
	we have
	$$
	\mathcal F(u, U;O) = m^\star(u, U; O).
	$$
\end{lemma}


\begin{theorem}\label{thm4.3FHP}
	Let $p\geq 1$ and assume that hypotheses {\rm (H1)}, {\rm (H2)}, and {\rm (H4)} hold. Then, for every $\nu \in \mathbb S^{N-1}$ and for every  $(u,U) \in StBD^p(\Omega)$, 
	we have 
	$$
	\lim_{\varepsilon \to 0^+}\frac{{\mathcal F}(u, U; Q_\nu(x_0, \varepsilon))}{\mu( Q_\nu(x_0,\varepsilon))}= \lim_{\varepsilon \to 0^+} \frac{m (u, U; Q_\nu(x_0, \varepsilon))}{\mu (Q_\nu(x_0,\varepsilon))}
	$$
	for $\mu$-a.e.\ $x_0 \in \Omega$, where $\mu:=\mathcal L^N\lfloor \Omega + |E^s u|.$
	\end{theorem}


  The proof of this  result is not presented since it follows along the lines of \cite[Theorem 4.3]{FHP},   the only difference being the limsup estimate \eqref{1m} (for $p>1$) obtained along cubes and not generic open sets. We now present the proof of the  main result of this section.

\begin{proof}[Proof of Proposition \ref{GMthmHSD}]
\emph{Step 1.}	In this step we prove  that, for $\mathcal L^N$-a.e.\ $x_0 \in \Omega$,
	\begin{align}\label{fproof}
	\frac{{\rm d} \mathcal F(u, U;\cdot)}{{\rm d} \mathcal L^N} (x_0)=	
	f\big(x_0, u(x_0),   \nabla   u(x_0), U(x_0)\big).
	\end{align}
Let $x_0$ be a fixed point in $\Omega$ satisfying the following properties 
\begin{align}
&\lim_{\varepsilon \to 0^+}\frac{1}{\varepsilon}\ave_{Q(x_0,\varepsilon)}|u(x) - u(x_0) - \nabla u(x_0)(x-x_0)| \, \de x = 0;\label{aproxdif} \\
& \lim_{\varepsilon \to 0^+}\frac{1}{\varepsilon^N}|Eu|(Q(x_0,\varepsilon)) = |\mathcal E u(x_0)|, \quad \lim_{\varepsilon \to 0^+}\frac{1}{\varepsilon^N}|E ^su|(Q(x_0,\varepsilon)) = 0;\label{DuDsu}\\
& \lim_{\varepsilon \to 0^+}\ave_{Q(x_0,\varepsilon)}|U(x) - U(x_0)| \, \de x = 0; \label{ULeb}\\
& \frac{{\rm d} \mathcal F(u, U;\cdot)}{{\rm d} \mathcal L^N} (x_0) =
\lim_{\varepsilon \to 0^+}\frac{\mathcal F(u, U;Q(x_0,\varepsilon))}{\varepsilon^N} =
\lim_{\varepsilon \to 0^+}\frac{m(u, U;Q(x_0,\varepsilon))}{\varepsilon^N}.\label{Fmu}
\end{align}
  It is well known that the above properties hold for $\mathcal L^N$-a.e.\ $x_0$ in $\Omega$, taking also in consideration Theorem~\ref{thm4.3FHP} for \eqref{Fmu}, 
  as well as Theorem \ref{approxsymdiff} which provides an $N\times N$ matrix $\nabla u(x_0)$ such that \eqref{aproxdif} is valid.

Having fixed $x_0$ as above, let $\varepsilon >0$ be small enough so that $Q(x_0,\varepsilon)\subset \Omega$. Given the definition of the density $f$ in \eqref{f}, due to \eqref{Fmu},  we want to show that
\begin{align}\label{toestimate}
	\lim_{\varepsilon \to 0^+}
	\frac{m(\ell_{x_0},U(x_0);Q (x_0, \varepsilon))}
	{\mathcal L^N (Q (x_0, \varepsilon))}  = \lim_{\varepsilon \to 0^+}\frac{m(u,U; Q (x_0, \varepsilon))}{\mathcal L^N (Q (x_0, \varepsilon))},
\end{align}
 where  for simplicity, we write $\ell_{x_0} := \ell_{x_0,u(x_0), \nabla   u(x_0)} $ $ = u(x_0)+   \nabla   u(x_0)(\cdot-x_0)$.  
 
 Let  $\delta \in (0,1)$   and  let  $(\widetilde u, \widetilde U) \in  C_{StBD^p}(\ell_{x_0}, U(x_0); Q(x_0, \delta \varepsilon))$ \color{black}  be such that
\begin{equation}\label{infmva}
\varepsilon^{N+1} + m(\ell_{x_0}, U(x_0); Q(x_0, \delta \varepsilon)) \geq \mathcal F(\widetilde u, \widetilde U; Q(x_0, \delta \varepsilon)).
\end{equation}
Then, as $\widetilde u = \ell_{x_0}$ on $\partial Q(x_0, \delta \varepsilon)$, we have
\begin{align}\label{trvaest}
|{\rm tr}\, u-{\rm tr}\,\widetilde u|(\partial Q(x_0,\delta \varepsilon)):= \int_{\partial Q(x_0,\delta\varepsilon)} |\widetilde u- u| \,\de \mathcal H^{N-1}=\int_{\partial Q(x_0,\delta\varepsilon)} |\ell_{x_0}- u|\, \de \mathcal H^{N-1}.
\end{align}
 We define
$$
\widetilde v_\varepsilon:=\left\{
\begin{array}{ll}
\widetilde u, &\hbox{ in }Q(x_0,\delta \varepsilon),\\
u, &\hbox{ in }\Omega\setminus Q(x_0,\delta \varepsilon),
\end{array}
\right.
$$
and let
$$\widetilde V_\varepsilon(x):= \left\{
\begin{array}{ll}
\widetilde U(x), & \hbox{ in }Q(x_0,\delta \varepsilon),\\
\displaystyle \frac{1}{\mathcal L^N(Q(x_0,\varepsilon)\setminus Q(x_0,\delta \varepsilon))}\left[\int_{Q(x_0,\varepsilon)} U(x)\, \de x- \int_{Q(x_0,\delta \varepsilon)} U(x_0)\,\de x\right], &\hbox{ in }\Omega\setminus Q(x_0,\delta \varepsilon).
\end{array}\right.
$$
Recall that  $\displaystyle \int_{Q(x_0,\delta \varepsilon)} \widetilde U(x) \,\de x = \int_{Q(x_0,\delta \varepsilon)} U(x_0) \,\de x = U(x_0)(\delta \varepsilon)^N$, so we get
 $(\widetilde v_\varepsilon, \widetilde V_{\varepsilon})\in \mathcal  C_{StBD^p}(u, U; Q(x_0,\varepsilon))$ (see \cite[Chapter II, Proposition 2.1]{T}).  Hence, 
 by Remark \ref{A1A2}, 
 we have
\begin{align}\label{mest}
	m(u,U;Q(x_0,\varepsilon))&\leq \mathcal F(\widetilde v_\varepsilon,\widetilde V_\varepsilon;Q(x_0,\varepsilon))\nonumber\\
	&\leq \mathcal F(\widetilde u, \widetilde U;Q(x_0,\delta\varepsilon)) + C\Big(\mathcal L^N(Q(x_0,\varepsilon) \setminus Q(x_0,\delta\varepsilon))\nonumber \\ 
	& \hspace{2cm} + \int_{Q(x_0,\varepsilon)\setminus Q(x_0,\delta \varepsilon)} |\widetilde V_\varepsilon  |^p \,\de x 
	 + |E \widetilde v_\varepsilon|(Q(x_0,\varepsilon)\setminus Q(x_0,\delta \varepsilon))\Big)\nonumber \\
	&\leq \varepsilon^{N+1}+ m(\ell_{x_0}, U(x_0); Q(x_0,\delta\varepsilon)) + C\Big( \varepsilon^N (1-\delta^N) +  \int_{Q(x_0,\varepsilon)\setminus Q(x_0,\delta \varepsilon)} |\widetilde V_\varepsilon  |^p \,\de x\Big) \notag\\
	& \hspace{2cm} + C\Big(  |E u|(Q(x_0,\varepsilon)\setminus \overline{Q(x_0,\delta \varepsilon)})+ |{\rm tr}\,\widetilde u-{\rm tr}\, u|(\partial Q(x_0,\delta\varepsilon)) \Big).
\end{align}
We observe that we have
\begin{align}\label{mest2}
& \int_{Q(x_0,\varepsilon)\setminus Q(x_0,\delta \varepsilon)} |\widetilde V_\varepsilon  |^p \,\de x \leq 
\frac{1}{\varepsilon ^{N(p-1)}(1-\delta^N)^{p-1}}\left|\int_{Q(x_0,\varepsilon)} U(x) \,\de x - \int_{Q(x_0,\delta \varepsilon)} U(x_0) \,\de x\right|^p  \nonumber\\
&\hspace{0,3cm}\leq\frac{C}{\varepsilon ^{N(p-1)}(1-\delta^N)^{p-1}}\left( \left|\int_{Q(x_0,\varepsilon)\setminus Q(x_0,\delta \varepsilon)} U(x) \, \de x\right|^p+ \left|\int_{Q(x_0,\delta\varepsilon)}(U(x)-U(x_0)) \,\de x\right|^p\right) \nonumber \\
&\hspace{0,3cm}\leq\frac{C \varepsilon^{Np}}{\varepsilon ^{N(p-1)}(1-\delta^N)^{p-1}}\left( \left|\ave_{Q(x_0,\varepsilon)}\hspace{-0,1cm} U \,\de x- \delta^N \ave_{Q(x_0,\delta\varepsilon)}\hspace{-0,1cm}U \,\de x \right|^p+ \left|\delta^N\ave_{Q(x_0,\delta\varepsilon)}\hspace{-0,1cm}(U -U(x_0))\,\de x\right|^p\right). 
\end{align}
 Taking  into account \eqref{mest}, \eqref{mest2}, we have
\begin{align}
& \lim_{\varepsilon \to 0^+}\frac{m(u,U; Q (x_0, \varepsilon))}{\mathcal L^N (Q (x_0, \varepsilon))} - \lim_{\varepsilon \to 0^+}
	\frac{m(\ell_{x_0},U(x_0);Q (x_0, \varepsilon))}
	{\mathcal L^N (Q (x_0, \varepsilon))}\nonumber\\
&\leq \lim_{\varepsilon \to 0^+}\frac{m(u,U; Q (x_0, \varepsilon))}{\varepsilon^N} - \limsup_{\delta\to 1^-}\lim_{\varepsilon \to 0^+}
	\frac{m(\ell_{x_0},U(x_0);Q (x_0, \delta\varepsilon))}
	{\varepsilon^N}\nonumber\\
	&\hspace{1cm}\leq\limsup_{\delta \to 1^-}\limsup_{\varepsilon \to 0^+} \left(\varepsilon + C(1-\delta^N)+ \frac{|E u|(Q(x_0,\varepsilon)\setminus \overline{Q(x_0,\delta \varepsilon)}+ |{\rm tr}\,\widetilde u-{\rm tr}\, u|(\partial Q(x_0,\delta\varepsilon))}{\varepsilon^N}+ \right. \nonumber\\
	&\left.\hspace{4cm}\frac{C}{(1-\delta^N)^{p-1}} |U(x_0)-\delta^N U(x_0)| ^p\right),\label{mest3}
\end{align}
where in the last line we have used \eqref{ULeb}.  Property  \eqref{DuDsu} 
yields
\begin{align}\label{uapproxest}
\limsup_{\delta \to 1^-}\limsup_{\varepsilon \to 0^+}\frac{|E u|(Q(x_0,\varepsilon)\setminus \overline{Q(x_0,\delta \varepsilon)})}{\varepsilon^N}\leq \lim_{\delta \to 1^-}|\mathcal E u(x_0)|(1-\delta^N)=0.
\end{align}
On the other hand, setting $ u_{\varepsilon, \delta}(y):=\frac{u(x_0+\delta\varepsilon y)-u(x_0)}{\delta\varepsilon}$,   by \eqref{trvaest} and a change of variables, we conclude that 
\begin{align}
\limsup_{\varepsilon \to 0^+} \frac{|{\rm tr}\,\widetilde u-{\rm tr}\, u|(\partial Q(x_0,\delta\varepsilon))}{\varepsilon^N}&=\limsup_{\varepsilon \to 0^+} \delta^N\frac{|{\rm tr}\,\ell_{x_0}-{\rm tr}\, u|(\partial Q(x_0,\delta\varepsilon))}{\delta^N\varepsilon^N}\nonumber\\
&=\limsup_{\varepsilon \to 0^+}\delta^N\int_{\partial Q} | {\rm tr} ( u_{\varepsilon, \delta}-  \nabla  u(x_0)y)|\, \de\mathcal H^{N-1}(y) =0,\label{esttr}
\end{align} 
 where the last step follows from Proposition \ref{traceB} and the fact that $u_{\varepsilon, \delta}$ converges strictly to the function $y \mapsto \nabla u(x_0)y$ on $Q$.  To see the latter, we can follow  the lines of  \cite[Proof of Proposition 4.1]{BFT}. More precisely,  by Theorem \ref{Thm2.8BFT} and \cite[Theorem 2.10]{BFT}   there exists $v \in BD( Q)$ such that
$$ { \lim_{\varepsilon \to 0}   ||u_{\varepsilon, \delta} - P(u_{\varepsilon,\delta}) - v||_{L^1(Q)} = 0,}$$
where $P$ denotes the projection onto the kernel of the operator $E$.
Then, as $ E u_{\varepsilon, \delta}(y) =  E u(x_0 + \delta\varepsilon y)$,  using also \eqref{aproxdif} and \eqref{DuDsu} we get  $u_{\varepsilon, \delta} \to  \nabla u(x_0)y$ in $L^1(Q;\mathbb R^d)$ and 
$|Eu_{\varepsilon,\delta}|(Q)\to |\mathcal E u(x_0)|$ as $\varepsilon \to 0^+$,  i.e., strict convergence holds.

Taking into account \eqref{mest3}, \eqref{uapproxest}, and \eqref{esttr} we conclude that
$$\lim_{\varepsilon \to 0^+}\frac{m(u,U; Q (x_0, \varepsilon))}{\varepsilon^N} \leq \lim_{\varepsilon \to 0^+}
	\frac{m(\ell_{x_0},U(x_0);Q (x_0, \varepsilon))}
	{\varepsilon^N}.$$
Interchanging the roles of $(u,U)$ and $(\ell_{x_0},U(x_0))$, the reverse inequality is proved in a similar fashion. 
This  shows \eqref{toestimate} and completes the proof of \eqref{fproof}.

\emph{Step 2.}	In this step we prove  that,
 for $\mathcal H^{N-1}$-a.e.\ $x_0 \in J_u$,
\begin{align*}
\frac{{\rm d} \mathcal F(u, U; \cdot)}{{\rm d} \mathcal H^{N-1}\lfloor{J_u}}(x_0) = \Phi\big(x_0, u^+(x_0), u^-(x_0), \nu_u(x_0)\big).
\end{align*}
For simplicity of notation, we denote by $\nu$ the unit vector $\nu_u$  and set $v = v_{x_0,u^+(x_0),u^-(x_0),\nu(x_0)}$, see \eqref{stepfun}.  It is well known that, for $\mathcal H^{N-1}$-a.e.\ $x_0 \in J_u$, it holds that 
\begin{align}
& \lim_{\varepsilon \to 0^+}\ave_{Q_{\nu}(x_0,\varepsilon)}|u(x) - v(x)| \, \de x =0;\label{jumppt}\\
& \lim_{\varepsilon \to 0^+}\frac{1}{\varepsilon^{N-1}}|Eu|(Q_{\nu}(x_0,\varepsilon)) = |([u]\odot \nu)(x_0)|;\label{Du}\\
& \frac{{\rm d} \mathcal F(u, U;\cdot)}{{\rm d} \mathcal H^{N-1}\lfloor{J_u}} (x_0) =
\lim_{\varepsilon \to 0^+}\frac{\mathcal F(u, U;Q_{\nu}(x_0,\varepsilon))}{\varepsilon^{N-1}} =
\lim_{\varepsilon \to 0^+}\frac{m(u, U;Q_{\nu}(x_0,\varepsilon))}{\varepsilon^{N-1}};\label{jFmu}\\
& \lim_{\varepsilon \to 0^+}\frac{1}{\varepsilon^{N-1}}
\int_{Q_{\nu}(x_0,\varepsilon)}|U(x)|^p \, {\rm d}x=0.
\label{Uip}
\end{align}
where  Theorem \ref{thm4.3FHP} was used in \eqref{jFmu}.

Let $x_0$ be a fixed point in $\Omega$ satisfying the above properties  and let $\varepsilon >0$ be small enough so that $Q_{\nu}(x_0,\varepsilon)\subset \Omega$. Given the definition of the density $\Phi$ in \eqref{Phi}, due to \eqref{jFmu},
we want to show that
\begin{align}\label{mestjump}
\lim_{\varepsilon \to 0^+}
\frac{m(v,0 ;Q_\nu (x_0, \varepsilon))}
{\varepsilon^{N-1}} =  \lim_{\varepsilon \to 0^+}\frac{m(u,U; Q_\nu(x_0, \varepsilon))}{\varepsilon^{N-1}},
\end{align}
where $0$ is the null function from $\Omega$ to  $\mathbb R^{N\times N}$.

To this end, let $\delta >0$ and let $(\widetilde u, \widetilde U) \in \mathcal C_{StBD^p}(v, 0); Q_\nu(x_0, \delta \varepsilon))$  be such that
\begin{equation}\label{infmvj}
\varepsilon^{N} + m(v,0; Q_\nu(x_0, \delta \varepsilon)) \geq \mathcal F(\widetilde u, \widetilde U; Q_\nu(x_0, \delta \varepsilon)).
\end{equation}
Notice that, as $\widetilde u = v$ on $\partial Q_\nu(x_0, \delta \varepsilon)$, we have
\begin{equation}\label{trvj}
|{\rm tr}\, u-{\rm tr}\,\widetilde u|(\partial Q_\nu(x_0,\delta \varepsilon))
:= \int_{\partial Q_\nu(x_0,\delta\varepsilon)} 
|{\rm tr}(\widetilde u - u )| \, \de \mathcal H^{N-1} 
=\int_{\partial Q_\nu(x_0,\delta\varepsilon)} 
|{\rm tr}(v- u )| \, \de \mathcal H^{N-1}.
\end{equation}
Define
$$
\widetilde v_\varepsilon:=\left\{
\begin{array}{ll}
\widetilde u &\hbox{ in }Q_\nu(x_0,\delta \varepsilon),\\
u &\hbox{ in } \Omega \setminus Q_\nu(x_0,\delta \varepsilon),
\end{array}
\right.
$$
and let
$$\widetilde V_\varepsilon(x):= \left\{
\begin{array}{ll}
\widetilde U(x) & \hbox{ in }Q_\nu(x_0,\delta \varepsilon),\\
\dfrac{1}{\mathcal L^N(Q_\nu(x_0,\varepsilon)\setminus Q_\nu(x_0,\delta \varepsilon))}\displaystyle \int_{Q_\nu(x_0,\varepsilon)} U(x) \, \de x &\hbox{ in } \Omega \setminus Q_\nu(x_0,\delta \varepsilon).
\end{array}\right.
$$
Recall that we have 
$ \displaystyle \int_{Q_\nu(x_0,\delta \varepsilon)} \widetilde U(x) \, \de x=0$.
Thus, 
$(\widetilde v_\varepsilon, \widetilde V_{\varepsilon})$ belongs to the class of admissible  functions 
 $ C_{StBD^p}(u, U; Q_\nu(x_0,\varepsilon))$ \color{black}, and therefore
we obtain, using also Remark \ref{A1A2}, 
\begin{align}
m(u,U;Q_\nu(x_0,\varepsilon))& \leq  
\mathcal F(\widetilde v_\varepsilon,\widetilde V_\varepsilon)\nonumber\\
&\leq \mathcal F(\widetilde u, \widetilde U;Q_\nu(x_0,\delta\varepsilon)) + C\Big(\mathcal L^N(Q_\nu(x_0,\varepsilon) \setminus Q_\nu(x_0,\delta\varepsilon))\nonumber \\ 
& \hspace{2cm} +  \int_{Q_\nu(x_0,\varepsilon)\setminus Q_\nu(x_0,\delta \varepsilon)} |\widetilde V_\varepsilon  |^p \,\de x 
+ |E \widetilde v_\varepsilon|(Q_\nu(x_0,\varepsilon)\setminus Q_\nu(x_0,\delta \varepsilon))\Big)\nonumber \\
&\leq \varepsilon^{N}+ m(v,0; Q_\nu(x_0,\delta\varepsilon)) + C\Big( \varepsilon^N (1-\delta^N) + \int_{Q_\nu(x_0,\varepsilon)\setminus Q_\nu(x_0,\delta \varepsilon)} |\widetilde V_\varepsilon  |^p \,\de x \Big)\notag\\ 
& \hspace{0.8cm}+ C\Big(|E u|(Q_\nu(x_0,\varepsilon)\setminus \overline{Q_\nu(x_0,\delta \varepsilon)})+ |{\rm tr}\,\widetilde u-{\rm tr}\, u|(\partial Q_\nu(x_0,\delta\varepsilon)) \Big). \label{mestjump2}
\end{align}
We have, using  H\"older's inequality,
\begin{align}
\int_{Q_\nu(x_0,\varepsilon)\setminus Q_\nu(x_0,\delta \varepsilon)} 
|\widetilde V_\varepsilon  |^p \, \de x &\leq 
\frac{1}{\varepsilon ^{N(p-1)}(1-\delta^N)^{p-1}}\left|\int_{Q_\nu(x_0,\varepsilon)} U \, \de x \right|^p  
&\leq \frac{\varepsilon^{N(p-1)}}{\varepsilon ^{N(p-1)}(1-\delta^N)^{p-1}} \|U\|^p_{L^p(Q_\nu(x_0,\varepsilon))} \nonumber\\
&= \frac{1}{(1-\delta^N)^{p-1}} 
\|U\|^p_{L^p(Q_\nu(x_0,\varepsilon))}. \label{mest4}
\end{align}
Hence, from \eqref{mestjump2} and \eqref{mest4}, taking into account \eqref{Uip} and Proposition \ref{traceB},  it follows that
\begin{align}
\lim_{\varepsilon \to 0^+}\frac{m(u,U; Q_\nu(x_0, \varepsilon))}{\varepsilon^{N-1}} &\leq 
\limsup_{\delta \to 1^-}\limsup_{\varepsilon \to 0^+} 
\Big(\varepsilon +\frac{m(v,0 ;Q_\nu (x_0,\delta\varepsilon))}
{\varepsilon^{N-1}}\nonumber\\
&\hspace{1cm} + C\Big( \varepsilon(1-\delta^N) + \frac{1}{(1-\delta^N)^{p-1}}
  \frac{1}{\eps^{N-1}} \|U\|^p_{L^p(Q_\nu(x_0,\varepsilon))} \nonumber\\
&\hspace{1cm} +\frac{|E u|(Q_\nu(x_0,\varepsilon)\setminus \overline{Q_\nu(x_0,\delta \varepsilon)})}{\varepsilon^{N-1}} + \frac{|{\rm tr}\,\tilde u-{\rm tr} \, u|(\partial Q_\nu(x_0,\delta\varepsilon))}{\varepsilon^{N-1}}\Big)\Big)\nonumber\\
&\leq \lim_{\varepsilon \to 0^+}\frac{m(v,0; Q_\nu(x_0, \varepsilon))}{\varepsilon^{N-1}} + \limsup_{\delta \to 1^-}(1-\delta^N) |[u](x_0)|\nonumber\\
&= \lim_{\varepsilon \to 0^+}\frac{m(v,0 ; Q_\nu(x_0, \varepsilon))}{\varepsilon^{N-1}}\label{ineq}
\end{align}
since, by \cite[(2.32) and (5.79)]{AFP}   and \eqref{Du},
\begin{align*}
\lim_{\varepsilon \to 0^+}\frac{|Eu|(Q_\nu(x_0,\varepsilon)\setminus \overline{Q_\nu(x_0,\delta\varepsilon)})}{\varepsilon^{N-1}}\leq (1-\delta^N)|[u](x_0)|
\end{align*}
and 
\begin{equation}\label{traces}
\lim_{\varepsilon \to 0^+}\frac{|{\rm tr}\,\tilde u-{\rm tr
} \, u|(\partial Q_\nu(x_0,\delta\varepsilon))}{\varepsilon^{N-1}}=0.
\end{equation}
To prove this last fact, we change variables and use \eqref{trvj} to obtain
\begin{align*}
\lim_{\varepsilon \to 0^+}\frac{|{\rm tr}\,\tilde u-{\rm tr}\, u|(\partial Q_\nu(x_0,\delta\varepsilon))}{\varepsilon^{N-1}}
&=\lim_{\varepsilon \to 0^+}\delta^{N-1}\int_{\partial Q_\nu}
|{\rm tr}(v(x_0+\delta \eps y) - u(x_0+\delta \eps y))| \, \de {\mathcal H}^{N-1}(y)\\
&= \lim_{\varepsilon \to 0^+}\delta^{N-1}\int_{\partial Q_\nu}
|{\rm tr}(v_{x_0,u^+(x_0),u^-(x_0),\nu(x_0)} (x_0 + y)  -u_{\delta, \varepsilon}(y))| 
\, \de {\mathcal H}^{N-1}(y).
\end{align*}
where
$u_{\delta,\varepsilon}(y):= u(x_0+\delta \varepsilon y)$. Then \eqref{jumppt} and \eqref{Du} yield 
$$u_{ \delta,\varepsilon}\to v_{x_0,u^+(x_0),u^-(x_0),\nu(x_0)} (x_0 + \cdot) \mbox{ in } L^1(Q_\nu;\mathbb R^N  ) \mbox{ as } \eps \to 0^+$$
and
$${|E u_{\delta,\varepsilon}|(Q_\nu)=\frac{1}{(\delta\varepsilon)^{N-1}}
|Eu|(Q_\nu(x_0,\delta \varepsilon))\to |[u]\odot \nu|(x_0)=|Ev_{x_0,u^+(x_0),u^-(x_0),\nu(x_0)}|( x_0 + Q_\nu)
\mbox{ as } \eps \to 0^+.}$$
Hence \eqref{traces} follows from Proposition \ref{traceB} and this completes the proof of inequality \eqref{ineq}. The reverse inequality can be shown in a similar way by interchanging the roles of $(u,U)$ and $(v,0)$
leading to the conclusion stated in \eqref{mestjump}.
\end{proof}
 
\begin{proof}[Proof of Theorem \ref{GMthmHSDsbd}]
The result is an immediate consequence of Proposition \ref{GMthmHSD}. \end{proof}

\begin{proof}[Proof of Corollary \ref{remStSBDrep}]
 Under the  assumptions of the corollary,  Remark \ref{traslinv}, Lemmata \ref{estHSDL} and \ref{Lemma 4.2FHP} and Theorem \ref{thm4.3FHP} hold with obvious modifications in the proof, replacing $m$ by $\tilde m$ and $\mathcal C_{StBD^p}$ by $\mathcal C_{StSBD^p}$ and defining the associated objects $\tilde m^\delta$, $\tilde m^\star$ in full analogy with $m^\delta$ and  $m^\star$.  The arguments in the proof of Proposition \ref{GMthmHSD} remain unchanged. 
\end{proof}

\section{Proof of the relaxation results}\label{relax}

We start this section with the proof of Theorem \ref{representation}. 
\begin{proof}[Proof of Theorem \ref{representation}]
	
Given $O \in \mathcal O(\Omega)$ and $(g,G) \in StBD^p(\Omega)$, we introduce the localized version of 
$I_p(g, G)$, namely
\begin{align}\label{Iploc}
{I_p(g,G;O)\coloneqq \inf\Big\{\liminf_{n\to\infty}  F (u_n)\colon \, u_n \in SBD(O),  \ u_n\wSD{*}(g,G) \hbox{ in }O\Big\}.}
\end{align}
Our goal is to verify that $I_p(g,G;O)$ satisfies assumptions {\rm (H1)}--{\rm (H4)} of 
Theorem \ref{GMthmHSDsbd}.

 \emph{Step 1: Proof of {\rm (H1)}.}  We start by proving the following nested subadditivity result: if $O_1, O_2, O_3$ are open subsets of $\Omega$ such that 
$O_1 \Subset O_2 \subseteq O_3$,  then
\begin{equation}\label{nestsub}I_p(g,G;O_3) \leq I_p(g,G;O_2) + I_p(g,G;O_3\setminus \overline {O_1}).
\end{equation}
 Without restriction, this is proved only for $p=1$ as the case $p>1$ is even easier. Let $\lbrace u_n \rbrace \subset SBD(O_2)$ and 
$\lbrace v_n  \rbrace \subset SBD(O_3\setminus \overline{O_1})$ be two sequences such that $u_n \rightarrow g$ in $L^1(O_2;{\mathbb{R}}^N)$, 
$\mathcal E u_n \overset{\ast}{\rightharpoonup} G$ in  $\mathcal M(O_2;\mathbb R^{N\times N}_{\rm sym}) $,  
$v_n  \overset{\ast}{\rightharpoonup}  g$ in $L^1(O_3\setminus \overline{O_1};\mathbb R^N)$, 
$\mathcal Ev_n  \overset{\ast}{\rightharpoonup}  G$ in 
 $\mathcal M(O_3\setminus \overline{O_1};\mathbb R^{N\times N}_{\rm sym}) $,  and 
\begin{equation*}
I_p(g,G;O_2) = \lim_{n\to +\infty}\left[\int_{O_2}W(x, \mathcal E u_n(x)) \, \de x + \int_{J_{u_n}\cap O_2}
\psi(x,[u_n](x),\nu_{u_n}(x)) \, \de\mathcal{H}^{N-1}(x)\right] 
\end{equation*}
as well as
\begin{equation*}
I_p(g,G;O_3\setminus \overline{O_1}) = \lim_{n\to +\infty} 
\left[\int_{O_3\setminus \overline{O_1}}W(x, \mathcal E v_n(x)) \, \de x + \int_{J_{v_n}\cap(O_3\setminus \overline{O_1})} \psi(x,[v_n](x),\nu_{v_n}(x))\, \de\mathcal{H}^{N-1}(x)\right].
\end{equation*}
\noindent Notice that 
\begin{equation}  \label{L1conv}
u_n - v_n \rightarrow 0 \; \text{ in }\; 
L^1( O_2 \setminus \overline{O_1}; {\mathbb{R}}^N)
\end{equation}
For $\delta > 0$, we  define $O_{\delta} := \{ x \in O_2: \,\, \mbox{dist}(x, O_1) < \delta\}$, and for $x \in O_2 $ we let $d(x):= \mbox{dist}(x, O_1)$. Since the distance
function to a fixed set is Lipschitz continuous (see \cite[Exercise 1.1]{Z}), we can apply the change of variables formula \cite[Section 3.4.3, Theorem 2]{EG}, to obtain 
\begin{equation*}
\int_{O_{\delta}\setminus \overline{O_1}} |u_n(x) - v_n(x)| \,  Jd(x)  \, \de x =
\int_{0}^{\delta}\left [ \int_{d^{-1}(y)} |u_n(x) - v_n(x)| \, 
\de \mathcal{H}^{N-1}(x)\right]\, \de y,
\end{equation*}
 where $Jd$ denotes the Jacobian   of  $d$. As $Jd$ is bounded and \eqref{L1conv} holds,
 by Fatou's Lemma, it
follows that for almost every $\rho \in [0, \delta]$ we have 
\begin{equation}  \label{aer}
\liminf_{n\rightarrow +\infty} 
\int_{d^{-1}(\rho)} |u_n(x) - v_n(x) |\, \de\mathcal{H}^{N-1}(x) 
= \liminf_{n\rightarrow +\infty} \int_{\partial O_{\rho}}
|u_n(x) - v_n(x) |\, \de\mathcal{H}^{N-1}(x) = 0.
\end{equation}
Fix $\rho_0 \in [0, \delta]$ such that 
$|G \chi_{O_2}|(\partial O_{\rho_0}) = 0$, 
$|G \chi_{O_3 \setminus \overline{O_1}}|(\partial O_{\rho_0}) = 0$  (where we consider the total variation of the measures $G\chi_{\cdot}\lfloor \mathcal L^N$),  and
such that \eqref{aer} holds. For this choice of $\rho_0$, we may pass to subsequences of  $u_n$ and $v_n$ (not relabeled) such that the liminf in \eqref{aer} is actually a limit. We observe that $O_{\rho_0}$ is a set with
locally Lipschitz boundary since it is a level set of a Lipschitz function
(see, e.g., \cite{EG}). Hence, we can consider $u_n, v_n$ 
on $\partial O_{ \rho_0}$ in the sense of traces and define 
\begin{align}\label{wn}
w_n = 
\begin{cases}
u_n & \text{ in}\; \overline{O}_{\rho_0} \\ 
v_n & \text{ in}\; O_3\setminus \overline{O}_{\rho_0}.
\end{cases}
\end{align}
\noindent Then, by the choice of $\rho_0$, $w_n$ is admissible for 
$I_p(g,G;O_3)$ so, by {\rm \ref{(psi1)}}, 
and \eqref{aer}, we obtain 
\begin{align*}
I_p(g,G;O_3) &\leq  \liminf_{n\to +\infty} 
\left[\int_{O_3}W(x, \mathcal E w_n(x)) \, \de x +
\int_{J_{w_{n}}\cap O_3} \psi(x,[w_n](x),\nu_{w_n}(x))\, \de\cH^{N-1}(x)\right] \\
&\leq  \liminf_{n\to +\infty} 
\left[\int_{O_2}W(x, \mathcal E u_n(x)) \, \de x +
\int_{J_{u_{n}}\cap O_2} \psi(x,[u_n](x),\nu_{u_n}(x))\, \de\cH^{N-1}(x) \right. \\
&  \hspace{1cm} + \int_{O_3\setminus \overline{O_1}}W(x, \mathcal  E v_n(x)) \, \de x  + \int_{J_{v_{n}}\cap (O_3 \setminus \overline{O_1})}
\psi(x,[v_n](x),\nu_{v_n}(x))\, \de\cH^{N-1}(x) \\
&  \left. \hspace{1cm} + \int_{J_{w_n} \cap \partial O_{\rho_0}} 
C |u_n(x) - v_n(x)| \, \de\mathcal{H}^{N-1}(x) \right] \\
& =  I_p(g,G;O_2) + I_p(g,G;O_3 \setminus \overline{O_1}).
\end{align*}
 This concludes the proof of \eqref{nestsub}.

From here, the reasoning in 
\cite[Proposition 2.22]{CF1997}  yields (H1).	 In fact,  the proof also works in the nonhomogeneous case  and in the $BD$-setting since it is exclusively  based on measure theory arguments.

 \emph{Step 2: Proof of {\rm (H2)}--{\rm (H4)}.} 		
To show {\rm  (H2)}, we argue exactly as in \cite[Proposition 5.1]{CF1997}. Indeed, we can prove lower semicontinuity of $I_p(\cdot,\cdot;O)$ along sequences $(g_n,G_n)$ converging in 
$L^1(\Omega;\mathbb R^N)_{\it strong}\times L^p(\Omega;  \mathbb R^{N\times N}_{\rm sym})_{\it weak}$ 
(the second convergence is weakly* in $\cM(\Omega;  \mathbb R^{N\times N}_{\rm sym})$, if $p=1$)   to  $(g, G) \in StBD^p(\Omega)$. (H3) is an immediate consequence of the previous lower semicontinuity property in $O$, as observed in \cite[(2.2)]{BFM}  whereas
(H4) follows by employing  assumptions {\rm \ref{(W1)_p}}, {\rm \ref{W3}}, {\rm \ref{W4}}, and {\rm \ref{(psi1)}}  together with  the lower semicontinuity of integral functionals of power type (or of the total variation along weakly*  converging sequences, if $p=1$),  see   \cite[Lemma 2.18]{CF1997} for details.   (We point out that, to obtain the lower bound in (H4), we   actually need to replace  $W$ by $W + \frac{1}{C_W}$ which however can be done without restriction.)

 \emph{Step 3: Proof of (i).}  Having checked  {\rm (H1)}--{\rm (H4)},  Theorem \ref{GMthmHSDsbd}  can be applied to conclude that,
for every $(g,G)\in StSBD^p(\Omega)$, 
we have
$$I_p(g,G) =  \int_\Omega f\big(x, g(x),  \nabla  g(x), G(x)\big) \, \de x + 
\int_{\Omega\cap J_g}\Phi\big(x, g^+(x), g^-(x),\nu_g(x)\big) \,\de \mathcal H^{N-1}(x),$$
where the relaxed densities $f$ and $\Phi$ are given by  \eqref{f} and \eqref{Phi}, respectively. 

  It is a standard matter to check  that the functional $I_p$ is invariant under translation in the first variable, i.e., 
$$I_p(g+a,G;O)= I_p(g,G;O) \ \text{ for all }  (g,G) \in  StBD^p(\Omega), \  O \in \mathcal O(\Omega), \ a \in \mathbb R^N.$$ 
Indeed, it suffices to notice that, if $\{u_n\}$ is admissible in the definition of  $I_p(g,G;O)$, then the sequence $\lbrace u_n + a\rbrace$ is admissible for $I_p(g+a,G;O)$. Hence, taking into account Remark \ref{traslinv}, and the abuse of notation stated therein, we obtain \eqref{reprelax} with $f$ and $\Phi$ given by 
\eqref{fdef} and \eqref{Phidef}, respectively. In a similar fashion,  the definition \eqref{102} and the fact that $W$ depends only on the symmetrized gradient ensures that $I_p(u+ Mx)= I_p(u)$ for every $M \in \mathbb R^{N \times N}_{\rm skew}$, hence (iii) of Remark \ref{traslinv} applies and therefore $f(x, g(x),\nabla g(x), G(x))
=f(x,   \mathcal E  g(x), G(x))$. 
This shows (i).

 \emph{Step 4: Proof of (ii).}   If $(g, G) \in StBD^p(\Omega)$, then we can invoke  Theorem \ref{GMthmHSDsbd}    to obtain 
\begin{align*}
 f(x_0,   \mathcal E g(x_0), G(x_0)) =\frac{{\rm d} I_p(g,G)}{{\rm d} \mathcal L^N}(x_0) \hbox{ for a.e.\ $x_0 \in \Omega$}\\
\Phi(x_0, [g](x_0),\nu_g(x_0))= \frac{{\rm d} I_p (g, G)}{{\rm d} \mathcal H^{N-1}\lfloor J_g} (x_0)  \hbox{ for }\mathcal H^{N-1}\hbox{-a.e.\ $x_0 \in \Omega$.} 
\end{align*}

\emph{Step 5: Proof of  \eqref{eq. effiproof} in  (iii).}   We claim that for a.e.\ $x_0\in\Omega$ and for every $\xi,B\in\R{N\times N}_{\rm sym} $ it holds that
\begin{equation}\label{clean}
f(x_0,\xi,B)=H_{p}(x_0,\xi,B)=\widetilde{H}_{p}(x_0,\xi,B),
\end{equation}
where 
\begin{align*}
\widetilde{H}_{p}(x_0,\xi,B) \coloneqq& 
\limsup_{\varepsilon \to 0^+} \inf\Big\{\liminf_{n\to \infty}
\int_Q W(x_0+ \varepsilon y,\mathcal E  u_n  (y))\, \de y  +\int_{Q\cap  J_{u_n} } \psi(x_0,[u_n](y),\nu_{u_n}(y))\,\de\cH^{N-1}(y)\colon \notag\\
&\{u_n\} \subset SBD(Q), \  u_n \to \ell_{0,0,\xi}  \hbox{ in } L^1(Q;\mathbb R^N), \
\mathcal E u_n \wsto B \mbox{ in } \cM(Q;\mathbb R^{N\times N}_{\rm sym})\Big\}.
\end{align*}
Notice that we  have readily replaced $ \ell_{0,u(x_0),\xi} $ by $ \ell_{0,0,\xi}$  due to the  translational invariance.
 The proof that $H_p(x_0,\xi,B) \leq \widetilde{H}_{p}(x_0,\xi,B)$ can be obtained as in \cite[Step 2 in Proposition 3.1]{CF1997}, in turn relying on a technique to pass from a given sequence to another admissible one with fixed boundary datum. 
However, instead of modifying the functions as in  \cite[Proposition 2.21]{CF1997}, we proceed as in the  analogous result in our setting (namely \eqref{nestsub}) to define  modifications similarly as in    \eqref{wn} which allows to   obtain  a sequence with precise boundary data from a generic given one.  Let us highlight that one first constructs a new admissible sequence $\lbrace w_n^\eps \rbrace $ for $\eps$ fixed,  and eventually one passes to the limit $\varepsilon \to 0$.  Moreover, careful inspection of the proofs of \eqref{nestsub} and \cite[Step 2 of Proposition 3.1]{CF1997} shows  that the $x$-dependence is not an issue and that the admissible  sequences for $\tilde H_p$ can be replaced by sequences such that $\int_Q \mathcal E \BBB w^\varepsilon_n \EEE \,\de x=B$ holds, namely, by functions  belonging to $\cC^\bulk_p(\xi,B)$.
\EEE


We now address the reverse inequality.  By   the definition of the  relaxed bulk energy density $f$ in \eqref{fdef} and Theorem \ref{thm4.3FHP}
 we get
 \begin{equation}\label{fHtilde}
f(x_0,\xi,B)= \limsup_{\varepsilon \to 0} \frac{m(  \ell_{x_0, 0,\xi},  B;Q(x_0,\varepsilon))}{\varepsilon^N} = \limsup_{\varepsilon \to 0}  \frac{I_{p}(\ell_{x_0, 0,\xi}, B;Q(x_0,\varepsilon))}{\varepsilon^N}. 
\end{equation}
Then, by a simple change of variables argument, using the functions $u^\eps_n(y) = \frac{1}{\eps} u_n(x_0 + \eps y)$ for a given sequence $\lbrace u_n\rbrace$ on $Q(x_0,\varepsilon)$ in the definition of $I_p$,   invoking property {\rm \ref{(psi2)}}, it follows that, for a.e.~$x_0\in \Omega$ and every $\xi,B \in \mathbb R^{N \times N}_{\rm sym}$,  
\begin{equation}\label{Htilde}
\limsup_{\varepsilon \to 0} \frac{I_{p}(  \ell_{x_0, 0,\xi},  B;Q(x_0,\varepsilon))}{\varepsilon^N} = \widetilde{H}_{p}(x_0,\xi,B).
\end{equation}
Now define $\widehat{m}( \ell_{x_0, 0,\xi}, B;Q(x_0,\varepsilon))\coloneqq \inf\big\{E(u;Q(x_0,\varepsilon)): \varepsilon^{-1}u(x_0 + \varepsilon \cdot) \in \mathcal C^\bulk_p(\xi,B)\big\}$.
It is easy to verify that for a.e.~$x_0\in \Omega$ and every $\xi,B \in \mathbb R^{N\times N}_{\rm sym}$ it holds that 
$${m( \ell_{x_0, 0,\xi}, B;Q(x_0,\varepsilon))\leq \widehat{m}( \ell_{x_0, 0,\xi},  B;Q(x_0,\varepsilon)).}
$$
Indeed, for any function $u$ that is admissible for $\widehat{m}( \ell_{x_0, 0,\xi}, B;Q(x_0,\varepsilon))$, 
the pair $(u,\mathcal E u)$ is a competitor for $m( \ell_{x_0, 0,\xi},  B;Q(x_0,\varepsilon))$.
 We now conclude as follows:  by  \eqref{Htilde} and \eqref{fHtilde} we find 
\begin{align*}
\widetilde H_p( x_0, \xi,  B) &= \limsup_{\varepsilon \to 0}  \frac{I_{p}(  \ell_{x_0, 0,\xi},B;Q(x_0,\varepsilon))}{\varepsilon^N} =\limsup_{\varepsilon \to 0} \frac{m( \ell_{x_0, 0,\xi},  B;Q(x_0,\varepsilon))}{\varepsilon^N}\\
&\leq \limsup_{\varepsilon \to 0} \frac{\widehat{m}(\ell_{x_0, 0,\xi},  B;Q(x_0,\varepsilon))}{\varepsilon^N}=H_{p}( x_0, \xi,  B),
\end{align*}
 where the last equality follows by the definition in  \eqref{906}, the continuity property in {\rm \ref{(psi4)}}, and again a change of variables argument. This along with the inequality $H_p(x_0,\xi,B) \leq \widetilde{H}_{p}(x_0,\xi,B)$ shown above and \eqref{fHtilde} yields \eqref{clean}.

 \emph{Step 6: Proof of  \eqref{eq. psiiproof} in  (iii)  for $p>1$.} Theorem \ref{thm4.3FHP}  applied to  $\mathcal F= I_p$ and   yield, 
for every $x_0\in \Omega$, $\lambda \in \mathbb R^{N}$, and
$\nu \in \mathbb S^{N-1}$, 
\begin{align*}
\Phi(x_0, \lambda,  \nu) &=  
\limsup_{\varepsilon \to 0^+}\frac{m(v_{x_0,\lambda,0,\nu}, 0; Q_{\nu}(x_0,\varepsilon))}{\varepsilon ^{N-1}} \\
&=\limsup_{\varepsilon \to 0^+}\frac{I_p(v_{x_0,\lambda,0,\nu}, 0; Q_{\nu}(x_0,\varepsilon))}{\varepsilon ^{N-1}}\\
&= \limsup_{\varepsilon \to 0^+}\frac{1}{\varepsilon^{N-1}}
\inf\Big\{\liminf_{n\to+\infty}\Big[\int_{Q_\nu(x_0,\varepsilon)} 
\hspace{-0,85cm}W(x,\mathcal E u_n(x)) \, \de x + \int_{Q_\nu(x_0,\varepsilon)\cap J_{u_n}}
\hspace{-0,81cm}\psi(x, [u_n](x),\nu_{u_n}(x)) \,\de\mathcal H^{N-1}(x)\Big] \colon\\
&\hspace{0.6cm} u_n \in SBD(Q_\nu(x_0,\varepsilon)), 
u_n \to v_{x_0,\lambda,0,\nu}\mbox{ in }  L^1(Q_\nu(x_0,\varepsilon);\mathbb R^N),  \  \mathcal E u_n \rightharpoonup 0 
\mbox{ in } L^p(Q_\nu(x_0,\varepsilon);\mathbb R^{N\times N}_{\rm sym}) \Big\} \\
&\leq \limsup_{\varepsilon \to 0^+}\frac{1}{\varepsilon^{N-1}}
\inf\Big\{\liminf_{n\to+\infty}\int_{Q_\nu(x_0,\varepsilon)\cap J_{u_n}}
\hspace{-0,78cm}\psi(x, [u_n](x),\nu_{u_n}(x)) \,\de\mathcal H^{N-1}(x) \colon \\
 &\hspace{0.6cm}  u_n \in SBD(Q_\nu(x_0,\varepsilon)), 
u_n \to v_{x_0,\lambda,0,\nu}  \mbox{ in }  L^1(Q_\nu(x_0,\varepsilon);\mathbb R^N),\ \mathcal E u_n  = 0 
\hbox{ a.e.\ in } Q_\nu(x_0,\varepsilon)\Big\},
\end{align*}
where  we have taken into account the growth condition on $W$ given by {\rm \ref{(W1)_p}} and hypothesis  {\rm \ref{W3}}, and the fact that the latter class of  competitors    is contained in the first one.

Given that this last expression no longer depends on the initial bulk density $W$, but only on $\psi$ for which the uniform continuity condition {\rm \ref{(psi4)}} holds, we may apply this condition to replace $x$ by $x_0$ and obtain
\begin{align*}
\Phi(x_0, \lambda, \nu) &\leq \limsup_{\varepsilon \to 0^+}\frac{1}{\varepsilon^{N-1}}
\inf\Big\{\liminf_{n\to+\infty}\int_{Q_\nu(x_0,\varepsilon)\cap J_{u_n}}
\hspace{-0,78cm}\psi(x_0, [u_n](x),\nu_{u_n}(x)) \,\de\mathcal H^{N-1}(x) \colon \\
 &\hspace{0.6cm}  u_n \in SBD(Q_\nu(x_0,\varepsilon)), 
u_n \to  v_{x_0,\lambda,0,\nu} \mbox{ in }  L^1(Q_\nu(x_0,\varepsilon);\mathbb R^N),\ \mathcal E u_n  = 0 
\hbox{ a.e.\ in } Q_\nu(x_0,\varepsilon)\Big\}.
\end{align*}
We can  invoke a periodicity argument entirely similar to the one used in the first part of the proof of \cite[Proposition 4.2]{CF1997}, namely  we define 
\begin{align*}
u:=  v_{x_0,\lambda,0,\nu}+ \phi
\end{align*}
for $\phi \in SBV(Q_{\nu}(x_0,\varepsilon);\mathbb{R}^N)$ 
with $\phi=0$ on $\partial Q_{\nu}(x_0,\varepsilon)$ and $\nabla \phi=0$ $\mathcal L^N$-a.e.\ in $Q_\nu(x_0,\varepsilon)$. Extending $\phi$ periodically to all of $\mathbb R^N$ with period
 $\varepsilon$ and setting $u_n :=   v_{x_0,\lambda,0,\nu}+ \frac{\phi(n (x-x_0))}{n}$, we easily see that $u_n \to  v_{x_0,\lambda,0,\nu}$ in $L^1(Q_\nu(x_0,\varepsilon)$, $\nabla u_n = \mathcal E u_n =0$ $\mathcal L^N$-a.e.,
and $$
\int_{Q_{\nu}(x_0,\varepsilon) \cap J_{u_n}} \psi(x_0,[u_n](x), \nu_{u_n}(x))  \, \de \mathcal H^{N-1}(x) \to \int_{Q_{\nu}(x_0,\varepsilon) \cap J_{u}} \psi(x_0, [u](x), \nu_{u}(x)) \, \de \mathcal H^{N-1}(x)$$
 as $n \to +\infty$.
\color{black}
Thus, we conclude that \begin{align*}
\Phi(x_0,\lambda, \nu) & 
\leq\limsup_{\varepsilon \to 0^+}\frac{1}{\varepsilon^{N-1}}
\inf\Big\{\int_{Q_\nu(x_0,\varepsilon)\cap J_v}
\psi(x_0, [v](x),\nu_v(x)) \,\de\mathcal H^{N-1}(x): 
v \in SBD(Q_\nu(x_0,\varepsilon)),\\
&\hspace{4,2cm}
\mathcal E v(x) = 0 \hbox{ a.e.\ in }  Q_\nu(x_0,\varepsilon),  \ 
v(x) =  v_{x_0,\lambda,0,\nu}(x) \text{ for } x \in \partial Q_\nu(x_0,\varepsilon)  \Big\}.
\end{align*}
By a simple change of variables,  this coincides with
\begin{align*}
&\inf\Big\{\int_{Q_\nu\cap J_v} \hspace{-0.2cm}
\psi(x_0, [u](y),\nu_u(y)) \,\de\mathcal H^{N-1}(y): 
u \in SBD(Q_\nu), \ 
\mathcal E u(x)= 0 \hbox{ a.e.\ in } Q_\nu,  
u(x) =  v_{\lambda,\nu}(x) \text{ for } x \in \partial Q_\nu  \Big\},
\end{align*}
where $v_{\lambda,\nu}$ is given in \eqref{909}.  The latter formula coincides with  $h_p(x_0, \lambda, \nu)$, see \eqref{907}, and thus  it follows that
$$\Phi(x_0,\lambda, \nu) \leq  h_p(x_0, \lambda, \nu).$$

To prove the reverse inequality, we argue as in \cite[Theorems 4.4 and 4.5]
{CF1997}. First, Theorem \ref{thm4.3FHP} guarantees 
\begin{align}\label{y1}
	\Phi(x_0,\lambda,\nu) 
	&=\limsup_{\varepsilon \to 0^+}\frac{I_p( v_{x_0,\lambda,0,\nu},  0; Q_{\nu}(x_0,\varepsilon))}{\varepsilon ^{N-1}}.
    \end{align}
   In particular, there exists $\{u_n\}  \subset SBD(Q_\nu(x_0,\varepsilon))  $ such that 
\begin{align*}
    u_n \to  v_{x_0,\lambda,0,\nu}   \mbox{ in }  L^1(Q_\nu(x_0,\varepsilon);\mathbb R^N), \quad 
		\mathcal E u_n \rightharpoonup 0\mbox{ in } L^p(\Omega;\mathbb R^{N\times N}_{\rm sym}) ,
\end{align*}
and, since $W\geq 0$ (up to adding $\frac{1}{c_W}$),  we deduce
 \begin{align}\label{y2}
 I_p( v_{x_0,\lambda,0,\nu},  0;\Omega)\color{black} &\geq \liminf_{n\to +\infty} \left(\int_\Omega W(x,\mathcal E u_n)\,{\rm d}x+ \int_{\Omega \cap J_{u_n}} \psi(x, [u_n],\nu_{u_n}) \, {\rm d} \mathcal H^{N-1}\right)\notag\\
 &\geq \liminf_{n\to \infty}  \int_{\Omega \cap J_{u_n}} \psi(x, [u_n],\nu_{u_n}) \, {\rm d} \mathcal H^{N-1}.
 \end{align}
 \color{black}
 Up to the extraction of a subsequence, there exists a nonnegative Radon measure $\mu$ such that
 \begin{align*}
 \psi(x, [u_n],\nu_{u_n}) \, {\rm d} \mathcal H^{N-1}\lfloor J_{u_n} \overset{\ast}{\rightharpoonup} \mu.
 \end{align*}
Next, we evaluate  $\frac{{\rm d} \mu}{{\rm d} \mathcal H^{N-1}\lfloor J_{v_{x_0,\lambda,0,\nu}}}$. 
We choose a family $\{\varepsilon\}$ such that $\mu(\partial Q(x_0,\varepsilon))=0$. Thus, setting $u_{n,\varepsilon}(y) := u_n(x_0+\varepsilon y) $, 
\begin{align*}
\frac{{\rm d} \mu}{{\rm d}\mathcal H^{N-1}\lfloor J_{v_{x_0,\lambda,0,\nu}}} &= \lim_{\varepsilon \to 0^+}\frac{1}{ \varepsilon^{N-1}}\lim_{n\to \infty}  \int_{Q_{\nu}(x_0,\varepsilon) \cap J_{u_n}} \psi(x, [u_n](x),\nu_{u_n}(x)) \, {\rm d} \mathcal H^{N-1}(x)\notag \\
&=\lim_{\varepsilon \to 0^+}\lim_{n\to \infty}  \frac{1}{\varepsilon^{N-1}}\int_{Q_{\nu}(x_0,\varepsilon) \cap J_{u_n}} \psi(x_0, [u_n](x),\nu_{u_n}(x)\, {\rm d} \mathcal H^{N-1}(x) \notag\\
&= \lim_{\varepsilon \to 0^+}\lim_{n\to \infty}  \int_{Q_{\nu}  \cap \frac{J_{u_n}-x_0}{\varepsilon}} \psi(x_0, [u_n](x_0+\varepsilon y),\nu_{u_{n}(x_0+\varepsilon y)})\, {\rm d} \mathcal H^{N-1}(y)
\notag \\
 &=\lim_{\varepsilon \to 0^+}\lim_{n\to \infty}  \int_{ Q_{\nu} \cap J_{u_{n,\varepsilon}}} \psi(x_0, [u_{n,\varepsilon}](y),\nu_{u_{n,\varepsilon}}(y))\, {\rm d} \mathcal H^{N-1}(y),
\end{align*}
where  we have exploited {\rm \ref{(psi4)}} and  used a change of variables. 
Arguing as in \cite[Theorem 4.5]{CF1997}, we can extract a diagonal sequence $v_k(y):= u_{n_k}(x_0+ \varepsilon_k y)$ such that $v_k(y)\to  v_{0,\lambda,0,\nu}  $ in $L^1 (Q_\nu)  $, $\mathcal E v_k \rightharpoonup  0$ in $L^p(  Q_\nu; \mathbb R^{N\times N}_{\rm sym})$, and 
\begin{align}\label{y3}
\frac{{\rm d} \mu}{{\rm d}\mathcal H^{N-1}\lfloor J_{v_{x_0,\lambda,0,\nu}}} 
 &= \lim_{k\to \infty}  \int_{ Q_{\nu}  \cap J_{v_k} }\psi(x_0, [v_k](y),\nu_{v_k}(y))\, {\rm d} \mathcal H^{N-1}(y).
\end{align}
 Next we follow \cite[Proposition 4.2]{CF1997} (see also Remark \ref{remStBDconv}). 
By Theorem \ref{Al}, for every $k$ there exists $f_k \in SBV(Q_\nu;\mathbb R^N)$ such that
$\nabla f_k=\mathcal E v_k$ and $|D f_k|(Q_\nu)\leq C \|\mathcal E v_k\|_{L^1(Q_\nu)}$.
By piecewise constant approximation in $BV(Q_\nu;\mathbb R^N)$, there exist $g_{k,m}$ such that  $g_{k,m} \to f_k$ in $L^1(Q_\nu;\mathbb R^N)$ and  $|D g_{k,m}|(Q_\nu)\to |D f_k|(Q_\nu)$ as $m \to +\infty$.
Consequently, we also obtain  
$|E g_{k,m}|(Q_\nu)\to |E f_k|(Q_\nu)$ as $m \to +\infty$.
Define
\begin{align*}
w_{k,m}:= v_k- f_k + g_{k,m}.
\end{align*}
Clearly, $\mathcal E w_{k,m}=0$ $\mathcal L^{N}$-a.e.\ and $\lim_{k\to \infty}\lim_{m \to \infty} \Vert w_{k,m}-  v_{0,\lambda,0,\nu}  \Vert _{L^1(Q_\nu)}=0$. Next, 
we have
\begin{align*}
|D^s f_k|(Q_\nu)+ |D^s g_{k,m}|(Q_\nu) \leq C \int_{Q_\nu}|\mathcal E v_k| \,  {\rm d} x  \to 0, \hbox{ as } k\to +\infty. 
\end{align*}
 Thus, by   {\rm \ref{(psi1)}} and {\rm \ref{(psi3)}} we get 
\begin{align*}
\lim_{k\to \infty}\lim_{m\to \infty}\int_{Q_{\nu} \cap J_{w_{k,m}}} \psi(x_0, [w_{k,m}],\nu_{w_{k,m}}) \, {\rm d} \mathcal H^{N-1} \leq \lim_{k\to \infty} \int_{Q_{\nu} \cap J_{v_k}} \psi(x_0, [v_{k}],\nu_{v_k})\, {\rm d}  \mathcal H^{N-1}.
\end{align*}
Hence, we may extract a diagonal sequence in $(k,m)$, say $\lbrace z_l\rbrace$, such that $z_l \to v_{0,\lambda,0,\nu} $ in $L^1(Q_\nu; \mathbb R^N), \mathcal E z_l = 0 \; \mathcal L^N$-a.e., and
\begin{align}\label{fubini0}
 \lim_{l \to \infty} \int_{Q_{\nu} \cap J_{z_l}} \psi (x_0, [z_l], \nu_{z_l}) \, {\rm d}  \mathcal H^{N-1} \leq \lim_{k \to \infty} \int_{Q_{\nu} \cap J_{v_k}} \psi(x_0, [v_{k}],\nu_{v_k})\, {\rm d}  \mathcal H^{N-1}.
 \end{align}
As in \cite[Proposition 4.2]{CF1997}, we can  change the sequence so that  it  equals $ v_{0,\lambda,0,\nu}$  on $\partial Q_{\nu}  (x_0)  $. Precisely, by Fubini's Theorem there exists $r_l \to 1^-$ such that, upon extracting a subsequence,
\begin{align}\label{fubini}
\int_{\partial Q_\nu (1 - r_l) } | {\rm tr}\; z_l - v_{0,\lambda,0,\nu} |  \, {\rm d} \mathcal H^{N-1} \to 0
\end{align}
as $l \to +\infty$. Define
\begin{equation*}
\tilde z_l = 
\begin{cases}
z_l & \text{ in}\; Q_\nu ( 1 - r_l) \\ 
&\\
 v_{0,\lambda,0,\nu} & \text{ in}\; Q_\nu\setminus \overline{Q_\nu ( 1 - r_l))}.
\end{cases}
\end{equation*}
 Clearly, $\mathcal E \tilde z_l = 0$ a.e., and by \eqref{fubini0}--\eqref{fubini}, {\rm \ref{(psi1)}}, and {\rm \ref{(psi3)}} 
\begin{align*}
 \lim_{l\to \infty} \int_{Q_{\nu} \cap J_{\tilde z_l}} \psi (x_0, [\tilde z_l], \nu_{\tilde z_l}) \, {\rm d} \mathcal H^{N-1} \leq \lim_{k \to \infty} \int_{Q_{\nu} \cap J_{v_k}} \psi(x_0, [v_{k}],\nu_{v_k}) \, {\rm d} \mathcal H^{N-1}.
 \end{align*}
 Thus, in view of \eqref{y1}, \eqref{y2}, and \eqref{y3}, \EEE we have  proved that, for $p > 1$,  
$$\Phi(x_0,\lambda,\nu)  \geq   h_p(x_0, \lambda, \nu)$$
for every $x_0\in \Omega$, $\lambda \in \mathbb R^N$, and 
$\nu \in \mathbb S^{N-1}$, where $h_p$ is the function given in \eqref{907}.

 \emph{Step 7: Proof of  \eqref{eq. psiiproof} in  (iii)  for $p=1$.}  We claim that, for a.e.~$x_0\in\Omega$ and for every $\lambda\in\R{N}$ and $\nu\in\mathbb{S}^{N-1}$, 
\begin{equation}\label{clean2}
\Phi(x_0,\lambda,\nu) =  h_{1}(x_0,\lambda,\nu)=\tilde{h}_{1}(x_0,\lambda,\nu), 
\end{equation}
where  
\begin{align*}
\tilde{h}_{1}(x_0, \lambda,\nu)  \coloneqq& \limsup_{\varepsilon \to 0} \inf\Big\{\liminf_{n\to \infty}
\widetilde{E}_{x_0,\eps}(u_n;Q_\nu) : 
\{u_n\} \subset SBD(Q_\nu), u_n \to  v_{0,\lambda,0, \nu}  \hbox{ in } L^1(Q_\nu;\mathbb R^N),\\
&\phantom{=\limsup_{\varepsilon \to 0^+} \inf\Big\{\liminf_{n\to \infty}
E^\bulk_{x_0,\eps}(u_n;Q):} 
\mathcal E u_n \wsto 0 \mbox{ in } \cM(Q_\nu;\mathbb R_{\rm sym}^{N\times N})\Big\},
\end{align*}
with
\begin{align*}
\widetilde{E}_{x_0,\eps}(u;Q_\nu):= \int_{Q_\nu} \varepsilon W\Big(x_0 + \varepsilon y, \frac{1}{\varepsilon} \mathcal E(u)(y)\Big) \, \de y
+ \!\! \int_{Q_\nu\cap J_u} \!\! \psi\big(x_0,[u](x),\nu_u(x)\big)\,\de\cH^{N-1}(x).
\end{align*}
The inequality $h_{1}(x_0,\lambda,\nu)\leq \tilde{h}_{1}(x_0,\lambda,\nu)$ is obtained arguing as above in the proof of $H_1\leq \tilde{H_1}$. Defining $\mathcal{U}_\eps = \lbrace  \text{$\lbrace u_n \rbrace \subset SBD(Q_\nu(x_0,\varepsilon))\colon$   $u_n \to v_{x_0,\lambda,0,\nu}$ in $L^1(Q_\nu(x_0,\varepsilon);\mathbb R^N)$}, \text{ $\mathcal E u_n \wsto 0$ in $\cM(Q_\nu(x_0,\varepsilon);\mathbb R_{\rm sym}^{N\times N})$} \rbrace$, 
for every $x_0\in \Omega$, $\lambda \in \mathbb R^{N}$, and $\nu \in \mathbb S^{N-1}$, we have by Theorem \ref{thm4.3FHP} and a change of variables 
\begin{eqnarray}\label{starstar}
&&\!\!\!\! \Phi(x_0,\lambda,\nu) = 
\limsup_{\varepsilon \to 0}\frac{m(v_{x_0,\lambda,0, \nu},  0; Q_{\nu}(x_0,\varepsilon))}{\varepsilon ^{N-1}}= \limsup_{\varepsilon \to 0}\frac{I_{1}(  v_{x_0,\lambda,0, \nu},  0; Q_{\nu}(x_0,\varepsilon))}{\varepsilon ^{N-1}}  \\
&\!\!\!\!=&\!\!\!\! 
\limsup_{\varepsilon \to 0}\frac{1}{\varepsilon^{N-1}}
\inf_{\mathcal{U}_\eps}\Bigg\{\liminf_{n\to\infty}\bigg[\int_{Q_\nu(x_0,\varepsilon)}  W(x,\mathcal E u_n(x)) \, \de x + \int_{Q_\nu(x_0,\varepsilon)\cap J_{u_n}} \psi(x, [u_n](x),\nu_{u_n}(x)) \,\de\mathcal H^{N-1}(x)\bigg] \Bigg\} \nonumber\\
&\!\!\!\!=&\!\!\!\! 
\limsup_{\varepsilon \to 0}
\inf_{\mathcal{U}_\eps}\Bigg\{\liminf_{n\to\infty}\bigg[ \eps \int_{Q_\nu}  W(x_0+\eps y,\mathcal E u_n(x_0+\eps y)) \, \de y \nonumber\\
&\!\!\!\!&\!\!\!\!
\phantom{\limsup\limits_{\varepsilon \to 0} \inf\Bigg\{\liminf_{n\to\infty}\bigg[}+ \int_{Q_\nu \cap \eps^{-1}(J_{u_n}-x_0)} \psi(x_0+\eps y, [u_n](x_0+\eps y),\nu_{u_n}(x_0+\eps y)) \,\de\mathcal H^{N-1}(y)\bigg]\bigg\}.\nonumber
\end{eqnarray}
 Setting $\mathcal{V} = \lbrace \text{$\lbrace v_n\rbrace \subset  SBD(Q_\nu)$,  $v_n \to v_{0,\lambda,0,\nu}$ in $L^1(Q_\nu;\mathbb R^N)$, $\mathcal E v_n \overset{*}{\rightharpoonup} 0$ in $\cM(Q_\nu;\mathbb R_{\rm sym}^{N\times N})$} \rbrace$,   we derive 
\begin{align}
 &\Phi(x_0,\lambda,\nu)  \nonumber\\
& =\limsup\limits_{\varepsilon \to 0} \inf\limits_{\mathcal{V}}\Bigg\{\liminf\limits_{n\to\infty}\bigg[\int_{Q_\nu} \eps\, W(x_0+\eps y,\eps^{-1}\mathcal E v_n(y)) \, \de y + \int_{Q_\nu \cap J_{v_n}} \!\!\!\! \psi(x_0+\eps y, [v_n](y),\nu_{v_n}(y)) \,\de\mathcal H^{N-1}(y)\bigg]\bigg\}\nonumber\\
&  
=\limsup\limits_{\varepsilon \to 0} \inf\limits_{\mathcal{V}}\Bigg\{\liminf\limits_{n\to\infty}\bigg[\int_{Q_\nu} \eps\, W(x_0+\eps y,\eps^{-1}\mathcal E v_n(y)) \, \de y + \int_{Q_\nu \cap J_{v_n}} \psi(x_0, [v_n](y),\nu_{v_n}(y)) \,\de\mathcal H^{N-1}(y)\bigg]\bigg\} \nonumber\\
& = 
\limsup\limits_{\varepsilon \to 0} \inf\limits_{\mathcal{V}}\Bigg\{\liminf\limits_{n\to\infty} \widetilde{E}_{x_0,\eps}(v_n;Q_\nu) \Bigg\}= \tilde{h}_{1}(x_0,\lambda,\nu), \label{oh_no}
\end{align}
where we have invoked the continuity of $\psi$ in the first variable, see  {\rm \ref{(psi4)}}. 

Now, define $\widehat{m}^{\surface}(\lambda,\nu;Q_\nu(x_0,\varepsilon))\coloneqq \inf\big\{ F (u;Q_\nu(x_0,\varepsilon))\colon u(x_0 + \varepsilon \cdot) \in \mathcal  C_1^\surface (\lambda,\nu)\big\}$.
 We verify that for a.e.~$x_0\in \Omega$, every $\lambda\in\R{N}$, and $\nu\in\mathbb{S}^{N-1}$ the inequality 
$${m(  v_{x_0\lambda,0,\nu}  ,0;Q_\nu(x_0,\varepsilon))\leq \widehat{m}^{\surface}(\lambda,\nu;Q_\nu(x_0,\varepsilon))}$$
holds true. Indeed,  we  observe that an argument entirely similar to \cite[Proposition 4.1, in turn Step 1 of Proposition 3.1]{CF1997} allows us to obtain a lower bound for $\widehat{m}^{\surface}$ in terms of 
$$\inf\Big\{ \liminf_{n\to \infty} F (u_n;Q_\nu(x_0,\varepsilon))\colon u_n(x_0 + \varepsilon \cdot) \wSD{*} (v_{0, \lambda,0,\nu}  , 0) \hbox{ in }Q_\nu \Big\},$$
which along with a change of variables   gives the inequality. Consequently, 
\begin{align*}
\widetilde{h}_{1}(x_0,\lambda,\nu) = \limsup_{\varepsilon \to 0} \frac{m(v_{x_0,\lambda,0,\nu}, 0;Q_\nu(x_0,\varepsilon))}{\varepsilon^{N-1}}
\leq \limsup_{\varepsilon \to 0} \frac{\widehat{m}^{\surface}(\lambda,\nu;Q_\nu(x_0,\varepsilon))}{\varepsilon^{N-1}}=h_{1}(x_0,\lambda,\nu),
\end{align*}
where  the first equality follows from \eqref{starstar}--\eqref{oh_no} and in the last equality we used again a change of variables argument. Therefore, \eqref{clean2} is proved.
\end{proof}

\color{black}

\begin{proof}[Proof of Proposition \ref{thm_propdens}]
Up to replacing  generic elements in $\mathbb R^{N\times N}$ by symmetric ones and gradients by symmetric gradients, estimates~\eqref{H_B} and~\eqref{JoE2.27} can be proved in the same way as \cite[(2.26)--(2.27)]{BMMOZ2022}, which were proved for the case $p>1$ and can be easily adapted to cover the case $p=1$.

\noindent (i): The thesis can be proved exactly in the same way as \cite[Theorem 2.10]{BMMOZ2022}, the only differences being the symmetry in the fields and the regularity assumptions on~$W$, which now is only measurable with respect to the $x$-variable. 
In the present setting, measurability of $H_{p}^B$ is granted by the fact that $H_{p}$ is a Radon--Nikodym derivative.

\noindent (ii): The results concerning  $H_{1}^B$  can be easily adapted from the proof of \cite[Theorem 2.10]{BMMOZ2022}, whereas those concerning  $h_{1}$ require more care. The symmetry property {\rm \ref{psi_0}} is immediate. 
To prove that  $h_{1}$  satisfies the growth condition from above, it suffices to consider an admissible $u\in \cC^{\surface}_1(\lambda,\nu)$ such that $\mathcal E u=0$ a.e.~in~$\Omega$ and to apply properties {\rm \ref{(W1)_p}} with $p=1$, {\rm \ref{W3}}, and {\rm \ref{(psi1)}} (estimate from above). To prove the estimate from below, \BBB we \EEE use the fact that $W\geq0$ (up to adding a constant $\frac{1}{c_W}$) 
and   apply property {\rm \ref{(psi1)}} (estimate from below).
\end{proof}

\color{black}
\color{black}

{

 }

\begin{proof}[Proof of Theorem \ref{representationbis}]
We recall that, as proven in Theorem \ref{representation},   $I_p \colon StBD^p(\Omega)\times \mathcal 
O(\Omega)  \to [0,+\infty)$ satisfies all the assumptions of  Theorem~\ref{GMthmHSDsbd}. Hence, it admits an integral representation. In view of Theorem~\ref{representation}, we know already how to represent $I_p(g,G;\Omega)$ in $StSBD^p$,
i.e., 
\begin{align*}
I_p(g,G)=\int_\Omega H_p(\mathcal E g ,G) \,{\rm d}x + \int_{\Omega \cap J_g} h_p([g], \nu_g) \, {\rm  d} \mathcal{H}^{N-1}
\end{align*}
for every $(g,G) \in StSBD^p(\Omega)$,  where $H_p$ and $h_p$ are given by \eqref{906} and   \eqref{907}, respectively. It is elementary to check that $H_p$ and $h_p$ are $x$-independent as the same property holds for $W$ and $\psi$  (see also Remark \ref{traslinv}). \color{black}  It  remains only to deduce the Cantor density.

 \emph{Step 1:}  To this end, we start by  observing that, for every $(g,G)\in StBD^p(\Omega)$,
\begin{align}\label{derCantor}
\frac{{\rm d} I_p(g, G)}{{\rm d} |E^c g|}(x_0)= \frac{{\rm d} I_p(g, 0)}{{\rm d} |E^c g|}(x_0) \hbox{ for }|E^c g|\hbox{-a.e.\ }x_0.
\end{align}
This can be proven  by following the same arguments used to compute the derivative with respect the jump part (Hausdorff measure) in \cite[Section 4.2]{AMMZ}. Indeed,  by the definition  of  $I_p$ in \eqref{Iploc},  with $g \in BD(\Omega)$  and $G \in L^p(\Omega;\mathbb{R}^{N \times N}_{\rm sym})$,  for every  $O \in \mathcal O(\Omega)$ \color{black} there exists a sequence $\{u_n\}\subset SBD(O)$ converging to $(g,G)$ in the sense of Definition \ref{StBDconv} such that
\begin{align}\label{recory}
I_p(g,G; O)= \lim_{n \to + \infty} \Big[ \int_O W(\mathcal E u_n) \,{\rm d}x + \int_{O \cap J_{u_n}} \psi([u_n],\nu_{u_n})  \, {\rm d} \mathcal{H}^{N-1}\Big].
\end{align}
Moreover by $(H4)$, $I_p(g, G, \Omega)$ and $I_p(g, 0, \Omega)$ are finite. 
Next, consider  a sequence $\{v_n\}$ given by  Theorem~\ref{silhavy_app} 
such that
\[
	v_n \wsto 0\quad\text{in $BV(O)$,} \qquad \nabla v_n= \mathcal E v_n =  -G\,
\]
and (in view of \cite[(2.3)]{MMOZ}) 
\begin{align}\label{estapp}
|E v_n|(O)  \leq |D v_n|(O) \leq C(N)   \int_O |G| \, \de x. 
\end{align}
In particular, the sequence $w_n \coloneqq u_n+ v_n$
is admissible for $I_p(g, 0;O)$. Recall that 
\[
	J_{u_n}= (J_{u_n} \setminus J_{u_n+v_n}) \cup (J_{u_n}\cap J_{u_n+v_n})
	\quad\text{and}\quad\quad 
	J_{u_n+v_n}=   (J_{u_n}\cap J_{u_n+v_n}) \cup (J_{u_n+v_n}\setminus J_{u_n}).
\]	
Using {\rm \ref{(W1)_p}}, {\rm \ref{W4}}, {\rm \ref{(psi1)}},  {\rm \ref{(psi3)}},   the coercivity of $I_p$,  \eqref{recory}--\eqref{estapp},  H\"older's inequality, and the fact that $I_p(g, G, \Omega)$ is finite, we get 
\begin{align*}
     I_p(g, 0;O)- I_p(g,G;O)  &\leq 		 \liminf_{n \to \infty}\left(\int_O \big(W(\mathcal E u_n+ \mathcal E v_n)- W(\mathcal E u_n)\big)\; \de x \right.
			\\
	&\left. \quad  \quad - \int_{O \cap J_{u_n}} \psi([u_n], \nu_{u_n}) \; \de \mathcal H^{N-1} 
			+ \int_{O \cap J_{u_n+ v_n}} \psi([u_n+ v_n], \nu_{u_n+v_n}) \; \de \mathcal H^{N-1} \right)\\		
		& \leq \limsup_{n\to \infty}\left(\int_{O} \bar{C} (1+ |G|^p + |G||\mathcal{E}u_n|^{p-1})    \;  \de x 
			- \int_{O \cap (J_{u_n} \setminus J_{u_n+ v_n})} \psi([u_n], \nu_{u_n})  \, \de \mathcal H^{N-1}   \right.\\
		&\left. \quad  \quad \quad   + \int_{O \cap (J_{u_n} \cap J_{u_n+ v_n})} C_{\psi} |[v_n]|  \, \de \mathcal H^{N-1} 
			+ \int_{O \cap (J_{u_n+v_n} \setminus J_{u_n})} C_{\psi} |[v_n]|  \, \de \mathcal H^{N-1} \right)\\
&\leq \bar{C} \left(\int_{O} (1+ |G|+ |G|^p)\,{\rm d}x \right)+  \hat{C}\left(\int_O |G|^p \de x \right)^{1/p},
\end{align*}
 where $\bar{C}$ depends on $\psi,W, N$, and $\hat{C}$ depends on $\psi, W, N, p, I_p(g, G, \Omega)^{\frac{p-1}{p}})$, and these   
 constants may vary from line to line. 
Interchanging the order of $I_p(g,0;O)$ and $I_p(g,G;O)$, applying it to any ball $O=  B(x_0,\varepsilon)$ and taking the average on it, in  the limit as $\varepsilon \to 0$ we obtain the desired estimate \eqref{derCantor}. 

\emph{Step 2:} Given Step 1, to compute the Radon-Nykodim derivative of $I_p(g,G)$ with respect to $|E^c g|$, it suffices to fix $G=0$. Observing that $I_p(\cdot,0;\cdot)$ is a functional defined in $BD(\Omega) \times \mathcal 
O(\Omega)$,
and in view of Theorem   \ref{representation}  and Remark \ref{remnox}, 
it satisfies  all  assumptions {\rm (H1)}--{\rm (H5)} of \cite[Section 2.2]{CFVG}. In particular, concerning \cite[{\rm (H4)}]{CFVG}, we observe that clearly there exists a modulus of continuity $\Psi$ such that
$$|I_p(v + g(\cdot - x_0), 0;  x_0 + A) - I_p(g,0; A)| \leq \Psi(|x_0| + |v|)(\mathcal L^N
(A) + |Eg|(A))$$ 
for all $(g, A, v, x_0) \in BD(\Omega) \times \mathcal 
O(\Omega) \times \mathbb R^N\times \Omega$, with $x_0 + A \subset \Omega,$
since $I_p(v + g(\cdot - x_0), 0; x_0 + A) = I_p(g,0; A)$, i.e., $\Psi = 0$.
Eventually,  \cite[{\rm (H5)}]{CFVG}  is satisfied since  our functional  
is invariant under  infinitesimal  rigid motions,  i.e., affine functions $a$ with $\nabla a \in \mathbb{R}^{N \times N}_{\rm skew}$.

Hence, $I_p(g,0;\cdot)$ admits the integral representation stated in  \cite[Theorem 2.3]{CFVG}.  From this representation, we will deduce the Radon-Nykodim derivative of $I_p(g,0;\cdot)$ with respect to the Cantor measure.

In particular, since $ I_p  (\cdot, 0; \cdot)$ is a functional of the type considered in \cite[Theorem 2.3]{CFVG},   we can apply the same considerations made in \cite[Section 5, Lemma 5.1]{CFVG}. Thus,  the bulk density is given by  $\frac{{\rm d} I_p (g,0;\cdot)}{{\rm d} \mathcal L^N}(x_0)= H_p(\mathcal E g(x_0),0)$ for $\mathcal L^N$-a.e.\ $x_0 \in \Omega$. Hence, when applying \cite[Theorem 2.3]{CFVG},  we can conclude that $\frac{{\rm d} I_p(g, 0)}{{\rm d} |E^c g|}(x_0)= H_p^\infty (\frac{{\rm d} E^c g}{|{\rm d} E^c g|}(x_0),0 )$. In fact, in that  theorem, the density of the Cantor part coincides with the weak recession function of the bulk density. 
This fact, together with \eqref{derCantor} and \eqref{frepSD}--\eqref{PhirepBD},   gives \eqref{reprelaxbis},  which concludes the proof.
 \end{proof}



 

 \section{Proof of the linearization result}\label{proof:lin}

In this section we give the proofs of the results announced in Subsection \ref{sec: lin}.

\begin{proof}[Proof of Proposition \ref{prop: rigidity}]
Consider a sequence $\{u_\delta\} \subset SBV^2_2(\Omega;\mathbb{R}^N)$ such that the functions  $y_\delta := {\rm id} + \delta u_\delta$ satisfy \eqref{eq: rot back} and $F_\delta (y_\delta) \leq M$.  We first prove the Poincar\'e-type estimate \eqref{gradientbound}(i) (Step 1). Afterwards,  we apply the $BV$ coarea formula to define suitable level sets for the gradient,  from which the definition of the sets $\{S_\delta\}$  follows (Step 2). Then, we prove \eqref{gradientbound}(ii),(iii) (Step 3) and eventually the compactness result (Step~4).

\emph{Step 1: Proof of \eqref{gradientbound}(i).}   By Poincar\'e's inequality in $BV$ we find $Y^0_\delta \in \mathbb{R}^{N \times N}$ such that 
\begin{align*}
\Vert  \nabla y_\delta - Y_\delta^0 \Vert_{L^{1^*}(\Omega)} \leq |D  (\nabla y_\delta)|(\Omega) \leq \Vert \nabla^2 y_\delta \Vert_{L^1(\Omega)}
+ \int_{J_{\nabla y_\delta}} |[\nabla y_\delta]| \, {\rm d}\mathcal{H}^{N-1},
\end{align*}
where $1^* = \frac{N}{N-1}$. 
Using H\"older's inequality, \eqref{F-en}, $F_\delta(y_\delta) \leq M$, and the fact that $\Psi$ satisfies {\rm \ref{(psi1)}}, we then get
\begin{align}\label{eq: L7XXX0}
\Vert  \nabla y_\delta - Y^0_\delta  \Vert_{L^{1^*}(\Omega)}  \leq (\mathcal{L}^d(\Omega))^{1/2}  \Vert \nabla^2 y_\delta \Vert_{L^2(\Omega)} + c_\Psi^{-1} \int_{J_{\nabla y_\delta}}  \Psi \big(x, [\nabla y_\delta], \nu_{\nabla y_\delta}\big) \, {\rm d}\mathcal{H}^{N-1} \leq C \delta^\beta 
\end{align}
for a constant $C>0$ only depending on $\Omega$ and $M$. Then, there exists a constant $c_0>0$ only depending on $\Omega$ and $M$ such that the set $A_\delta := \lbrace |\nabla y_\delta - Y^0_\delta | \leq  c_0\delta^\beta  \rbrace$ satisfies $\mathcal{L}^N(A_\delta   ) \geq \frac{1}{2} \mathcal{L}^N(\Omega)$. We claim that ${\rm dist}(Y^0_\delta,SO(N)) \leq 2c_0 \delta^\beta$. Otherwise, if we had ${\rm dist}(Y^0_\delta,SO(N)) > 2c_0 \delta^\beta $, we would find ${\rm dist}(\nabla y_\delta,SO(N) ) \geq c_0 \delta^\beta$ on $A_\delta$, and then along with  \ref{assumptions-Wiii} 
$$\frac{1}{\delta^2}\int_\Omega V (\nabla y_\delta) \, {\rm d}x\geq  \frac{\mathcal{L}^N(\Omega)}{2\delta^2}  c \big(c_0\delta^\beta\big)^2.  $$
As $\beta <1$, for $\delta$ small enough the right-hand side exceeds $M$ which contradicts the energy bound $F_\delta(y_\delta) \leq M$. 

Then, from ${\rm dist}(Y^0_\delta,SO(N)) \leq 2c_0 \delta^\beta $, \eqref{eq: rot back}, and \eqref{eq: L7XXX0} we get $|{\rm Id} - Y_\delta^0| \leq C \delta^\beta$. Again by \eqref{eq: L7XXX0} this implies 
\begin{align}\label{eq: L7XXX1}
\Vert  \nabla y_\delta - {\rm Id} \Vert_{L^{1^*}(\Omega)}  \leq   C \delta^\beta. 
\end{align}
We recall the linearization formula  (see \cite[(3.20)]{FrieseckeJamesMueller:02}) 
\begin{align}\label{rig-eq: linearization}
|{\rm sym}(Z -{\rm Id})| =  {\rm dist}(Z,SO(N)) + {\rm O} (|Z- {\rm Id}|^2) \quad \text{for $Z \in \mathbb{R}^{N \times N}$.}
\end{align}
By distinguishing the cases $|Z- {\rm Id}| \leq N  $ and $|Z- {\rm Id}|>  N $, it is elementary \color{black} to see that 
\begin{align*}
|{\rm sym}(Z -{\rm Id})| \leq   {\rm dist}(Z,SO(N)) + C\big(|Z- {\rm Id}|^{1^*} +  {\rm dist}^2(Z,SO(N))\big) \quad \text{for $Z \in \mathbb{R}^{N \times N}$.}
\end{align*} 
Then, from \ref{assumptions-Wiii}, H\"older's inequality,  \eqref{eq: L7XXX1}, and $F_\delta( y_\delta ) \leq M$  we get
\begin{align*}
\int_\Omega |{\rm sym}(\nabla y_\delta -{\rm Id})| \, {\rm d}x \leq C (\mathcal{L}^N(\Omega))^{1/2} (M\delta^2)^{1/2} + CM\delta^2 + C(\delta^\beta)^{1^*}.
\end{align*}
By the fact that  $\beta \geq \frac{N-1}{N}$ and the definition $y_\delta = {\rm id} + \delta u_\delta$ we then deduce
\begin{align}\label{compi1}
\Vert \mathcal{E} u_\delta \Vert_{L^1(\Omega)} \leq C.
\end{align}
By {\rm \ref{(psi1)}}--{\rm \ref{(psi2)}}, \eqref{F-en}, and $F_\delta(y_\delta) \leq M$ we further find
\begin{align}\label{compi2}
|D^{\rm j} u_\delta |(\Omega) \leq c_\psi^{-1} \int_{J_{u_\delta}} \psi( [u_\delta], \nu_{u_\delta}) \, {\rm d}\mathcal{H}^{N-1} = c_\psi^{-1} \frac{1}{\delta}\int_{J_{y_\delta}} \psi( [y_\delta], \nu_{y_\delta}) \, {\rm d}\mathcal{H}^{N-1} \leq CM.
\end{align}
The two estimates \eqref{compi1}--\eqref{compi2} show $|Eu_\delta|(\Omega) \leq C$. From \eqref{eq: rot back} we further infer that $\fint_\Omega u_\delta \, {\rm d}x = 0$ and $\fint_\Omega (\nabla u_\delta^T- \nabla u_\delta) \, {\rm d}x = 0$. Indeed, to see that latter, we write  $A:=  \fint_\Omega \nabla y_\delta \, {\rm d}x$ as the polar decomposition $A = RS$ for an orthogonal matrix $R$ and a symmetric matrix $S$. Then, \eqref{eq: rot back} indeed shows $R = {\rm Id}$, and thus $A$ is symmetric.  Then,  the variant of the  Korn-Poincar\'e inequality in $BD$ stated in Proposition \ref{prop: KP} below  yields
$$ \Vert u_\delta \Vert_{L^1(\Omega)} \leq C|Eu_\delta|(\Omega) + C|Du_\delta(\Omega)| \leq C|Eu_\delta|(\Omega) + C| D^{\rm j}  u_\delta(\Omega) + \int_\Omega \mathcal{E}u_\delta \, {\rm d}x|   \leq C, $$
where we also used \eqref{compi1}--\eqref{compi2}. This concludes the proof of \eqref{gradientbound}(i).

\emph{Step 2: Construction of the sets $S_\delta$.} Next, we apply the $BV$ coarea formula (see \cite[Theorem 3.40]{AFP}) on each component $(\nabla y_\delta)_{ij} \in SBV^2(\Omega)$, $1 \leq i,j\leq N$,  to write 
\begin{align*}
 \int_{-\infty}^\infty \mathcal{H}^{N-1}\big( \Omega   \cap \partial^* \lbrace (\nabla y_\delta)_{ij} > t \rbrace \big) \,{\rm d}t &= |D  (\nabla y_\delta)_{ij}|(\Omega ) = \Vert \nabla^2 y_\delta \Vert_{L^1(\Omega)}
+ \int_{J_{\nabla y_\delta}} |[\nabla y_\delta]| \, {\rm d}\mathcal{H}^{N-1} .
\end{align*}
Then, similarly as in \eqref{eq: L7XXX0}, we get  
\begin{align}\label{eq: L7XXX}
 \int_{-\infty}^\infty \mathcal{H}^{ N -1}\big(\Omega   \cap \partial^* \lbrace (\nabla y_\delta)_{ij} > t \rbrace \big) \,{\rm d}t   \leq C \delta^\beta. 
\end{align}
Fix $\gamma \in (\frac{2}{3},\beta)$ and define $T_\delta= \delta^{\gamma}$. For all  $k \in \mathbb{Z}$ we find $t^{ij}_k \in (kT_\delta, (k+1)T_\delta]$    such that 
\begin{align}\label{eq: coarea}
\mathcal{H}^{N-1}\big( \Omega  \cap   \partial^*\lbrace (\nabla y_\delta)_{ij} >t^{ij}_k \rbrace \big) \leq \frac{1}{T_\delta} \int_{kT_\delta}^{{(k+1)T_\delta}} \mathcal{H}^{N-1}\big( \Omega  \cap  \partial^*\lbrace (\nabla y_\delta)_{ij} >t \rbrace \big)\, {\rm d}t.
\end{align}
Let $G_k^{\delta,ij} = \lbrace (\nabla y_\delta)_{ij} > t^{ij}_{k} \rbrace \setminus \lbrace (\nabla y_\delta)_{ij} > t^{ij}_{k+1} \rbrace$ and note that each set has finite perimeter in $\Omega$ since it is the difference of two sets of finite perimeter.  Now,  \eqref{eq: L7XXX} and  \eqref{eq: coarea} yield
\begin{align}\label{eq: nummer geben}
\sum\nolimits_{k\in\mathbb{Z}} \mathcal{H}^{N-1}\big( \Omega  \cap \partial^* G_k^{\delta, ij}  \big) \leq 2T^{-1}_\delta \cdot C\delta^\beta \leq C \delta^{\beta-\gamma}.
\end{align}
Since $\mathcal{L}^N(\Omega \setminus \bigcup_{k \in \mathbb{Z}} G_k^{\delta, ij})=0$, $\{G_k^{\delta,ij}\}_{k\in \mathbb{Z}}$ is a Caccioppoli partition of $\Omega$.  

We let $\{P^\delta_l\}_{l\in\mathbb{N}}$ be the Caccioppoli partition of $\Omega$ consisting of the nonempty sets of 
$$\big\{ G_{k_{11}}^{\delta, 11} \cap G_{k_{12}}^{\delta, 12}  \cap \ldots \cap  G_{k_{NN}}^{\delta, NN}: \  k_{ij} \in \mathbb{Z}\text{ for } i,j=1,\ldots,N\big\}.$$
Then,   \eqref{eq: nummer geben}  implies 
\begin{align}\label{eq: coarea3}
\sum\nolimits_{l =1}^\infty\mathcal{H}^{N-1}\big(\partial^*P_l^\delta \cap \Omega   \big) \leq C\delta^{\beta-\gamma}
\end{align}
\color{black}
for a constant $C>0$ independent of $\delta$.   Recalling $T_\delta = \delta^{\gamma}$ we get $|t_k^{ij} -t_{k+1}^{ij}| \leq 2T_\delta=2\delta^\gamma$ for all $k \in \mathbb{Z}$, $i,j=1,\ldots,N$. Therefore, by the definition of $G_k^{\delta,ij}$, for each component $P^\delta_l$ of the Caccioppoli partition, we find a matrix $Z_l^\delta \in \mathbb{R}^{N\times N}$ such that

\begin{align}\label{eq: coarea4}
\Vert\nabla y_\delta - Z_l^\delta \Vert_{L^\infty(P^\delta_{l})} \leq C\delta^{\gamma}. 
\end{align}
Without restriction we order the components $\{P^\delta_ {l}\}_l$ such that $\mathcal{L}^N(P_l^\delta)$ is decreasing in $ l$.  As by  this ordering at most $P^\delta_1$ can have volume larger or equal to $\frac{1}{2}\mathcal{L}^N(\Omega)$, we get   $ \mathcal{L}^N(P^\delta_{ l}) \leq \frac{1}{2}\mathcal{L}^N(\Omega)$ for all $ l \geq 2$. 
Then, defining $S_\delta := \bigcup_{l \geq 2} P_l^\delta$, we get by \eqref{eq: coarea3} and the relative isoperimetric inequality in $\Omega$, 

\begin{align}\label{voluem}
{\mathcal L}^N(\Omega \setminus P_1^\delta) = {\mathcal L}^N(S_\delta) = \sum_{l\geq 2} {\mathcal L}^N(P_l ^\delta) \leq \sum_{l\geq 2} C_\Omega \big(\mathcal{H}^{N-1}(\partial^* P^\delta_{l} \cap \Omega) \big)^{\frac{N}{N-1}} \leq C \sum_{l\geq 2} \mathcal{H}^{N-1}(\partial^* P^\delta_{l} \cap \Omega)   \leq C\delta^{\beta -\gamma}.
\end{align}
Again using \eqref{eq: coarea3}, we also find $\mathcal{H}^{N-1}(\partial^* S_\delta \cap \Omega) \leq C\delta^{\beta-\gamma}$. As $\gamma < \beta$,   \eqref{S-conv} follows from Lemma \ref{lemma: iso} below.

\emph{Step 3: Proof of \eqref{gradientbound}(ii),(iii).}   First, \eqref{voluem} implies $\mathcal{L}^N(P^\delta_1) \geq \frac{1}{2}\mathcal{L}^N(\Omega)$ for $\delta$ small enough. This along with  \eqref{eq: L7XXX1} and \eqref{eq: coarea4} yields, for a constant $C>0$ only depending on $M$ and $\Omega$,
\begin{align*}
|Z^\delta_1 - {\rm Id}| \leq C\Vert  Z_1^\delta  - {\rm Id} \Vert_{L^1(P^\delta_1)} \leq  C\Vert \nabla y_\delta  - {\rm Id} \Vert_{L^1(P^\delta_1)}  + C\delta^\gamma \leq C\delta^\beta + C\delta^\gamma \leq C\delta^\gamma,
\end{align*}
where we again used H\"older's inequality. This together with \eqref{eq: coarea4}  shows  
\begin{align}\label{eq: coarea5}
\Vert  \nabla y_\delta - {\rm Id} \Vert_{L^\infty(P^\delta_1)} \leq C\delta^{\gamma}. 
\end{align}
Recalling $P^\delta_1 = \Omega \setminus S_\delta$  and $\nabla y_\delta = {\rm Id} + \delta \nabla u_\delta$, we obtain \eqref{gradientbound}(ii).  We proceed with the proof of \eqref{gradientbound}(iii). By \eqref{rig-eq: linearization} and  Young's inequality  we have
\begin{align*}
|{\rm sym}(Z -{\rm Id})|^2& \leq  C {\rm dist}^2(Z,SO(N))+ C|Z-{\rm Id}|^4 .
  \end{align*}
Then, we obtain
\begin{align*}
\int_{\Omega \setminus S_\delta} |{\rm sym}( \nabla y_\delta - {\rm Id})|^2 &\leq C\int_{\Omega \setminus S_\delta}  \big( {\rm dist}^2(\nabla y_\delta ,SO(N)) +  |\nabla y_\delta-{\rm Id}|^4 \big) \, {\rm d}x ,
\end{align*}
and therefore \eqref{eq: coarea5} yields
 \begin{align*}
\int_{\Omega \setminus S_\delta} |\mathcal{E}  u_\delta |^2 \, {\rm d}x = \frac{1}{\delta^2}  \int_{\Omega \setminus S_\delta} |{\rm sym}( \nabla y_\delta - {\rm Id})|^2 \, {\rm d}x  \leq   C\frac{1}{\delta^2} \int_{\Omega  } {\rm dist}^2(\nabla y_\delta,SO(N)) \, {\rm d}x + C\frac{1}{\delta^2} \mathcal{L}^d(\Omega) \delta^{4\gamma}  \leq C,
\end{align*}
where in the last step we have   used $\gamma >\frac{2}{3} \geq \frac{1}{2}$, \ref{assumptions-Wiii},  \eqref{F-en},  and $F_\delta(y_\delta) \leq M$. This shows   \eqref{gradientbound}(iii).

\emph{Step 4: Compactness.} We recall from \eqref{compi1}--\eqref{compi2} that  $|Eu_\delta|(\Omega) \leq C$ for $C>0$ independent of $\delta$. Thus, by a compactness result in the space $BD$ we find $g \in BD(\Omega)$ such that, up to a subsequence (not relabeled), $u_\delta \to  g $ in $L^1(\Omega; \mathbb{R}^N)$. The existence of $G \in L^2(\Omega;\mathbb{R}^{N \times N}_{\rm sym})$ such that $\mathcal E u_\delta \, \chi_{\Omega \setminus S_\delta} \rightharpoonup G$ in $L^2(\Omega; \mathbb R^{N\times N}_{\rm sym})$ follows by  \eqref{gradientbound}(iii) and weak compactness in $L^2$.
\end{proof}

As a preparation for the proof of the $\Gamma$-convergence result, we need the following lemma.

\begin{lemma}\label{lem- eu}
Let $u \in BD(\Omega)$ and let $\{P_n\} \subset \Omega$ be a sequence of smooth closed sets with $\mathcal{H}^{N-1}(\partial P_n ) \to 0$ as $n \to \infty$. Then, $|Eu|(P_n) \to 0 $ as $n \to +\infty$. 
\end{lemma}

\begin{proof}
We use the characterization of $BD$ functions  by slicing established in \cite{ACDM}. To this end, we first need to introduce some further notation. For $\xi \in \mathbb{S}^{N-1}$, we let
\begin{equation*}
\xi^\perp:=\{y\in  \mathbb{R}^N  \colon y\cdot \xi=0\},\qquad B^\xi_y:=\{t\in \mathbb{R}\colon y+t\xi \in B\} \ \ \ \text{ for any $y\in \mathbb{R}^N$ and $B\subset \mathbb{R}^N$}\,,
\end{equation*}
and for every function $v\colon B\to  \mathbb{R}^N  $ and $t\in B^\xi_y$ let
\begin{equation*}
v^\xi_y(t):=v(y+t\xi) \cdot \xi\,.
\end{equation*}
By  \cite[Proposition 3.2]{ACDM}, given $u \in BD(\Omega)$ and $\xi \in \mathbb{S}^{N-1}$, the function $y \mapsto |Dv^\xi_y|(\Omega^\xi_y)$ lies in $L^1(\xi^\perp;\mathcal{H}^{N-1})$ with  
\begin{align}\label{slicing}
(Ev \xi,\xi)(B) = \int_{\xi^\perp} Dv^\xi_y(B^\xi_y)\, {\rm d}\mathcal{H}^{N-1}  (y) \quad \quad  \text{for any Borel set $B\subset \mathbb{R}^N$}. 
\end{align}
Let $\{e_i\}_{i=1,\ldots,N}$ be an orthonormal basis of $\mathbb{R}^N$ and let $\mathcal{V} = \lbrace e_i + e_j \colon i,j=1,\ldots,N\rbrace$. Clearly, we have
$$|Eu|(B) \leq \sum_{\xi \in \mathcal{V}}|(Eu \xi,\xi)|(B) \quad \text{for any Borel set $B\subset \mathbb{R}^N$}.    $$
Thus, in order to prove the statement, it suffices to show that $|(Eu \xi,\xi)|({P_n}) \to 0$ as $n \to \infty$ for each $\xi \in \mathcal{V}$.  

Fix any $\xi \in \mathcal{V}$ and denote by $P_n^\xi$ the orthogonal projection of $P_n$ onto $\xi^\perp$. By the area formula it is elementary to check that  $\mathcal{H}^{N-1}(P_n^\xi) \leq \mathcal{H}^{N-1}(\partial P_n)$.    As $ \mathcal{H}^{N-1}(\partial P_n) \to 0$, we thus get $\mathcal{H}^{N-1}(P_n^\xi) \to 0$ for $n \to \infty$. By \eqref{slicing} we derive
\begin{align*}
|(Eu \xi,\xi)|(P_n) =   \int_{\xi^\perp} |Dv^\xi_y|\big((P_n)^\xi_y\big)\, {\rm d}\mathcal{H}^{N-1} \leq  \int_{P_n^\xi} |Dv^\xi_y|(\Omega^\xi_y)\, {\rm d}\mathcal{H}^{N-1}. 
\end{align*}
As $\mathcal{H}^{N-1}(P_n^\xi) \to 0$ for $n \to \infty$ and  the function $y \mapsto |Dv^\xi_y|(\Omega^\xi_y)$ lies in $L^1(\xi^\perp;\mathcal{H}^{N-1})$, we conclude $|(Eu \xi,\xi)|(P_n) \to 0$ as $n\to \infty$.
\end{proof}

We close with the proof of Theorem \ref{th: gamma-conv}.

\begin{proof}[Proof of Theorem \ref{th: gamma-conv}]
(i)  Consider $(g,G)\in StBD^2(\Omega)$ and a sequence $\{u_\delta\}$ with $u_\delta  \rightsquigarrow (g,G)$. It is not restrictive to assume that $\sup_{\delta >0} {F}^{\rm dis}_{\delta} (u_\delta) < +\infty$ as otherwise there is nothing to show.  We apply Proposition~\ref{prop: rigidity} and let $\{S_\delta\}$ be the sets corresponding to $\{u_\delta\}$. Clearly, by $u_\delta  \rightsquigarrow (g,G)$ and Proposition~\ref{prop: rigidity} we get  $\mathcal E u_\delta \, \chi_{\Omega \setminus S_\delta} \rightharpoonup G$ weakly in $L^2(\Omega; \mathbb R^{N\times N}_{\rm sym})$ as $\delta \to 0$.

 Up to passing to a subsequence (not relabeled), and still using a continuous parameter  for the sequence for the sake of brevity, we can assume that 
\begin{align}\label{sumdelta}
\sum_{\delta >0}   \mathcal{H}^{N-1}(\partial^* S_\delta)    < +\infty  
\end{align}
and 
\begin{align}\label{liminf}
\liminf_{\delta \to 0} F^{\rm dis}( u_\delta, \Omega) = \lim_{\delta\to 0} \int_\Omega W(\mathcal E u_\delta)\,\de x+ \lim_{\delta\to 0} \int_{\Omega\cap J_{u_\delta}} \psi([u_\delta](x), \nu_{u_\delta}(x))\,{\rm d}\mathcal{H}^{N-1}(x),
\end{align}
where we recall the definition of $F^{\rm dis}$ in \eqref{Elin}. Following verbartim the proof of \cite[Theorem 2.7]{Friedrich}, by using \eqref{gradientbound}(ii) for $\gamma >\frac{2}{3}$, we get
$$\liminf_{\delta \to 0}  \int_\Omega \frac{1}{\delta^2} V( {\rm Id} + \delta \nabla u_\delta) \, {\rm d}x \geq \liminf_{\delta \to 0}  \int_{\Omega \setminus S_\delta} W(\mathcal{E} u_\delta) \, {\rm d}x. $$
Therefore, recalling $y_\delta = {\rm id} + \delta u_\delta$, \eqref{F-en} and \eqref{def-energy},    by \eqref{liminf} and the nonnegativity of $\psi$, $\Psi$, as well as by {\rm \ref{(psi2)}}  we get
\begin{align}\label{prepari}
\liminf_{\delta \to 0} {F}^{\rm dis}_{\delta} (u_\delta) \geq \liminf_{\delta \to 0} \int_{\Omega  \setminus S_\delta } W(\mathcal{E} u_\delta) \, {\rm d}x+ \lim_{\delta\to 0} \int_{\Omega\cap J_{u_\delta}} \psi([u_\delta], \nu_{u_\delta})\,{\rm d}\mathcal{H}^{N-1} \geq \liminf_{\delta \to 0} F^{\rm dis}(u_\delta,\Omega \setminus S_{\delta\color{black}}).
\end{align}
We define $T_\eta = \bigcup_{0 <\delta \leq \eta} S_\delta$ and note that $ \mathcal{H}^{N-1}(\partial^* T_\eta) \to 0$ as $\eta \to 0$ by \eqref{sumdelta}.  We can approximate $T_\eta$ by a smooth closed set $U_\eta \supset T_\eta$ with $\mathcal{H}^{N-1}(\partial U_\eta) \leq C\mathcal{H}^{N-1}(\partial^* T_\eta)$, see e.g.\ \cite[Equation (1.9)]{schmidt}. Therefore, using also the isoperimetric inequality  we have 
\begin{align}\label{etalimit}
 \mathcal{L}^{N}(U_\eta)  +  \mathcal{H}^{N-1}(\partial U_\eta) \to 0 \quad \text{ as $\eta \to 0$.}
 \end{align}
Since $W$ and $\psi$ are nonnegative densities, we derive from \eqref{prepari} and an application of  Theorem \ref{representationbis} on $F^{\rm dis}$ that
\begin{align*}
\liminf_{\delta \to 0} {F}^{\rm dis}_{\delta} (u_\delta) & \geq \liminf_{\delta \to 0} F^{\rm dis}(u_\delta,\Omega \setminus T_\eta) \geq \liminf_{\delta \to 0} F^{\rm dis}(u_\delta,\Omega\setminus U_\eta) \\
&  \geq  \int_{\Omega \setminus U_\eta}H(\mathcal E g,G) \, \de x + 
		\int_{(\Omega \setminus U_\eta) \cap J_g}h([g],\nu_g) \, \de \mathcal H^{N-1}  + \int_{\Omega \setminus U_\eta} H^\infty\left(\frac{d E^c g}{d |E^c g|} ,0\right) d |E^c g|.
\end{align*}
Eventually, by using Lemma \ref{lem- eu}, \eqref{etalimit}, and the fact that $(g,G) \in StBD^2(\Omega)$, in the limit $\eta \to 0$ we derive
$${\liminf_{\delta \to 0} {F}^{\rm dis}_{\delta} (u_\delta) \geq  \int_{\Omega }H(\mathcal E g,G) \, \de x + 
		\int_{\Omega  \cap J_g}h([g],\nu_g) \, \de \mathcal H^{N-1} + \int_{\Omega} H^\infty\left(\frac{d E^c g}{d |E^c g|} ,0\right) d |E^c g|  = I_{\rm lin}(g,G).} $$
 Here, we have also used that $H$ satisfies the upper bound in \eqref{JoE2.27} and that $h$ satisfies the upper bound in   {\rm \ref{(psi1)}}  (see Theorem \ref{thm_propdens}(i)).

(ii) We fix an arbitrary sequence $\lbrace\delta_n\rbrace$	with $\delta_n \to 0$ as $n \to \infty$. By Theorem \ref{representation} applied on $F^{\rm dis}$ (see  again \eqref{Elin}) we find a sequence $\{v_m\}\subset SBD^2(\Omega)$ with 
\begin{align}\label{limsup1}
\text{$v_m\wSD{*}(g,G)$ \ \  as $m \to \infty$  \ \ \   and \ \ \  $\lim_{m\to\infty} F^{\rm dis}(v_m,\Omega) = I_{\rm lin}(g,G)$}.
\end{align} 
By a density result in $SBD^2(\Omega)$, see \cite[Theorem 1.1]{crismale}, we can approximate each $v_m$ by a sequence $v_{m,k} \in SBV^2(\Omega;\mathbb{R}^N)$ such that each $J_{v_{m,k}}$ is closed and included in a finite union of closed connected pieces of $C^1$-hypersurfaces, $v_{m,k} \in C^\infty(\overline{\Omega} \setminus J_{v_{m,k}},\mathbb{R}^N)\cap W^{2,\infty}(\Omega \setminus J_{v_{m,k}};\mathbb{R}^N)$, and 
$$\lim_{k\to \infty} \Big( \Vert v_{m,k} - v_m \Vert_{L^1(\Omega)} +| Ev_{m,k} - Ev_m|(\Omega) + \Vert \mathcal E v_{m,k} - \mathcal E v_m \Vert_{L^2(\Omega)}   \Big) = 0. $$
In particular, by the fact that $W$ is a quadratic form and by the continuity of $\psi$ we get 
$$\lim_{k \to \infty} F^{\rm dis}(v_{m,k},\Omega)  = F^{\rm dis}(v_{m},\Omega)  \quad \text{for all $m \in \mathbb{N}$}. $$
By a diagonal argument, we can choose nondecreasing sequences $\lbrace m_n\rbrace$ and $\lbrace k_n\rbrace$ with $m_n, k_n \to \infty$ as $n\to \infty$ such that
\begin{align}\label{limsup2}
\lim_{n \to \infty} \Big( \Vert v_{m_n,k_n} - v_{m_n} \Vert_{L^1(\Omega)} + \Vert \mathcal E v_{m_n,k_n} - \mathcal E v_{m_n} \Vert_{L^2(\Omega)}   + \big| F^{\rm dis}(v_{m_n,k_n},\Omega) - F^{\rm dis}(v_{m_n},\Omega)\big| \Big) = 0, 
\end{align}
and 
\begin{align}\label{limsup3}
\Vert \nabla v_{m_n,k_n}  \Vert_{L^\infty(\Omega)} +  \Vert \nabla^2 v_{m_n,k_n}  \Vert_{L^\infty(\Omega)} \leq \delta_n^{-(1-\beta)/2}. 
\end{align}
Defining $u_n := v_{m_n,k_n} \in SBV^2_2(\Omega;\mathbb{R}^N)$, by \eqref{limsup1} and \eqref{limsup2} we get
$$\text{$u_n\wSD{*}(g,G)$ \ \  as $ n  \to \infty$  \ \ \   and \ \ \  $\lim_{n\to\infty} F^{\rm dis}(u_n,\Omega) = I_{\rm lin}(g,G)$}. $$
 To conclude, it suffices to show that $\lim_{n\to\infty} F^{\rm dis}(u_n,\Omega) = \lim_{n\to\infty} F_{\delta_n}^{\rm dis} (u_n)$. To this end, recalling the relation $y_n = {\rm id} + \delta_n u_n$ and comparing \eqref{F-en} with \eqref{Elin}, we need to check that 
 $$\lim_{n\to \infty} \int_\Omega \frac{1}{\delta_n^2} V({\rm Id} + \delta_n \nabla u_n) + \frac{1}{\delta_n^{2\beta}} \delta_n^2 |\nabla^2 u_n|^2 \, {\rm d}x   +  \frac{ 1}{\delta_n^{\beta}} \delta_n \int_{J_{\nabla u_n}}  \Psi \big( [\nabla u_n], \nu_{\nabla u_n}\big) \, {\rm d}\mathcal{H}^{N-1} = \lim_{n\to \infty} \int_{\Omega}   W( \mathcal E u_n)\, {\rm d}x.  $$
The second and the third term on the left-hand side vanish  by \eqref{limsup3}, {\rm \ref{(psi2)}} for $\Psi$, and   $\beta <1$.  For the first term, we argue as in the proof of \cite[Theorem 2.7]{Friedrich}: we use the Taylor expansion $V({\rm Id} + F) =  W( {\rm sym}(F)) + \omega(F)$ with $|\omega(F)|\leq C|F|^3$ for $|F| \leq 1$, and   compute by \eqref{limsup3} (recall that $W$ is a quadratic form)
\begin{align*} 
\lim_{n \to \infty}\frac{1}{\delta_n^2} \int_{\Omega} V({\rm Id} + \delta_n\nabla u_n) &  = \lim_{n\to \infty}  \int_{\Omega}  \Big( W( \mathcal E u_n) + \frac{1}{\delta_n^2} \omega(\delta_n \nabla u_n)  \Big) \, {\rm d}x  \\ & = \lim_{n\to \infty} \int_{\Omega}   W( \mathcal E u_n)\, {\rm d}x +   \lim_{n \to \infty}\int_{\Omega}    {\rm O}\big( \delta_n|\nabla u_n|^3 \big)=  \lim_{n\to \infty} \int_{\Omega}   W( \mathcal E u_n)\, {\rm d}x,
\end{align*}
where the last step follows from the fact that $\beta >\frac{1}{3}$. This concludes the proof.
\end{proof}

\appendix

\section{BV and BD fields}\label{BD}

 
We state some approximation results dealing with $BV$ fields which are used throughout the paper.

\begin{theorem}[{Alberti's Theorem \cite[Theorem~2.8]{CF1997}}]\label{Al}
Let $d \in \mathbb N$ and $G \in L^1(\Omega; \R{d{\times} N})$. 
Then there exist a function $f \in SBV(\Omega; \R d)$, a Borel function $\beta\colon\Omega\to\R{d{\times} N}$, and a constant $C_N>0$ depending only on $N$ such that
\begin{equation*}
Df = G \,{\cL}^N + \beta \cH^{N-1}\res J_f, \qquad
\int_{\Omega\cap J_f}  |\beta| \, \de \cH^{N-1} \leq C_N \lVert G\rVert_{L^1(\Omega)}.
\end{equation*}
\end{theorem}
\begin{lemma}[{\cite[Lemma~2.9]{CF1997}}]\label{ctap}
Let $d \in \mathbb N$ and $u \in BV(\Omega; \R d)$. Then there exist piecewise constant functions $\bar u_n\in SBV(\Omega;\R d)$  such that $\bar u_n \to u$ in $L^1(\Omega; \R d)$ and
\begin{equation*} 
|Du|(\Omega) = \lim_{n\to \infty}| D\bar u_n|(\Omega) = \lim_{n\to \infty} \int_{\Omega\cap J_{\bar u_n}} |[\bar u_n]|\; \de\cH^{N-1}.
\end{equation*}
\end{lemma}

\begin{theorem}[{\cite[Theorem~1.2]{Silhavy}}]\label{silhavy_app}
 Let  $(g,G)\in BV(\Omega;\mathbb R^N)\times L^1(\Omega;\mathbb R^{N\times N})$. Then,  there exists a sequence $\{g_n\} \subset SBV(\Omega;\mathbb R^N)$ such that
$g_n \wsto g$ in $BV(\Omega;\mathbb R^N)$ and $\nabla g_n =G$.
\end{theorem}

We also recall some facts about functions of bounded deformation,  referring to   \cite{ACDM, BFT, BCDM, T, TS} for more details. A function $u\in L^{1}(\Omega;{\mathbb{R}}^{N})$ is said to be of \emph{bounded deformation},  written $u\in BD(\Omega)$, if the symmetric part  $ Eu : = \frac{Du + Du^T}{2}$ of its distributional derivative $Du$, is a matrix-valued bounded Radon measure. 
The space $BD(\Omega)$ is a Banach space when endowed with the norm 
\begin{equation*}
\|u\|_{BD(\Omega)} = \|u\|_{L^1(\Omega)} + | E u|(\Omega ).
\end{equation*}
The strict convergence in the space $BD(\Omega)$ is the one determined by the distance
$$d(u,v) := \|u-v\|_{L^1(\Omega)} + \big| | Eu|(\Omega) - | Ev |(\Omega)\big| \quad \text{for } u, v \in BD(\Omega).$$
Hence, a sequence $\{u_n\} \subset BD(\Omega)$ converges to a function $u \in BD(\Omega)$ with respect to this topology if and only if $u_n \to u$ in $L^1(\Omega;\mathbb R^N)$, $ Eu_n \wsto  Eu$ in the sense of measures, and
$| Eu_n|(\Omega) \to | Eu|(\Omega)$. 

Recall that, if $u_n \to u$ in $L^1(\Omega;\mathbb R^N)$ and there exists $C > 0$ such that $|Eu_n|(\Omega) \leq C$ for all  $n \in \mathbb N$, then $u \in BD(\Omega)$ and 
\begin{equation*}
| Eu|(\Omega) \leq \liminf_{n\to +\infty}| Eu_n|(\Omega).
\end{equation*}
The following result  (see    \cite{T}, \cite[Theorem 2.6]{BFT}, and {\cite[Lemma 4.15]{BDG}} for the area-strict version) shows that it is possible to approximate any $u \in BD(\Omega)$  by a sequence of smooth functions.

\begin{proposition} 
\label{prop:densityuptoboundary}
     Let $u\in BD(\Omega)$. Then, there exists a sequence $\lbrace u_j\rbrace\subset C^\infty(\overline{\Omega};\mathbb R^N )$ such that $u_j$ converges to $u$ in $L^1(\Omega;\mathbb R^N)$  and
$\int_\Omega|\mathcal E u_n| \,  {\rm d}x  \to |E u|(\Omega). $ 
\end{proposition}
\color{black}

\color{black}


By the Lebesgue Decomposition Theorem, $Eu$ can be split into the sum of two
mutually singular measures, namely an absolutely continuous
part $\mathcal E u$ and a singular part $E^s u$ with respect to the Lebesgue measure $\mathcal{L}^N$.   More precisely,  if $u \in BD(\Omega )$, then $J_u$ is countably $(N-1)$-rectifiable, see \cite{ACDM}, 
and the following decomposition holds 
\begin{equation*}
Eu= \mathcal E u \mathcal{L}^N \lfloor \Omega + [u] \odot \nu_u {\mathcal{H}}^{N-1}\lfloor J_u +  E^cu,
\end{equation*}
\noindent where $[u] = u^+ - u^-$, $u^{\pm}$ are the traces of $u$ on the sides of $J_u$ determined by the unit normal $\nu_u$ to $J_u$, and $\mathcal{E}^cu$ is the Cantor part of the measure $Eu$ which vanishes on  Borel sets $B$ with 
$\mathcal H^{N-1}(B) < + \infty.$  If $E^cu \BBB \equiv 0 \EEE$, we say $u \in SBD(\Omega)$. The following result holds (cf. \cite[Theorem 4.3]{ACDM} and \cite[Theorem 2.5]{E2}).  

\begin{theorem}
(Approximate Symmetric Differentiability)\label{approxsymdiff} 
If $u \in BD(\Omega),$ then for $\mathcal{L}^N$-a.e.\ $x \in\Omega$,
there exists an $N\times N$ matrix $\nabla u(x)$ such that
\begin{equation}\label{apdifBD}
\lim_{\varepsilon \rightarrow 0^+} \frac{1}{\varepsilon^{N+1} }
\int_{B_\varepsilon(x)} |u(y) - u(x) - \nabla u(x)(y-x)|\, {\rm d}y  = 0,
\end{equation}
\begin{equation}\label{apsymdif}
\lim_{\varepsilon \rightarrow 0^+} \frac{1}{\varepsilon^{N} }
\int_{B_\varepsilon(x)} \frac{|\langle u(y) - u(x) - \mathcal E u(x)(y-x), y-x\rangle|}{|y-x|^2}\, {\rm d}y  = 0,
\end{equation}
for $\mathcal L^N$- a.e. $x \in \Omega$.

\end{theorem}
From \eqref{apdifBD} and \eqref{apsymdif} it follows that
$  \mathcal E u=\frac{\nabla u+\nabla u^T}{2}$.  There is a linear and surjective trace operator ${\rm tr}   \colon BD(\Omega) \to L^1(\partial \Omega,\mathbb{R}^N)$ which satisfies the following continuity property proved in  \cite[Lemma 3.4]{Ba}.

\begin{proposition}\label{traceB}
Let $u \in BD(\Omega)$ and $\lbrace u_k\rbrace \subset BD(\Omega)$ be such that $u_k \to u$ strictly  in $BD(\Omega)$. Then,
 ${\rm tr} \,u_k  \to {\rm tr} \, u$ strongly in $L^1(\partial \Omega; \mathbb R ^N)$.  
\end{proposition}

We continue with a version of the approximation result in Theorem \ref{appTHMh} for $SBD$-functions.

\begin{theorem}\label{StSBDappthm}
Let  $(g,G) \in  StSBD^p(\Omega)$. Then, there exists $\lbrace u_n \rbrace  \subset  SBD(\Omega)$ such that $u_n \to g$ in $L^1(\Omega;\mathbb{R}^N)$, $\mathcal {E}u_n  =   G$, $E u_n \wsto Eg$, and $|E u_n|(\Omega)\leq C ( |E g|(\Omega)+ \|G\|_{L^1(\Omega)})$. 
\end{theorem}
 
\begin{proof}
The proof follows the same arguments as in the $SBV$ case, by applying Alberti's Theorem and a piecewise constant approximation of $L^1$ fields. Indeed, define $u_n= g + h - h_n$,
where by Alberti's Theorem $h \in SBV(\Omega; \mathbb{R}^{N\times N})$,  is such that
\begin{align}\label{817bis} 
\nabla h =  \mathcal{E}h = G - \mathcal{E}g, \quad \quad |Dh|(\Omega) \leq C \Vert G \Vert_{L^1(\Omega)}  + C \Vert \mathcal{E}g \Vert_{L^1(\Omega)}. 
\end{align}
Then, there exists a piecewise constant sequence $\lbrace h_n \rbrace$ with $h_n \to h$ in  $L^1(\Omega; \mathbb{R}^N)$  and $|Dh_n|(\Omega) \to |Dh|(\Omega).$  As  $\Omega$ is regular, we can extend $h$ outside $\Omega$ such that $|D h|(\partial \Omega)=0$ (see \cite[Theorem 2.16]{G}).  
Clearly, also $| Eh|(\partial \Omega) = 0.$
Then, since $|Dh_n|(\Omega)$ is bounded, we have that $| E h_n|(\Omega)$ is also bounded and therefore $ Eh_n \wsto  Eh.$
We conclude that
\begin{align}\label{Ehn*}| Eh|(\Omega) \leq \liminf_{n \to +\infty} | E h_n|(\Omega) \leq\limsup_{n\to +\infty} | Eh_n|(\Omega)\leq | Eh|(\Omega).
\end{align}
It is easily seen that $u_n \to g$ strongly in $L^1(\Omega;\mathbb R^{N})$ and $\mathcal Eu_n= \mathcal Eg+ \mathcal Eh= G$. Taking into account the definition of $u_n$, we have
$E u_n= E g+ Eh- E h_n$, thus the weak* convergence is immediately obtained. Concerning  the bound, we have
\begin{align*}
|E u_n|(\Omega)\leq |E g|(\Omega)+ |E h|(\Omega)+ |E h_n|(\Omega).
\end{align*}
 Up  to the choice of a suitable tail of $\{u_n\}$, this  gives the estimate, in view of \eqref{Ehn*} and \eqref{817bis}.
\end{proof}


   The orthogonal projection operator $P \colon BD(\Omega)\to \mathcal R$ \color{black} belongs to the class considered in the following Poincar\'e-Friedrichs
type inequality for $BD$ functions (see \cite{ACDM}, \cite{K}, and \cite{T}).

\begin{theorem}\label{Thm2.8BFT}
Let $\Omega$ be a bounded, connected, open subset of $\mathbb R^N$, with Lipschitz boundary, and let $R \colon BD(\Omega)\to \mathcal R$ be a continuous linear map which leaves the elements of $\mathcal R$ fixed. Then there exists a constant 
$C(\Omega, R)$ such that
\begin{align*}
\int_\Omega |u(x) - R(u)(x)| \, {\rm d}x \leq C(\Omega,R) \, | E u|(\Omega) \  \;  \mbox{ for every } u\in BD(\Omega).
\end{align*}
\end{theorem}



  \section{Auxiliary statements}\label{sec: aux-app}

\begin{lemma}\label{lemma: iso}
Let $\Omega \subset \mathbb{R}^N$ be open, bounded with Lipschitz boundary. Then, there is a constant $C_\Omega >0$ such that for each set of finite perimeter $E \subset \Omega$ it holds that
\begin{align}\label{brian-1}
\mathcal{H}^{N-1}(\partial^* E) \leq C_\Omega \big(\mathcal{H}^{N-1}(\partial^* E\cap \Omega) + (\mathcal{L}^N(E))^{\frac{N-1}{N}} \big).  
\end{align}
The constant $C_\Omega$ is invariant under rescaling of the domain.
\end{lemma}

A similar statement in dimension two can be found in \cite[Lemma A.2]{briani}.

\begin{proof}
By the  splitting $\partial^* E = (\partial^* E \cap \Omega) \cup (\partial^* E \cap \partial \Omega)$ we see that we need to show 
$$\mathcal{H}^{N-1}(\partial^* E \cap \partial \Omega) \leq  C_\Omega \big(\mathcal{H}^{N-1}(\partial^* E\cap \Omega) + (\mathcal{L}^N(E))^{\frac{N-1}{N}} \big).$$
As $\partial \Omega$ is Lipschitz, we can cover it by a finite number of cuboids $D_1,\ldots, D_M$ such that each $\partial \Omega \cap D_i$ can be represented as the graph of a Lipschitz function. Clearly, it suffices  to show
\begin{align}\label{brian0}
\mathcal{H}^{N-1}(\partial^* E \cap \partial \Omega \cap D_i)  \leq C \big(\mathcal{H}^{N-1}(\partial^* E\cap D_i) + (\mathcal{L}^N(E \cap D_i))^{\frac{N-1}{N}} \big) \quad \text{for all $i=1,\ldots,M$}.
\end{align}
Then, \eqref{brian-1} follows, and the invariance of the constant under rescaling is a consequence of standard rescaling arguments since both sides are positively $(N-1)$-homogeneous.

We fix any $D_i$ and write $D$ instead of $D_i$ for simplicity. After a transformation it is not restrictive to assume that $D = (-l,l)^{N-1} \times (-h,h)$ for $l,h >0$, and that
$$\partial \Omega \cap D = \lbrace  (z,\varphi(z)) \colon   z \in (-l,l)^{N-1} \rbrace, \quad   \Omega \cap D = \lbrace  (z,t) \colon   z \in (-l,l)^{N-1}, \, t < \varphi(z) \rbrace, $$
where $\varphi \colon (-l,l)^{N-1} \to (-\frac{h}{2},\frac{h}{2})$ is a Lipschitz function. We let $\omega_E = \lbrace z \in (-l,l)^{N-1} \colon    (z,\varphi(z)) \in \partial^* E \rbrace$ and let $\omega_E^{\rm full} \subset \omega_E$ be defined as the points such that the entire segment $\lbrace z \rbrace \times (-h,\varphi(z)) $ is contained in $E$, i.e., $\omega^{\rm full}_E = \lbrace z \in (-l,l)^{N-1} \colon    \mathcal{H}^1( \lbrace z \rbrace \times (-h,\varphi(z)) \setminus E) = 0\rbrace$. By the area formula we get
\begin{align}\label{brian1}
\mathcal{H}^{N-1}(\partial^* E \cap \partial \Omega \cap D) \leq   C \mathcal{H}^{N-1}(\omega_E),  
\end{align}
where $C>0$ only depends on the Lipschitz  constant of $\varphi$. We note that for $\mathcal{H}^{N-1}$-a.e.\  $z \in \omega_E \setminus \omega_E^{\rm full}$, the segment $\lbrace z \rbrace \times (-h,\varphi(z))$ intersects $\partial^*E \cap D$, and thus, by the area formula,
\begin{align}\label{brian2}
\mathcal{H}^{N-1}(\omega_E \setminus \omega_E^{\rm full} ) \leq \mathcal{H}^{N-1}(\partial^*E \cap D).
\end{align}
Moreover, by Fubini's theorem and the fact that $D = (-l,l)^{N-1} \times (-h,h)$,  $\varphi \geq - \frac{h}{2}$, we get 
\begin{align}\label{brian3}
\mathcal{H}^{N-1}( \omega_E^{\rm full} ) \leq \frac{2}{h}\mathcal{L}^N(E \cap D). 
\end{align}
Combining \eqref{brian1}--\eqref{brian3} we get
$$\mathcal{H}^{N-1}(\partial^* E \cap \partial \Omega \cap D) \leq C\big(\mathcal{H}^{N-1}(\partial^*E \cap D) + \tfrac{2}{h}\mathcal{L}^N(E \cap D) \big). $$
Eventually, it holds   $h \geq \frac{1}{C_\Omega}\mathcal{L}^N(\Omega)^{1/N}$ for a constant $C_\Omega >0$ depending on $\Omega$. Thus, $\frac{1}{h} \leq C_\Omega (\mathcal{L}^N(\Omega))^{-1/N} \leq C_\Omega (\mathcal{L}^N(E \cap D))^{-1/N}$. This shows \eqref{brian0} and concludes the proof. 
\end{proof}

\begin{proposition}[Korn-Poincar\'e inequality]\label{prop: KP}
Let $\Omega \subset \mathbb{R}^N$ be open, bounded,  connected with Lipschitz boundary. There exists a constant $C>0$ only depending on $\Omega$ such that for each $u \in BV(\Omega;\mathbb{R}^N)$ it holds that
$$   \Vert u - \bar{u} \Vert_{L^1(\Omega)} \leq C\big(|Eu|(\Omega) +   |Du(\Omega)|  \big),  $$
where $\bar{u} := \fint_\Omega u(x) \, {\rm d}x$.
\end{proposition}

Note that the main point is that on the right-hand side it does \emph{not} appear the total variation of $Du$ but just  the  Frobenius norm  of $Du(\Omega) \in \mathbb{R}^{N \times N}$.

\begin{proof}
We assume by contradiction that the statement was false, i.e., for each $n \in \mathbb{N}$ there exists $u_n \in BV(\Omega; \mathbb{R}^N)$ such that
\begin{align}\label{kp1}
 \Vert u_n - \overline{u_n} \Vert_{L^1(\Omega)} \geq \frac{1}{n} \big(|Eu_n|(\Omega) +   |Du_n(\Omega)|\big),
 \end{align}
 where $\overline{u_n} := \fint_\Omega u_n(x) \, {\rm d}x$.
Up to rescaling (not relabeled), we can suppose that  $\Vert u_n - \overline{u_n} \Vert_{L^1(\Omega)} = 1$ for all $n \in \mathbb{N}$. As $\Vert u_n - \overline{u_n} \Vert_{BD(\Omega)}$ is bounded, by compactness in $BD$ (see \cite[Chapter 2, Theorem 2.4]{T}) we obtain $u \in BD(\Omega)$ with $Eu = 0$  such that $u_n - \overline{u_n} \to u$ in $L^1(\Omega;\mathbb{R}^N)$. In particular, $\Vert u \Vert_{L^1(\Omega)} = 1$, $\fint_\Omega u(x) \, {\rm d}x=0$,  and $u(x) = Ax + b$ for some $A = (a_{ij})_{ij}\in \mathbb{R}^{N \times N}_{\rm skew}$ and $b \in \mathbb{R}^N$.

As a special case of the divergence theorem in $BV$ (see e.g.\ \cite[Equation (3.85)]{AFP}) we get
$$    D_ju_n^{i}(\Omega) = \int_{\partial \Omega}   u_n^i \nu^j_\Omega \, {\rm d}\mathcal{H}^{N-1} \quad \text{for all $i,j=1,\ldots,N$}, $$
where on the right hand side $u_n^i$ denotes the $i$-th component of the trace of $u_n$ on $\partial \Omega$, and $\nu_\Omega^j$ the $j$-th component of the outer unit normal. By \eqref{kp1} we get $|Du_n(\Omega)| \to 0$ as $n \to \infty$. Thus, by the continuity of the trace in $BD$ with respect to strict convergence, see \cite[Proposition 3.4]{Ba}, we deduce
$$ 0 = \int_{\partial \Omega}    u^i \nu^j_\Omega \, {\rm d}\mathcal{H}^{N-1} \quad \text{for all $i,j=1,\ldots,N$}.  $$
As $u = A \cdot + b$ is affine, we can use the divergence theorem once more to find$$0 =    \int_\Omega  \partial_j  u^{i} \, {\rm d}x =   \mathcal{L}^N(\Omega) a_{ij} \quad \text{for all $i,j=1,\ldots,N$}, $$
i.e., $A = 0$. Therefore $u \equiv b$ is a constant function. This, however, contradicts $\Vert u \Vert_{L^1(\Omega)} = 1$ and $\fint_\Omega u(x) \, {\rm d}x=0$.   
\end{proof}

\subsection*{Acknowledgements}

  The authors thank D.R.~Owen for his insights and suggestions on the topic of this paper.

The work of JM was partially funded by Funda\c{c}\~{a}o para a Ci\^{e}ncia e Tecnologia (FCT), Portugal, through grant No. UID/4459/2025 and by the FCT-mobility program (outgoing) in 2025.

He also gratefully acknowledges the support and hospitality of Sapienza-University of Rome in May-June 2025.

EZ acknowledges the support of Piano Nazionale di Ripresa e Resilienza (PNRR) - Missione 4 ``Istruzione e Ricerca''
- Componente C2 Investimento 1.1, "Fondo per il Programma Nazionale di Ricerca e
Progetti di Rilevante Interesse Nazionale (PRIN)" - Decreto Direttoriale n. 104 del 2 febbraio 2022 - CUP 853D23009360006. She is a member of the Gruppo Nazionale per l'Analisi Matematica, la Probabilit\`a e le loro Applicazioni (GNAMPA) of the Istituto Nazionale di Alta Matematica ``F.~Severi'' (INdAM). 
The work of EZ is also supported by Sapienza - University of Rome through the projects Progetti di ricerca medi, (2024) and (2025), coordinator  M. Amar, and through the project https://doi.org/10.54499/UID/04674/2025 of the RCMI, Universidade de \'Evora.
She also gratefully acknowledges the hospitality and support of CAMGSD, IST-ID.

\bibliographystyle{plain}

\end{document}